\documentclass[11pt,a4paper]{article}
\usepackage{rotating}
\usepackage[utf8]{inputenc}
\usepackage[T1]{fontenc}
\usepackage[english]{babel}
\usepackage[a4paper,margin=1.08in]{geometry}
\usepackage{lmodern}
\usepackage{mathtools,amssymb,amsthm}
\usepackage{authblk}
\usepackage{graphicx}
\usepackage{caption}
\usepackage[dvipsnames]{xcolor}
\usepackage{csquotes}
\usepackage{comment}
\usepackage{tikz-cd}
\usepackage[
backend=biber,
style=alphabetic,
sorting=nyt
]{biblatex}
\usepackage{biblatex2bibitem}
\usepackage{hyperref}
\hypersetup{
    colorlinks = true,
    citecolor = black,
    urlcolor = blue,
    linkcolor = black,
    linktocpage=true,
    pdftitle={A kernel calculus for solutions to Stein equations and their derivatives},
    pdfauthor={Ferdinand Rapin and Yvik Swan},
    pdfsubject={Stein equations, kernel representations, and Stein factors},
    pdfkeywords={Stein's method, Stein equation, Stein factors, density approach, kernel representation, Edgeworth expansion, Curie--Weiss model}
}
\graphicspath{{./}}
\allowdisplaybreaks[2]
\newtheorem{theorem}{Theorem}[section]
\newtheorem{proposition}[theorem]{Proposition}
\newtheorem{lemma}[theorem]{Lemma}
\newtheorem{corollary}[theorem]{Corollary}
\theoremstyle{definition}
\newtheorem{definition}[theorem]{Definition}
\newtheorem{example}[theorem]{Example}
\newtheorem{notation}[theorem]{Notation}
\newtheorem{assumption}[theorem]{Assumption}
\theoremstyle{remark}
\newtheorem{remark}[theorem]{Remark}
\numberwithin{equation}{section}

\DeclareMathOperator{\sgn}{\text{sgn}}
\DeclareMathOperator{\E}{\mathbb{E}}

\title{Representations of solutions to Stein equations and their derivatives over their support}
\author[1]{Ferdinand Rapin}
\author[1,2]{Yvik Swan}
\affil[1]{Département de mathématique, Université Libre de Bruxelles}
\affil[2]{Vrije Universiteit Brussel}
\date{\today}

\begin{document}

\maketitle

\begin{abstract}
We give pointwise representations for the solution of a weighted
first-order Stein equation associated with a univariate absolutely
continuous distribution. These representations hold throughout the
positivity set of the density and extend to derivatives of the solution
of every order. For any prescribed derivative order of the test function,
we obtain an integral representation with explicit pointwise correction
terms when needed.

The construction relies on an array of kernels satisfying two identities.
The resulting ``kernel calculus'' gives, for each of the three
neighbouring orders $m\in\{n-1,n,n+1\}$, necessary 
conditions under which the $n$th derivative of the solution admits a
single-term representation in terms of the $m$th derivative of the test
function. It also explains why the Stein kernel plays a distinguished
role as a choice of weight. More generally, under suitable assumptions
on the support, regularity, integrability, and boundary behaviour, the
representations yield pointwise envelopes and uniform and weighted
Stein-factor bounds for arbitrary univariate absolutely continuous
distributions. We give verifiable conditions for these bounds to be sharp.

For a general distribution, we first work out the representations for
the solution and its first two derivatives. We then specialize to the
Kolmogorov Stein equation, for which we give representations at every
derivative order. We illustrate both sets of results for the symmetrized
Maxwell and Subbotin families. For the integrated Pearson family, the
representations take a particularly simple form and yield explicit sharp
envelopes at every admissible order. We obtain exact uniform Stein
factors in several regimes and explicit bounds in the remaining cases.
Some of these results rely on new Mills ratio identities for integrated
Pearson distributions, which are of independent interest. In each
example, our bounds recover or improve those available for the same
test-function classes.

We conclude the paper with applications to distributional approximation. We
sharpen the constants in several known bounds, including a Wasserstein
bound for the P\'olya--Eggenberger urn, and derive first-order
corrections for its beta approximation and for the quartic Subbotin
approximation in the critical Curie--Weiss model.

\

\noindent 
\underline{Automation and AI assistance.} We formulated the terminology and recursions of
our kernel calculus with automation in mind. With the assistance of
generative AI tools, we wrote a series of companion papers that derive
Stein envelopes and factors for target laws including the Gaussian, Gamma,
Student and Subbotin distributions. The bounds are organized by the metric
of interest, covering the Kolmogorov, total variation, Wasserstein,
Fortet--Mourier and integer-order Zolotarev distances. These papers provide
a common template for applying the calculus to other densities on the
real line. We have implemented this template in the
\href{http://ferdinand.rapin.web.ulb.be/stein}{Stein Atlas}, a web tool
that takes a user-supplied density, computes the corresponding envelopes
and factors, and generates a publication-ready paper detailing the bounds
and the conditions under which they apply.

\end{abstract}



\tableofcontents 

\section{Introduction}\label{sec:introduction}

Stein's method is an effective and popular tool for quantitative
distributional approximation. At its core, it replaces the problem of
directly comparing probability distributions with the analytic study
of a characterizing operator.

The main ideas are most easily illustrated in the standard Gaussian
setting. Let $N\sim\mathcal N(0,1)$ have density
\(\varphi(x)=({2\pi})^{-1/2}e^{-x^2/2}\).
The Gaussian integration-by-parts identity, also called Stein's
identity, states that
\begin{equation}
    \mathbb E[f'(N)]
    =
    \mathbb E[Nf(N)]
    \quad
    \mbox{for all } f\in\mathcal F_N,
    \label{eq:gaussian-stein-identity}
\end{equation}
where $\mathcal F_N$ is the class of locally absolutely continuous
functions $f:\mathbb R\to\mathbb R$ for which both expectations are
finite.  A sufficient condition for membership of $\mathcal F_N$ is that $f'$ be
Gaussian integrable; see \cite{nourdin_peccati_2012}.
If $f\in\mathcal F_N$, then
{both $f\varphi$ and $(f\varphi)'=(f'-xf)\varphi$ are
Lebesgue integrable on $\mathbb R$. It follows that $f\varphi$ has
finite limits at both ends of the real line and that both limits
are zero. Hence}
\[
 \int_{-\infty}^{\infty}(f(x)\varphi(x))'\,dx
 =
 \mathbb E[f'(N)-Nf(N)]
 =
 0.
\]
Stein's identity also characterizes the standard Gaussian
distribution{. Indeed, if a real random variable $X$ satisfies}
\[
 \mathbb E[f'(X)]=\mathbb E[Xf(X)]
 \quad {\text{for all }f\in C_0^\infty(\mathbb R),}
\]
{then $X\sim\mathcal N(0,1)$. Here $C_0^\infty(\mathbb R)$
denotes the class of smooth functions with compact support.}
These assertions are classical{, see, e.g.,
\cite{chen_goldstein_shao_2011,nourdin_peccati_2012}.}
The first-order differential operator
\begin{equation}
\label{eq:standgaussteineop}
    \mathcal Af(x)=f'(x)-xf(x)
\end{equation}
is called the standard Gaussian Stein operator.

Let $X$ be a real random variable whose law we wish to compare
with the standard Gaussian law, and let
$\mathcal G$ be a collection of {locally absolutely continuous
functions. For each $g\in\mathcal G$, fix a Borel representative of
$g'$ and assume that $g'(X)-Xg(X)$ is integrable.}
{The quantity}
\begin{equation}
\label{eq:gaussian-stein-discrepancy}
    \sup_{g\in\mathcal G}
    \left|
        \mathbb E[g'(X)-Xg(X)]
    \right|
\end{equation}
measures the failure of $X$ to satisfy the Gaussian Stein identity
{over the class $\mathcal G$}.
We call this a \emph{Stein discrepancy}.
For suitable choices of \(\mathcal G\), these discrepancies provide
useful analytic and computational measures of non-Gaussianity{,
see \cite{anastasiou2021stein,liu_mackey_oates_2026}.}
To relate them to standard probabilistic distances, consider
integral probability metrics (IPMs).
Given two random variables \(X\) and \(Y\) and a
{class $\mathcal H$ of Borel functions integrable with
respect to both laws}, the corresponding IPM is
\begin{equation}
\label{eq:integral-probability-metric}
    d_{\mathcal H}(X,Y)
    :=
    \sup_{h\in\mathcal H}
    \left|
        \mathbb E[h(X)]-\mathbb E[h(Y)]
    \right|.
\end{equation}
{Taking $\mathcal H$ to be the class of indicators of
half-lines $(-\infty,z]$ gives the Kolmogorov distance. Taking all
Borel functions with supremum norm at most one gives twice the total
variation distance under the usual convention. For laws with finite
first moments, taking all $1$-Lipschitz functions gives the
Wasserstein-1, or Kantorovich--Rubinstein, distance.}
Other choices give the bounded-Lipschitz or Fortet--Mourier distance,
metrics based on smoothed indicators, and Zolotarev ideal metrics
defined through bounds on higher-order derivatives.

Stein's method connects the IPMs in
\eqref{eq:integral-probability-metric} to the Stein discrepancies
in \eqref{eq:gaussian-stein-discrepancy} through the
\emph{Stein equation}.
For a pointwise-defined Borel test function
\(h\in L^1(N)\), {write $\Phi(h):=\mathbb E[h(N)]$.}
The (Gaussian) \(h\)-Stein equation is the ordinary differential
equation
\begin{equation}
\label{eq:gaussian-stein-eq}
    f'(x)-xf(x)=h(x)-\Phi(h),
    \qquad \textrm{a.e. }x\in\mathbb R.
\end{equation}
A solution is a locally absolutely continuous
function {$f$ satisfying \eqref{eq:gaussian-stein-eq}
almost everywhere.}
Since \(f'-xf=(f\varphi)'/\varphi\), {integration gives
the unique solution} satisfying
$\lim_{x\to\pm\infty}f(x)\varphi(x)=0${, namely}
\begin{equation}
\label{eq:gaussian-stein-solution}
    f_{\varphi,h}(x)
    :=
    \frac{1}{\varphi(x)}
    \int_{-\infty}^x
    \bigl(h(t)-\Phi(h)\bigr)\varphi(t)\,dt.
\end{equation}
As we shall see, for the test classes considered below, where $h$
or one of its derivatives is bounded, this solution belongs to
$\mathcal F_N$.

When evaluating the Stein equation at an arbitrary
random variable $X$, we use the representative
{
\[
  f_{\varphi,h}'(x):=xf_{\varphi,h}(x)+h(x)-\Phi(h),
  \qquad x\in\mathbb R.
\]
Equation~\eqref{eq:gaussian-stein-eq} ensures that this choice agrees
with the classical derivative almost everywhere. Its advantage is that
the Stein equation now holds at every $x$, including the atoms of $X$.
Thus, whenever $h(X)$ is integrable, so is
$f_{\varphi,h}'(X)-Xf_{\varphi,h}(X)$, and}
\[
 \mathbb E[
 f_{\varphi,h}'(X)-Xf_{\varphi,h}(X)
 ]
 =
 \mathbb E[h(X)]-\Phi(h).
\]
{Taking the supremum over $h\in\mathcal H$ gives}
\begin{equation}
\label{eq:stein-transfer}
    d_{\mathcal H}(X,N)
    =
    \sup_{h\in\mathcal H}
    \left|
        \mathbb E\left[
            f_{\varphi,h}'(X)-Xf_{\varphi,h}(X)
        \right]
    \right|.
\end{equation}
We call this Stein's \emph{transfer principle}{. It
transfers the problem of bounding $d_{\mathcal H}(X,N)$ to that of
bounding a Stein discrepancy over the solutions
$f_{\varphi,h}$, $h\in\mathcal H$, with the derivative representative
specified for each $h$.}

The usefulness of \eqref{eq:stein-transfer} depends on analytic
control of the Stein solution \(f_{\varphi,h}\), both over
\(x\in\mathbb R\) and in terms of \(h\in\mathcal H\).
One needs existence, regularity and sharp bounds on
\(f_{\varphi,h}\) and its derivatives; in particular, pointwise
and uniform bounds over all \(h\in\mathcal H\) are crucial.

\begin{definition}
\label{def:stein-factors-and-envelopes}
Bounds on the solutions \eqref{eq:gaussian-stein-solution}
and their derivatives that are uniform in $h$ and pointwise
in $x$ are called \emph{Stein envelopes}; bounds that are
uniform in both $h$ and $x$ are called \emph{Stein factors}.
An envelope (resp.\ a factor) is \emph{sharp} when it equals
the corresponding supremum over test functions
(resp.\ over test functions and the essential
supremum over $x$), whether or not this supremum is attained.
\end{definition}

\begin{remark}
The term \emph{Stein factors} is standard.
We use the term \emph{Stein envelopes} for their pointwise
counterparts, often called non-uniform Stein factors in the
literature.
\end{remark}

Much is known about Gaussian Stein envelopes and factors.
We begin with the Kolmogorov class, for which the relevant
expressions are explicit. We collect in the following theorem
known representations and uniform Stein factors for the
Kolmogorov Stein solution and its first derivatives.

\begin{theorem}
\label{theo:kolmosteif}
Let \(\Phi(x)=\int_{-\infty}^x\varphi(u)\,du\) and
\(h_z=\mathbf1_{(-\infty,z]}\).
The canonical solution of the $h_z$-Stein equation
$f'(x)-xf(x)=h_z(x)-\Phi(z)$ is
\begin{equation}
\label{eq:explicikolmo}
 f_{\varphi,h_z}(x)
 =
 \frac{\Phi(z\wedge x)\bigl(1-\Phi(z\vee x)\bigr)}
      {\varphi(x)},
 \qquad x\in\mathbb R.
\end{equation}
It is differentiable on \(\mathbb R\setminus\{z\}\), and
for $x\neq z$,
\begin{equation}
\label{eq:kolmogorov-all-derivatives}
\begin{split}
 f_{\varphi,h_z}'(x)
 &=
 xf_{\varphi,h_z}(x)+h_z(x)-\Phi(z)\\
 &=
 \begin{cases}
 (1-\Phi(z))(x\Phi(x)+\varphi(x))/\varphi(x),
 &\text{if }x<z,\\
 \Phi(z)(x(1-\Phi(x))-\varphi(x))/\varphi(x),
 &\text{if }x>z.
 \end{cases}
\end{split}
\end{equation}
Its one-sided derivatives satisfy
\begin{equation}
\label{eq:kolmogorov-derivative-jumps}
 f_{\varphi,h_z}'(z+)-f_{\varphi,h_z}'(z-)=-1.
\end{equation}
The second derivative exists on $\mathbb R\setminus\{z\}$,
\begin{align*}
 f_{\varphi,h_z}''(x)
 &=
 (x^2+1)f_{\varphi,h_z}(x)+x(h_z(x)-\Phi(z))\\
 &=
 \begin{cases}
 (1-\Phi(z))
 \bigl((x^2+1)\Phi(x)/\varphi(x)+x\bigr),
 &\text{if }x<z,\\
 \Phi(z)
 \bigl((x^2+1)(1-\Phi(x))/\varphi(x)-x\bigr),
 &\text{if }x>z.
 \end{cases}
\end{align*}
Consequently, the Stein envelopes at the first three orders are
\begin{align}
 \sup_{z\in\mathbb R}|f_{\varphi,h_z}(x)|
 &=
 \frac{\Phi(x)(1-\Phi(x))}{\varphi(x)},
 \label{eq:kolmogorov-exact-envelope-f}\\
 \sup_{\substack{z\in\mathbb R\\z\neq x}}
 |f_{\varphi,h_z}'(x)|
 &=
 (1-\Phi(|x|))
 \left(|x|\frac{\Phi(|x|)}{\varphi(x)}+1\right),
 \label{eq:kolmogorov-exact-envelope-f'}\\
 \sup_{\substack{z\in\mathbb R\\z\neq x}}
 |f_{\varphi,h_z}''(x)|
 &=
 (1-\Phi(|x|))
 \left((x^2+1)\frac{\Phi(|x|)}{\varphi(x)}+|x|\right),
 \label{eq:kolmogorov-exact-envelope-f''}
\end{align}
and the corresponding Stein factors are
\begin{equation}
\label{eq:gaussian-kolmogorov-uniform-factors}
 \sup_{z\in\mathbb R}\|f_{\varphi,h_z}\|
 =
 \sqrt{\pi/8},
 \qquad
 \sup_{z\in\mathbb R}\|f_{\varphi,h_z}'\|
 =
 1,
 \qquad\text{and}\qquad
 \sup_{z\in\mathbb R}\|f_{\varphi,h_z}''\|
 =
 \infty.
\end{equation}
Here the derivative norms are the essential
suprema of the classical derivatives calculated separately
on $(-\infty,z)$ and $(z,\infty)$. In particular, the
second-derivative norm does not impose global regularity
across the threshold, unlike the convention in
Notation~\ref{not:derivative-norms} below.
\end{theorem}

\begin{remark}
The uniform bounds in
\eqref{eq:gaussian-kolmogorov-uniform-factors} are classical.
The function \eqref{eq:explicikolmo} has been studied
extensively since \cite{stein_approx_1986}; see
\cite[Section~2.2]{chen_goldstein_shao_2011} for an overview.
Our representations presented in
Section~\ref{sec:kolmogorov-solutions} will recover these
results, and extend them to higher derivatives and other
distributions.
\end{remark}

\begin{figure}
 \centering
 \href{https://yvik.swan.web.ulb.be/animationoxford.html}{%
   \includegraphics[width=\textwidth]{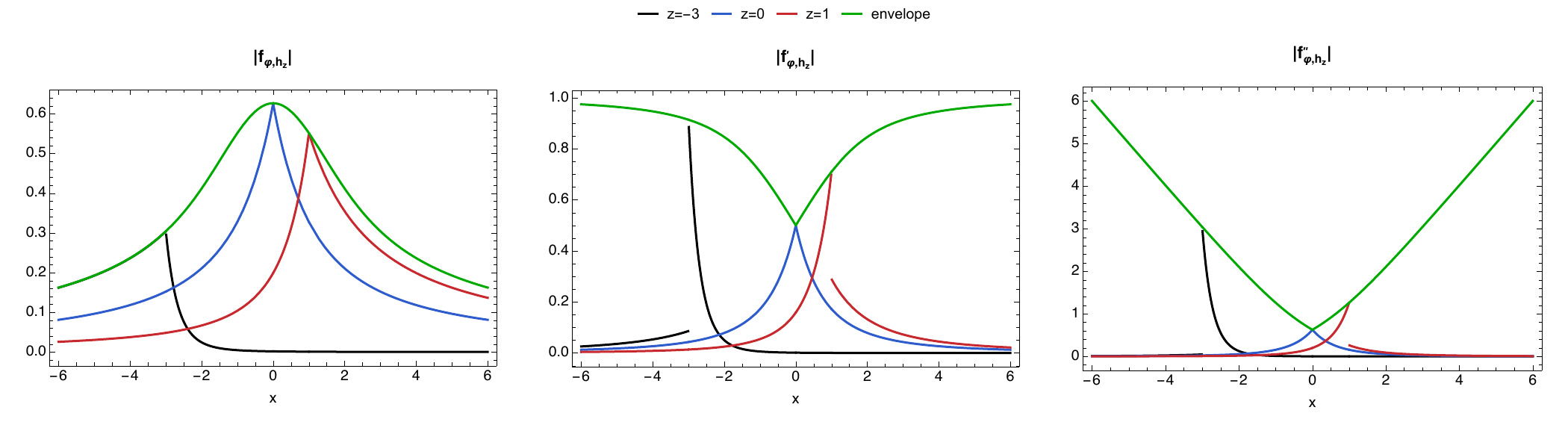}
 }
 \caption{Kolmogorov Stein solutions and their first two
 derivatives for \(z=-3,0,1\). The green curves are the envelopes
 \eqref{eq:kolmogorov-exact-envelope-f}--
 \eqref{eq:kolmogorov-exact-envelope-f''}.
 The envelopes of orders \(0\) and \(1\) are bounded,
 whereas the order-\(2\) envelope is unbounded, in agreement
 with \eqref{eq:gaussian-kolmogorov-uniform-factors}.}
 \label{fig:kolmogorov-factors}
\end{figure}

Figure~\ref{fig:kolmogorov-factors} illustrates these bounds
and, in particular, the strikingly different behaviour of
the first- and second-order envelopes.
Such bounds are used for approximation in the Kolmogorov
distance via Stein's method.
The unbounded envelope of order $2$ limits
direct Taylor arguments with indicators. Berry--Esseen
bounds often use additional smoothing or concentration
arguments, involving, for instance, smoothed indicator
functions in \eqref{eq:gaussian-stein-solution}; see
\cite{chen_shao_2001,chen_goldstein_shao_2011,nourdin_peccati_2012}.

Developing Stein's method in other IPMs, such as total variation,
Wasserstein and Zolotarev distances, requires
similar pointwise envelopes and uniform factors for the
corresponding Stein solutions \(f_{\varphi,h}\) and their
derivatives. For general smooth test classes, the resulting
bounds are naturally expressed in terms of the supremum
norms of $h$ and its derivatives.

\begin{notation}
\label{not:derivative-norms}
For a pointwise-defined function \(g:\mathbb R\to\mathbb R\),
set \(g^{(0)}=g\), and let \(g^{(i)}\) denote its \(i\)th
derivative wherever that derivative exists.
For \(i=0\), we set
\[
 \|g^{(0)}\|=\|g\|
 =
 \operatorname*{ess\,sup}_{x\in\mathbb R}|g(x)|.
\]
For \(i>0\), we set
\[
 \|g^{(i)}\|
 :=
 \operatorname*{ess\,sup}_{x\in\mathbb R}|g^{(i)}(x)|
\]
when $g\in C^{i-1}$ and \(g^{(i-1)}\) is locally absolutely
continuous, otherwise \(\|g^{(i)}\|=\infty\).
Here and below, derivatives of locally absolutely continuous
functions are understood almost everywhere.
We also write
\[
 \operatorname{osc}(g)
 :=
 \operatorname*{ess\,sup}g-\operatorname*{ess\,inf}g
 \leq2\|g\|.
\]
For functions defined on an open interval
or a union of open intervals, the same conventions apply
on that domain, with regularity understood on each component.

Throughout the paper, a representation or a bound of a
derivative of a Stein solution that involves \(h^{(k)}\)
is asserted for every test function \(h\) for which
the required derivatives exist in this sense,
that is \(h\in C^{k-1}\) with \(h^{(k-1)}\) locally absolutely
continuous (no regularity condition if \(k=0\)),
and whenever the seminorms on its right-hand side are finite.
Derivative identities and bounds are understood
almost everywhere unless a pointwise assertion is explicitly
made. Integral representations of the solution itself are
pointwise on the positivity set of the density.
\end{notation}

Considerable attention has been devoted to expressing
\(\|f_{\varphi,h}^{(n)}\|\) in terms of \(\|h^{(k)}\|\),
for \(n,k\in\mathbb N\).
The following theorem summarizes the envelopes and factors
known for the Gaussian Stein solution and its
derivatives.
Each bound is understood whenever the seminorm on its
right-hand side is finite.

\begin{theorem}
\label{thm:known-gaussian-stein-factors}
Let \(f_{\varphi,h}\) be the solution of the Gaussian Stein
equation given by \eqref{eq:gaussian-stein-solution}.
A representation of $f'$ is obtained directly from
\eqref{eq:gaussian-stein-eq} and, for $n\geq2$,
if $h\in C^{n-2}$ and $h^{(n-2)}$ is locally
absolutely continuous,
then almost everywhere
\begin{equation}
\label{eq:gaussian-daly-rep-fn}
\begin{split}
 f_{\varphi,h}^{(n)}(x)
 &=
 h^{(n-1)}(x)\\
 &\quad+
 \frac{(-1)^{n-1}}{(n-2)!}
 \left[
 \left(\frac{1-\Phi}{\varphi}\right)^{(n)}(x)I_{n,1}(x)
 +
 \left(\frac{\Phi}{\varphi}\right)^{(n)}(x)I_{n,2}(x)
 \right],
\end{split}
\end{equation}
with
\begin{align*}
 I_{n,1}(x)
 &:=
 \int_{-\infty}^x
 h^{(n-1)}(t)\varphi(t)
 \left(\frac{\Phi}{\varphi}\right)^{(n-2)}(t)\,dt,\\
 I_{n,2}(x)
 &:=
 \int_x^\infty
 h^{(n-1)}(t)\varphi(t)
 \left(\frac{1-\Phi}{\varphi}\right)^{(n-2)}(t)\,dt.
\end{align*}
The solution and its first derivatives admit the following
envelopes:
\begin{align}
 |f_{\varphi,h}(x)|
 &\leq
 \min\left\{
 2\frac{\Phi(x)(1-\Phi(x))}{\varphi(x)}
 \|h^{(0)}\|,
 \|h^{(1)}\|
 \right\},
 \label{eq:gaussian-envelope-f}\\
 |f_{\varphi,h}'(x)|
 &\leq
 \min\Bigg\{
 2\|h^{(0)}\|,
 2\frac{
 \left(\displaystyle\int_{-\infty}^x\Phi(u)\,du\right)
 \left(\displaystyle\int_x^\infty(1-\Phi(u))\,du\right)
 }{\varphi(x)}
 \|h^{(1)}\|,
 \frac12\|h^{(2)}\|
 \Bigg\}.
 \label{eq:gaussian-envelope-f-prime}
\end{align}
The corresponding uniform Stein factors are
\begin{align}
 \|f_{\varphi,h}\|
 &\leq
 \min\left\{
 2\sqrt{\frac{\pi}{8}}\,\|h^{(0)}\|,
 \|h^{(1)}\|
 \right\},
 \label{eq:gaussian-f-factors}\\
 \|f_{\varphi,h}'\|
 &\leq
 \min\left\{
 2\|h^{(0)}\|,
 \sqrt{\frac2\pi}\,\|h^{(1)}\|,
 \frac12\|h^{(2)}\|
 \right\},
 \label{eq:gaussian-f-prime-factors}
\end{align}
and, for every \(n\geq1\),
\begin{equation}
\label{eq:gaussian-general-magic-factors}
 \|f_{\varphi,h}^{(n)}\|
 \leq
 \min\left\{
 c^{n,n-1}_\varphi\|h^{(n-1)}\|,
 c^{n,n}_\varphi\|h^{(n)}\|,
 c^{n,n+1}_\varphi\|h^{(n+1)}\|
 \right\},
\end{equation}
where
\[
 c^{n,n-1}_\varphi=2,
 \qquad
 c^{n,n}_\varphi
 =
 \frac{\Gamma((n+1)/2)}
      {\sqrt2\,\Gamma(n/2+1)},
 \qquad
 c^{n,n+1}_\varphi=\frac1{n+1}.
\]
Each constant in \eqref{eq:gaussian-general-magic-factors}
is optimal for the corresponding seminorm.
Moreover, if
\(k\in\mathbb N\setminus\{n-1,n,n+1\}\),
then there is no finite constant \(c^{n,k}_\varphi\),
depending only on \(n\) and \(k\), such that
\[
 \|f_{\varphi,h}^{(n)}\|
 \leq c^{n,k}_\varphi\|h^{(k)}\|
\]
holds uniformly over all test functions for which the
right-hand side is finite.
\end{theorem}

\begin{proof}[Proof by bibliography]
The representation \eqref{eq:gaussian-daly-rep-fn} was
obtained by Daly in \cite{daly_2008},
with the almost-everywhere formulation above
following by approximation.
The pointwise estimates \eqref{eq:gaussian-envelope-f} and
\eqref{eq:gaussian-envelope-f-prime} are similarly well-known;
they follow e.g.\ from covariance representations; see,
for example, \cite[Example~2.31]{ernst_distances_2021}.
Taking suprema in the corresponding envelopes yields the
uniform bounds in \eqref{eq:gaussian-f-factors} and
\eqref{eq:gaussian-f-prime-factors}.
These classical estimates can also be obtained directly
from the explicit solution of the Gaussian Stein equation;
see, for instance,
\cite[Section~2.2]{chen_goldstein_shao_2011}.
The bound \(\|f_{\varphi,h}\|\leq\|h^{(1)}\|\) is due to
D\"obler \cite{dobler2012steinsmethodexchangeablepairs}.

We now turn to the Stein factors in
\eqref{eq:gaussian-general-magic-factors}.
The estimate involving \(\|h^{(n+1)}\|\) follows from the
Ornstein--Uhlenbeck semigroup representation of the solution
to the associated Poisson equation.
This approach goes back to the generator method of
\cite{barbour_1988,gotze_1991}.
The corresponding higher-order estimate was obtained,
under stronger regularity assumptions, by
\cite{Goldstein_Rinott}, with the general form following
by approximation.
The estimate involving \(\|h^{(n)}\|\) is obtained from
the same semigroup representation by transferring one
derivative from the test function to the Gaussian density
through Gaussian integration by parts; see \cite{gaunt_2016}.
Finally, the estimate involving \(\|h^{(n-1)}\|\) was proved
by \cite{daly_2008} using
the representation
\eqref{eq:gaussian-daly-rep-fn} and stochastic comparison
arguments.

Finally, we comment on the sharpness.
The sharpness of the constants \(2\) and \(1/(n+1)\) is
proved e.g.\ in \cite[p.~568]{daly_2008}.
\cite{gaunt2024steinfactorssteinsmethod} proved the optimality
of the intermediate constant and completed the picture
by showing that no uniform estimate in terms of
\(\|h^{(k)}\|\) is possible when
\(k\notin\{n-1,n,n+1\}\).
\end{proof}

\begin{remark}\label{rk:norms-gaussian} Theorem~\ref{thm:known-gaussian-stein-factors} is stated,
as in the literature, with $\|h^{(0)}\|$. Since $f_{\varphi,h}$ does not change when a constant is
added to $h$, there is a priori no hierarchy between these three norms and the corresponding Stein factors. 
For instance
\[
 |f_{\varphi,h}(x)|
 \leq
 \frac{\Phi(x)(1-\Phi(x))}{\varphi(x)}
 \operatorname{osc}(h),
 \qquad
 \|f_{\varphi,h}'\|\leq\operatorname{osc}(h).
\]
For indicators, $\operatorname{osc}(h_z)=1$, and these
bounds coincide with the Kolmogorov factors of
Theorem~\ref{theo:kolmosteif}: the factor $2$ between the
two theorems only comes from the seminorm, the same
convention under which $\{\|h\|\leq1\}$ yields twice the
total variation distance.
From Section~\ref{sec:density-approach} on, bounded test
functions are measured with $\operatorname{osc}(h)$.
\end{remark}

\begin{remark}
Lef\`evre and Utev \cite{lefevre_utev_2003} also determined
the exact diagonal operator norm with respect to the total
variation seminorm of the test function.
It may be interesting to determine whether analogous exact
bounds can be obtained in terms of total variation seminorms
of derivatives of orders different from \(n\).
\end{remark}

One of the principal strengths of Stein's method is its
adaptability to different target distributions.
Beyond the Gaussian case, the Gamma distribution is a major
example for which a substantial theory of Stein solutions
has been developed; see
\cite{luk_1994,gaunt_2013,gaunt_pickett_reinert_2017}.
Overviews are given in
\cite{bailly_gaunt_ouimet_richards_vonsachs_2026,bailly_swan2026}.
The resulting bounds have much the same general flavor as
their Gaussian counterparts: they control the solution and
its derivatives in terms of regularity seminorms of the test
function. An important difference is that the coefficient
of the derivative in the Gamma Stein operator vanishes at
the left endpoint of the support.
Consequently, the natural estimates often involve the
multiplicative weight \(x\).
Indicator, smooth, weighted, and higher-order bounds have
been obtained through a variety of arguments adapted to
the test class and to the behavior of the solution near
the boundary.

A similarly substantial collection of Stein factors is
available for the Beta distribution; see
\cite{goldstein_reinert_beta_2013,dobler_beta_2015,
dobler_gaunt_vollmer_2017}.
Here again, the estimates resemble those of the Gaussian
and Gamma settings, but the bounded support and the
degeneracy of the Stein operator at both endpoints
introduce additional boundary-dependent features.

In the same spirit, bounds have been obtained for
quartic-exponential and related Subbotin-type distributions
arising in statistical mechanics
\cite{Chatterjee_2011,eichelsbacher_lowe_2010}
and, more recently, \cite{eichelsbacher_2026};
for Pearson and integrated Pearson distributions
\cite{schoutens_2001,germain_swan_2025};
and for the generalized inverse Gaussian and Kummer
distributions \cite{konzou_koudou_2020}.
More general density-based approaches to univariate Stein
operators and their solutions were developed in
\cite{ley_swan_2013,dobler_beta_2015,ley_reinert_swan_2017,
ernst_reinert_swan_2020}; universal first- and
second-derivative bounds for general targets were also
obtained in \cite{eden_viquez_2015}.

These results concern diverse distributions but share a
common setting: the target is univariate, has a positive
and sufficiently regular density \(p\) on the interior
of an interval, and admits a tractable first-order
Stein operator
\begin{equation}
\label{eq:weighted-density-stein-operator-intro}
 \mathcal A_{p,w}f(x)
 =
 w(x)f'(x)+s_w(x)f(x),
 \qquad
 s_w(x)=\frac{(w(x)p(x))'}{p(x)}.
\end{equation}
More precisely, there is a sufficiently rich class
\(\mathcal F_{p,w}\) such that
\[
 \mathbb E[\mathcal A_{p,w}f(Z)]=0,
 \qquad f\in\mathcal F_{p,w},
\]
whenever \(Z\sim p\)
and $\mathcal A_{p,w}f(Z)$ is integrable.
Every sufficiently regular density admits operators of
this form; the substantive issue is to choose the weight
function \(w\) so that the coefficients, boundary conditions,
canonical solution, and resulting Stein factors remain
tractable.

The weighted density Stein operator
\eqref{eq:weighted-density-stein-operator-intro} is naturally
related to a one-dimensional diffusion generator.
Indeed, for positive $w$,
\[
 \mathcal L_{p,w}g
 :=
 \mathcal A_{p,w}(g')
 =
 wg''+s_wg'
\]
is, under suitable assumptions, the infinitesimal generator
of a diffusion with invariant density \(p\).
This observation is the starting point of the generator
approach to Stein's method
\cite{barbour_1988,kusuoka_tudor_2012,dobler_gaunt_vollmer_2017}.
Classical examples are the Hermite, Laguerre, and Jacobi
generators associated with the Gaussian, Gamma, and Beta laws.
Semigroup representations of the corresponding Poisson
equation, combined with derivative intertwinings or gradient
estimates, yield smooth bounds relating derivatives of the
Stein solution to derivatives of the test function
\cite{dobler_gaunt_vollmer_2017}.
For a modern overview of diffusion Stein operators and the
generator approach, see
\cite{gorham_duncan_vollmer_mackey_2019,anastasiou2021stein}
and the references therein.

The generator approach explains part, but not all, of the
common structure. Even in the classical examples, it does
not by itself recover the complete family of
neighbouring-order estimates involving \(h^{(n-1)}\),
\(h^{(n)}\), and \(h^{(n+1)}\), nor the pointwise envelopes
from which the corresponding uniform factors arise.
Such bounds have generally been obtained through
distribution-specific arguments adapted to the density,
support, and boundary behaviour, leaving their common
first-order mechanism largely implicit.

\paragraph{Contributions.}
The principal contribution of this paper consists of
pointwise representations of the Stein solution and its
derivatives, originating from a single general mechanism.
This mechanism also produces Stein envelopes and
(weighted) factors for derivatives of all
orders, under explicit regularity,
integrability, and boundary assumptions. 

The representations rely on a two-index array of kernels
\(K_{p,w}^{i,j}\), for an arbitrary absolutely continuous
target with density $p$ satisfying our
standing assumptions and admissible non-vanishing
weight \(w\).
The first index records differentiation of the Stein
solution and the second records transfer of derivatives
to the test function.
Lemma~\ref{lma:Kernel calculus -- integration by parts} and Lemma~\ref{lma:Kernel calculus -- differentiation} show that this array
is generated by two operations:  integration by parts in the test
variable and differentiation in the
spatial variable. Crucially, both operations identify explicitly
the pointwise correction terms that are usually hidden
in distribution-specific calculations. Figure \ref{fig:kernel-array-calculus} illustrates that mechanism. 

{Theorem~\ref{thm:kernel-representations-of-derivatives}
gives three corrected representations of
\(f_{p,w,h}^{(n)}\), involving the neighbouring
test-function orders \(h^{(n-1)}\), \(h^{(n)}\), and
\(h^{(n+1)}\), respectively.
Lemma~\ref{lem:no-other-orders} proves,
under its respective moment and regularity
assumptions, that these are the only orders at which
finite single-seminorm factors can hold.
Corollary~\ref{cor:closed-three-neighbour-calculus}
provides explicit conditions on $p$ and $w$ under which
the pointwise corrections in these
representations cancel. Moreover when the kernel signs can be realised or
approximated by admissible test derivatives, the resulting
envelopes are sharp. Sharp uniform factors then follow
by taking their suprema over the support; when these
suprema cannot be evaluated explicitly, the envelopes
still provide computable bounds.
}

{At low orders, this gives universal representations of
the canonical solution \(f_{p,w,h}\) and its first two
derivatives without imposing a specific relation between
\(p\) and \(w\); see
Theorem~\ref{thm:universal-low-order-calculus}.
The correction coefficients reveal a distinguished
normalization. When \(p\) has finite mean and
\(w=\tau_p\) is its Stein covariance kernel (see Definition \ref{def:stein-kernel}),
the correction terms involving
$(\tau_p/w)'$ and $(\tau_p/w)''$ vanish.
This gives a direct explanation, at the level of the
Stein equation and its derivatives, for the effectiveness
of the Stein covariance kernel weight.
The representations directly yield Stein envelopes and factors as we illustrate for symmetrized Maxwell distributions (with weight $w=\tau_p$) in Example~\ref{subsec:symmetrized-maxwell} and the Subbotin
family (with weight $w=1$) in Example~\ref{subsec:subbotin-example}. The obtained factors improve known ones, and are proven to be sharp.}

{
When the test function is chosen as a lower half-line indicator, the kernel representations provide tractable representation of all the derivatives of the solution, away from the threshold. We explicit the envelope for the orders 0, 1 and 2 and illustrate the results in  Figures~\ref{fig:ip-kolmogorov-gamma}--\ref{fig:ip-kolmogorov-student}. We then perform the computations for the symmetrized Maxwell distributions (Example~\ref{ex:symmax-kolmo}), for which it appears that the computations have never been recorded, and the Subbotin
family (Example~\ref{ex:subbotin-kolmo}). The obtained factors recover available ones, and are proven to be sharp.
}

We study the integrated Pearson family with the weight $w=\tau_p$ for which the simplification conditions from Corollary~\ref{cor:closed-three-neighbour-calculus} are proven to hold under simple moment conditions. This result is obtained through a new formula for the iterated tail of an integrated Pearson density. In particular they hold at all orders for the Gaussian,
the Gamma and the Beta distributions
with $w=\tau_p$. The resulting envelopes follow from new results about the iterated tails and the derivative of the Mills ratio of integrated Pearson distribution, that are of independent interest. Simple closed forms are obtained for the lower-
and upper-neighbour factors.
For the diagonal factors, we identify the supremum for
the whole family in the boundary regime and for symmetric
targets, and give an explicit upper bound otherwise,
under the conditions of
Theorem~\ref{thm:IP-diagonal}. Table~\ref{tab:ip-factors} records the consequent Stein factors for the Gaussian, Exponential, Gamma, Beta and Student distributions. For the Gaussian target, we recover all
the statements of
Theorem~\ref{thm:known-gaussian-stein-factors}.
We compute numerical results for the Beta distribution with parameters $\alpha =2, \beta = 5$ (Example~\ref{ex:beta}) and discuss in detail the Exponential, where one can choose between two weights
(Example~\ref{ex:expon}). We finally prove general conditions to derive weighted envelopes and factors in Proposition~\ref{prop:all-order-weighted-factor}. 

In summary, the framework both unifies a collection of
previously distribution-specific arguments and yields
bounds that were not previously available in this
general form.

The sharper Stein factors open the door
to revisiting quantitative results
obtained through Stein's method to improve the constants,
thus reducing the sufficient number of
observations or steps needed to obtain
a desired precision in a given metric. This is illustrated for the Beta approximation of the last return to zero
of a simple symmetric random walk and the quartic Subbotin approximation of the critical Curie--Weiss model.
As another application, the authors have already used
the Stein factors {for the Gamma distribution} in the analysis of Satterthwaite's
approximation, capitalizing on the explicit dependence
on the parameters of the target Gamma distribution
\cite{bailly_swan2026}.
In the present article, we present in addition, a first-order Edgeworth correction for the P\'olya--Eggenberger
Beta approximation with an \(O(n^{-2})\) remainder for test functions
with Lipschitz second derivative,
and a first-order correction for the
critical Curie--Weiss magnetization with an
\(O(n^{-3/4})\) remainder uniformly over
$1$-Lipschitz test functions, complementing the Cram\'er-type
expansion of the tail probabilities by Can and
Pham~\cite{can_pham_2017}.

A further contribution is the \href{https://ferdinand.rapin.web.ulb.be/stein/}{Stein Atlas}, which turns the kernel calculus of Section~\ref{sec:density-approach} into an algorithm. For a density given in closed form, possibly with parameters, the Atlas returns the envelopes and Stein factors of Section~\ref{sec:closed-kernel-calculi} for the weights $w = 1, w = \tau_p$ or a weight chosen by the user. Every value is labelled as a closed form proved in this paper, an exact supremum obtained symbolically, or a numerical supremum. All computations are deterministic and run in the browser (SymPy via Pyodide).

\begin{remark} 
The approach developed here can also be adapted to the
discrete setting, using the Stein operators and covariance
identities in
\cite{ernst_distances_2021,ernst_reinert_swan_2020};
see also our course notes for a Poisson example
\cite[Chapter~2]{swan_stein_notes}.
A systematic treatment lies beyond the present scope.

At around the time of our first arXiv
submission, Eichelsbacher~\cite{eichelsbacher_2026}
proposed Mills-ratio techniques for bounding Stein
solution envelopes and their suprema.
His discrete analysis builds on the birth--death
framework of \cite{brown_xia_2001} and applies, in
particular, to Poisson and negative binomial targets.
The underlying representations can also be obtained
from a discrete adaptation of our kernels, while his
techniques provide tools for bounding the resulting
expressions.
His approach also applies in some continuous settings,
particularly the Subbotin distributions.
We discuss this in more detail in
Example~\ref{subsec:subbotin-example}.
\end{remark}

\begin{remark} 
The explicit structure of our kernel representations
lends itself to symbolic computation and AI-assisted
exploration.
{ This makes the paper easy to use with
generative AI tools}: once the vocabulary, identities,
and recursions are established, many of the
distribution-specific computations readily lend
themselves to automation.
The accompanying \href{https://ferdinand.rapin.web.ulb.be/stein/}{Stein library},
provides an interactive presentation of envelopes and
factors, with adjustable parameters and derivative
orders, and downloadable PDF and \LaTeX{} summaries.
It also provides tools for exploring other densities
and choices of weight.

To illustrate the scope of the method, we have
written, with the assistance of generative
AI tools, companion papers for every distribution treated in this paper as well as the 
F, inverse Gamma, chi, Fréchet, Gumbel and half-normal distributions.
These PDFs are available on the \href{https://ferdinand.rapin.web.ulb.be/stein/}{Stein Atlas webpage}.
This collection will be expanded: the same procedure
can, in principle, generate a companion document for
each probability density on the real line satisfying
our assumptions, systematically recording the available
representations, envelopes, and factors.
The companion papers follow a common template,
organizing envelopes and factors (ordinary and weighted) by test class and derivative order.
They are intended both as usable references
for particular distributions and as demonstrations of
how the computations can be automated once the general
framework is in place.
We do not intend to submit them to arXiv.
Automation does not dispense with checking
the mathematical arguments and the assumptions under
which the resulting formulas apply.
\end{remark}

\begin{remark}
The present paper is restricted to first-order Stein
equations for absolutely continuous target distributions,
with the equations studied on the
positivity set of the target density.
Possible extensions include adapting the calculus to
second- and higher-order operators, such as those
associated with the Laplace, variance-gamma, and product
distributions
\cite{pike_ren_2014,gaunt_laplace_2021,
gaunt_variance_gamma_2014,gaunt_variance_gamma_factors_2022,
gaunt_products_2018,gaunt_mijoule_swan_2019},
and studying Stein solutions outside the support of
the target distribution, as initiated in
\cite{gammaStein_DoblerPeccati}.
\end{remark}

\subsection*{Outline of the paper}

Section~\ref{sec:density-approach} contains the general
theory. Section~\ref{sn:not} fixes the
notation and introduces the two Stein operators and
the weighted Stein equation.
Section~\ref{sec:kernel-representations} introduce the kernel array and its two basic
calculus identities. The kernel calculus yields the neighbouring representations of the $n$th derivatives of the Stein solution in Section~\ref{subsec:all-order-calculus}.

In Section~\ref{sec:closed-kernel-calculi}, we detail the theory in three conditions. We first consider low-orders derivatives in Section~\ref{subsec:universal-low-order} i.e. we explicit the kernel representations for the Stein solution $(n=0)$ and its first $n=1,2$ derivatives. Then in Section~\ref{sec:kolmogorov-solutions}, we restrict our attention to the family of lower half-line indicator test functions $h$, giving explicit solutions, derivatives,
and exact envelopes at every order under the corresponding
smoothness assumptions. Finally, we treat extensively the integrated Pearson family in Section~\ref{sec:integrated-pearson}.

Section~\ref{sec:applications} presents applications to
distributional approximation.
Section~\ref{sec:Improvement-known-distances} improves
constants in known Beta and critical Curie--Weiss
approximation bounds without changing the underlying
probabilistic constructions.
Section~\ref{subsec:edgeworth-corrections} then derives
Edgeworth corrections:a first-order correction for the P\'olya--Eggenberger Beta
approximation, and a first-order correction for the quartic
Subbotin approximation of the critical Curie--Weiss
magnetization, uniformly over $1$-Lipschitz test functions.

Proofs are delayed to the Appendix~\ref{app:proof-of-no-other-orders} for Section~\ref{sec:density-approach}, Appendix~\ref{sn:proofsectionsecin} for Section~\ref{sec:closed-kernel-calculi} and Appendix~\ref{app:proof-from-applications} for Section~\ref{sec:applications}.

\section{Representations, Stein envelopes and Stein factors}
\label{sec:density-approach}

This section presents our main results in an abstract setting. All proofs are deferred
to Appendix~\ref{app:proof-of-no-other-orders}. Examples are detailed in Section \ref{sec:closed-kernel-calculi}. 

\subsection{Notations}
\label{sn:not}

Let \(p\) be a probability density on \(\mathbb R\), and let \(Z\) be a random variable with density \(p\). 
We impose the following standing assumption on \(p\).

\begin{assumption}
\label{ass:density-support}
There exist \(m\geq1\) and
$-\infty\leq l=e_0<e_1<\cdots<e_{m-1}<e_m=u\leq\infty$
such that
$\{x\in\mathbb R:p(x)>0\}
 =
 \mathcal I_p
 :=
 \bigcup_{r=1}^m(e_{r-1},e_r)$.
The density \(p\) is locally absolutely continuous on each
\((e_{r-1},e_r)\).  
\end{assumption}

 We denote by $P$ and $\overline P$ the distribution and survival
functions associated with $p$, that is 
\[
    P(x):=\int_{-\infty}^x p(t)\,dt,
    \qquad
    \overline P(x):=\int_x^\infty p(t)\,dt=1-P(x),
    \qquad x\in\mathbb R.
\]
Let $w:\mathbb R\to\mathbb R$ be nonzero on $\mathcal I_p$. 
We  suppose from here onward that the weight $w$ is chosen in such a way  that $wp$ is locally absolutely continuous on each
$(e_{r-1},e_r)$, and set 
\[
    s_w(x):=\frac{(wp)'(x)}{p(x)},
    \qquad x\in\mathcal I_p
\]
where the derivative is understood in the sense of  Remark \ref{rk:abscontderiv} below.  Finally, we define 
\begin{equation}
\label{eq:weighted-mills-ratios}
 \underline m_{p,w}(x):=\frac{P(x)}{w(x)p(x)},
 \qquad
 \overline m_{p,w}(x):=\frac{\overline P(x)}{w(x)p(x)},
 \qquad x\in\mathcal I_p,
\end{equation}
which we call the   \emph{weighted lower and upper Mills ratios} of $p$.  

\begin{remark}
\label{rk:abscontderiv}
If a function $g$ is locally absolutely continuous
on an open interval, we denote by $g'$ a fixed measurable
representative of its almost-everywhere derivative, agreeing
with the classical derivative wherever the latter exists.
On any open subinterval where this derivative admits a continuous
version, we take that version; it then coincides everywhere
with the classical derivative of $g$. 
\end{remark}

 We denote by \(\mathcal F_{p,w}\) the class of functions $f:\mathcal I_p\to\mathbb R$ that are
locally absolutely continuous on each \((e_{r-1},e_r)\), for which all
the limits below exist and are finite, and which satisfy
\(\lim_{x\downarrow l} w(x)f(x)p(x)
 =
 \lim_{x\to u-} w(x)f(x)p(x)
 =
 0,\)
and
\(\lim_{x\uparrow e_r} w(x)f(x)p(x)
 =
 \lim_{x\downarrow e_r} w(x)f(x)p(x)\), \(1\leq r<m. \)

We now define the  differential and integral Stein operators,  and the corresponding Stein equation and solution. 
   First, let $f:\mathcal I_p\to\mathbb R$ be locally absolutely
continuous on each $(e_{r-1},e_r)$. With $f'$ denoting the
measurable representative fixed according to
Remark~\ref{rk:abscontderiv}, we define the differential
Stein operator
\begin{equation}
\label{eq:general-density-stein-operator}
    \mathcal A_{p,w}f(x)
    :=w(x)f'(x)+s_w(x)f(x),
    \qquad x\in\mathcal I_p.
\end{equation}
Local absolute continuity gives
\(    \mathcal A_{p,w}f(x)={(wfp)'(x)}/{p(x)}
\) for a.e.\ $x\in\mathcal I_p$.
For $h\in L^1(p)$, the weighted Stein equation is
\begin{equation}
\label{eq:w-stein-equation}
    \mathcal A_{p,w}f(x)
    =h(x)-\mathcal Ph,
    \qquad\text{a.e. }x\in\mathcal I_p
\end{equation}
where
\(
    \mathcal Ph:=\mathbb E[h(Z)]
    =\int_{-\infty}^{\infty}h(x)p(x)\,dx.
\) A solution to \eqref{eq:w-stein-equation} is a function
$f:\mathcal I_p\to\mathbb R$ that is locally absolutely
continuous on each $(e_{r-1},e_r)$ and satisfies this identity
for Lebesgue-almost every $x\in\mathcal I_p$.
 Finally, for $h\in L^1(p)$,  we define the integral Stein operator
\begin{equation}
\label{eq:caninv}
    \mathcal S_ph(x)
    :=\frac{1}{p(x)}
      \int_l^x \bigl(h(t)-\mathcal Ph\bigr)p(t)\,dt
    =\frac{1}{p(x)}
      \int_x^u \bigl(\mathcal Ph-h(t)\bigr)p(t)\,dt,
    \qquad x\in\mathcal I_p
\end{equation}
Clearly the function 
\begin{equation}
\label{eq:canonical-weighted-stein-solution}
    f_{p,w,h}(x):=\frac{\mathcal S_ph(x)}{w(x)},
    \qquad x\in\mathcal I_p,
\end{equation}
 solves \eqref{eq:w-stein-equation} everywhere on $\mathcal I_p$.  More precisely, standard arguments (see e.g.\  \cite{ernst_distances_2021}) yield the following.

\begin{proposition}
\label{prop:stein-solution}
Under the standing assumptions on $p$ and $w$, let $h\in L^1(p)$.
Then the function \eqref{eq:canonical-weighted-stein-solution} belongs to $\mathcal F_{p,w}$ and is the unique solution of 
\eqref{eq:w-stein-equation} in this class. 
\end{proposition}

\begin{remark}
    The construction above generalizes the one we already encountered in the Gaussian case. Indeed, taking $p = \varphi$ and  $w=1$: operator \eqref{eq:general-density-stein-operator} simplifies to \eqref{eq:standgaussteineop},   the integral operator \eqref{eq:caninv}  gives the Stein solution  \eqref{eq:gaussian-stein-solution}, and the Stein equation \eqref{eq:w-stein-equation} becomes  \eqref{eq:gaussian-stein-eq}.
\end{remark}

Equation \eqref{eq:canonical-weighted-stein-solution} gives a representation for the solution to the Stein equation \eqref{eq:w-stein-equation} which is defined everywhere. Since $\mathcal A_{p, w} f = w f' + s_w f$, a representative of the  first derivative of $f_{p, w, h}$ naturally comes from the Stein equation. The goal of this paper is obtain optimal sharp bounds on the derivatives of $f_{p, w, h}$ in terms of $p, w, h$. We will achieve this by iterating the above identities, hereby identifying sufficient conditions on $p, w$ and $h$ under which  $f_{p, w, h}$ admits higher order derivatives, and providing explicit representations thereof.

To do so, we will need the following functions. 
We set \(\underline P_0=\overline P_0=p\) and, for
\(j\geq1\), we recursively define
\[
 \underline P_j(x)
 :=
 \int_l^x\underline P_{j-1}(t)\,dt,
 \qquad
 \overline P_j(x)
 :=
 \int_x^u\overline P_{j-1}(t)\,dt,
\]
whenever these integrals are finite.  In particular, $\underline P_1=P$ and
$\overline P_1=\overline P$. 
For $j\geq2$, Tonelli's theorem gives
\[
 \underline P_j(x)
 =
 \frac{\mathbb E[(x-Z)_+^{j-1}]}{(j-1)!},
 \qquad
 \overline P_j(x)
 =
 \frac{\mathbb E[(Z-x)_+^{j-1}]}{(j-1)!}
\] 
for all densities with finite $j-1$th moment. 
All densities considered in this paper have a first moment. Hence, by the above, $\underline P_2$ and $\overline P_2$ exist. Integration by parts gives 
\begin{align}
    \underline P_2 = - s_{\tau_p} P + \tau_p \, p  \quad \mbox{ and } \quad
    \overline P_2 = s_{\tau_p} \overline P + \tau_p \, p \label{eq:P2-integration-by-parts}
\end{align}
where  the function $\tau_p$ is defined as follows. 
\begin{definition}
\label{def:stein-kernel}
Assume that \(\mathbb E|Z|<\infty\) and set
\(\mu:=\mathbb EZ\). The Stein covariance kernel of \(p\) is the
function
\[
 \tau_p(x)
 :=
 -\mathcal S_p\mathrm{Id}(x)\left(=
 \frac{1}{p(x)}
 \int_l^x(\mu-t)p(t)\,dt
 =
 \frac{1}{p(x)}
 \int_x^u(t-\mu)p(t)\,dt\right),
 \qquad x\in\mathcal I_p.
\]
\end{definition}

\begin{remark}
The function in Definition~\ref{def:stein-kernel} has its origins in
the weighted covariance and variance inequalities of
\cite{cacoullos_1982,cacoullos_papathanasiou_1985} and in the general
framework of \cite{stein_approx_1986}. Its interpretation as a
canonical inverse-Stein object in the density approach was developed
systematically in \cite{ley_reinert_swan_2017}. Stein kernels play an
important role in Malliavin--Stein and information-theoretic methods
\cite{nourdin_peccati_2012,nourdin_peccati_swan_2014,
ledoux_nourdin_peccati_2015}, and in weighted Poincar\'e and
concentration inequalities
\cite{saumard_2019,ernst_reinert_swan_2020}.
Under our standing assumptions,
Definition~\ref{def:stein-kernel} gives the canonical non-negative
Stein kernel, unique up to \(p\)-null sets; see \cite{doebler_2025}. 
\end{remark}

\begin{remark}
    The Stein kernel is strictly positive on $\mathcal I_p$. Moreover,
the function $\tau_pp$ extends continuously to $(l,u)$ and satisfies
the boundary conditions \(\tau_pp\to0\) when $x\downarrow l$ or
$x\uparrow u$. 
Choosing the Stein kernel as the weight for the Stein equation gives
\[
 \mathcal A_{p,\tau_p}f(x)
 =
 \tau_p(x)f'(x)+(\mu-x)f(x),
\]
because \((\tau_pp)'=(\mu-x)p\).  
\end{remark}

\subsection{The kernel array}
\label{sec:kernel-representations}

Without further ado, we now present the main objects that underpin the entire theory.  To begin with,
note how rearranging \eqref{eq:caninv} gives, for
$h\in L^1(p)$ and $x\in\mathcal I_p$,
\begin{equation}
\label{eq:inverseoperatorrepresentation}
 \mathcal S_ph(x)
 =
 \frac{\overline P(x)}{p(x)}
 \int_l^x h(v)p(v)\,dv
 -
 \frac{P(x)}{p(x)}
 \int_x^u h(v)p(v)\,dv,
\end{equation}
and thus, recalling the weighted Mills ratios
\(\underline m_{p,w}\) and \(\overline m_{p,w}\) of
\eqref{eq:weighted-mills-ratios},
\begin{equation}
\label{eq:stein-solution-h-not-C1}
 f_{p,w,h}(x)
 =
 \overline m_{p,w}(x)
 \mathbb E\bigl[h(Z)\mathbf 1_{\{Z\leq x\}}\bigr]
 -
 \underline m_{p,w}(x)
 \mathbb E\bigl[h(Z)\mathbf 1_{\{Z>x\}}\bigr].
\end{equation}
For \(i,j\geq0\), we set
\begin{equation}
\label{eq:def-kernels}
 K_{p,w}^{i,j}(x,v)
 :=
 \frac{1}{p(v)}
 \left[
 (-1)^j
 \overline m_{p,w}^{(i)}(x)
 \underline P_j(v)\mathbf 1_{\{v\leq x\}}
 -
 \underline m_{p,w}^{(i)}(x)
 \overline P_j(v)\mathbf 1_{\{v>x\}}
 \right],
\end{equation}
for \(x,v\in\mathcal I_p\). Obviously 
\begin{equation}
\label{eq:master-start}
 f_{p,w,h}(x)
 =
 \mathbb E\left[K_{p,w}^{0,0}(x,Z)h(Z)\right],
 \qquad h\in L^1(p).
\end{equation}
The kernels with  higher order indices  $i, j$ will appear when representing the higher order derivatives of the Stein solution $f_{p, w, h}$. When computing bounds on these functions, we will need to take moments and absolute moments of the resulting expressions. In this respect, for $i\geq0$, $j\geq-1$, and $x \in \mathcal I_p$ we introduce the mean and absolute mean functions 
\begin{align}
\label{eq:def-M}
 M_{p,w}^{i,j}(x)
 &:=
 (-1)^j\overline m_{p,w}^{(i)}(x)\underline P_{j+1}(x)
 -
 \underline m_{p,w}^{(i)}(x)\overline P_{j+1}(x),
 \\
\label{eq:general-absolute-kernel-envelope}
 U_{p,w}^{i,j}(x)
 &:=
 |\overline m_{p,w}^{(i)}(x)|\underline P_{j+1}(x)
 +
 |\underline m_{p,w}^{(i)}(x)|\overline P_{j+1}(x).
\end{align}
Each of the  quantities in \eqref{eq:def-kernels},  \eqref{eq:def-M} and \eqref{eq:general-absolute-kernel-envelope} is considered whenever
the integrated tails appearing in its definition are finite and
the displayed derivatives exist.  The connection between $K, M$ and $U$ is made explicit in our next proposition. 

\begin{proposition}
    \label{lma:basis-prop-of-kernels}
Fix $i,j\geq0$ and $x\in\mathcal I_p$. Assume that
\(\underline P_{j+1}(x)\), \(\overline P_{j+1}(x)\), and the
Mills-ratio is $i$ times differentiable at $x$.
Then
\(K_{p,w}^{i,j}(x,\cdot)\in L^1(p)\), with
\begin{align}
 \mathbb E[K_{p,w}^{i,j}(x,Z)]
 &=
 M_{p,w}^{i,j}(x),
 \label{eq:expectation-of-kernel}\\
 \mathbb E[|K_{p,w}^{i,j}(x,Z)|]
 &=
 U_{p,w}^{i,j}(x).
 \label{eq:integrability-of-kernel}
\end{align}
\end{proposition}

To illustrate, we compute the first expressions explicitly. 

 \begin{example}
    \label{ex:low-order-kernel-means}
Assume that \(\mathbb E|Z|<\infty\) and that all the derivatives
displayed below exist.
 By  \eqref{eq:expectation-of-kernel}, the kernel means are  
\begin{align}
   &  M_{p,w}^{0,0} =0,
 \qquad
 M_{p,w}^{0,-1}=M_{p,w}^{1,0}=-\frac1w,
 \qquad
 M_{p,w}^{0,1}=-\frac{\tau_p}{w}\label{eq:low-order-mean-01}\\
 &  M_{p,w}^{1,1}
 =
 -\left(\frac{\tau_p}{w}\right)',
 \qquad
 M_{p,w}^{2,1}
 =
 -\frac1w-\left(\frac{\tau_p}{w}\right)''.\label{eq:low-order-mean-21}
\end{align}
By \eqref{eq:integrability-of-kernel}, the corresponding 
absolute means are
\begin{align}
U_{p,w}^{0,0} 
 &=
 2\frac{P\overline P}{|w|p}
 \quad\mbox{and}\quad
 U_{p,w}^{0,1}
 =
 \frac{\tau_p}{|w|}.
 \label{eq:low-order-abs-mean-01}
\end{align}
If, moreover, \(w>0\) and
\begin{equation}
\label{eq:cond-first-order-signs-mills}
 \overline m_{p,w}'\leq0,
 \quad\mbox{and}\quad
 \underline m_{p,w}'\geq0
\end{equation}
on $\mathcal I_p$, then $K_{p,w}^{1,0}(x,v)
 =
 \overline m_{p,w}'(x)\mathbf 1_{\{v\leq x\}}
 -
 \underline m_{p,w}'(x)\mathbf 1_{\{v>x\}}
 \leq0$
on $\mathcal I_p\times\mathcal I_p$, and therefore
\begin{equation}
\label{eq:absolute-mean-K10-simplified}
U_{p, w}^{1, 0} = - M_{p, w}^{1, 0}  
 =
 \frac1{w}.
\end{equation}
To go further, choosing $w=\tau_p$ and algebraic manipulation yield
\begin{equation}
\label{eq:stein-kernel-mills-identities}
 \underline m_{p,\tau_p}'
 =
 \frac{\underline P_2}{\tau_p^2p} \ge 0,
 \qquad
 \overline m_{p,\tau_p}'
 =
 -\frac{\overline P_2}{\tau_p^2p} \le 0.
\end{equation}
Consequently, \eqref{eq:cond-first-order-signs-mills} is verified so \eqref{eq:absolute-mean-K10-simplified} hold. Moreover
\[
 U_{p,\tau_p}^{1,1}
 =
 2\underline P_2\overline P_2/(\tau_p^2p).
\]
\end{example}

We now wish to connect the kernels at indices $i, j$  to those at indices $i\pm1, j\pm1$. This we achieve through integration of the function $v \mapsto K_{p, w}^{i, j}(x, v)$ and differentiation of $x \mapsto K_{p, w}^{i, j}(x, v)$. The following lemmas summarise this (see Appendix \ref{app:proof-of-no-other-orders} for the detailed argument). 
\begin{lemma}[Kernel calculus -- integration by parts]\label{lma:Kernel calculus -- integration by parts}
\label{cor:kernelcalculus}
  Fix $i, j \ge 0$. Instate the conditions from Proposition \ref{lma:basis-prop-of-kernels}. Let $h$ be locally
absolutely continuous on $(l,u)$ and such  that
\begin{align*}
 &\int_l^x \underline P_j(v)|h(v)|\,dv
 +\int_x^u \overline P_j(v)|h(v)|\,dv
 <\infty,\\
 &\int_l^x \underline P_{j+1}(v)|h'(v)|\,dv
 +\int_x^u \overline P_{j+1}(v)|h'(v)|\,dv
 <\infty.
\end{align*}
Assume also that
\[
 \lim_{v\downarrow l}h(v)\underline P_{j+1}(v)
 =
 \lim_{v\uparrow u}h(v)\overline P_{j+1}(v)
 =
 0.
\]
Then 
\begin{align}
 \mathbb E[K_{p,w}^{i,j}(x,Z)h(Z)]
 &=
 \mathbb E[K_{p,w}^{i,j+1}(x,Z)h'(Z)]
 +
 M_{p,w}^{i,j}(x)h(x), \quad x \in \mathcal I_p. 
 \label{eq:master-vertical-move}
 \end{align}
 
\end{lemma}

\begin{lemma}[Kernel calculus -- differentiation]\label{lma:Kernel calculus -- differentiation}
Fix $i, j \ge 0$. Instate the conditions from Proposition \ref{lma:basis-prop-of-kernels}. Let $h$ be locally
absolutely continuous on $(l,u)$ and such  that
\begin{align*}
 &\int_l^x \underline P_j(v)|h(v)|\,dv
 +\int_x^u \overline P_j(v)|h(v)|\,dv
 <\infty.
\end{align*}
Suppose also that the  Mills ratio derivatives at order $i$ are  finite and  locally absolutely continuous.
Then  the function $x \mapsto \mathbb E[K_{p,w}^{i,j}(x,Z)h(Z)]$ is locally absolutely continuous on $\mathcal I_p$ and 
 \begin{equation} \label{eq:master-horizontal-move} 
  (\mathbb E[K_{p,w}^{i,j}(x,Z)h(Z)])' =   \mathbb E[K_{p,w}^{i+1,j}(x,Z)h(Z)]
 -
 M_{p,w}^{i,j-1}(x)h(x), \quad x \in \mathcal I_p  
 \end{equation}
 where the derivative on the lhs is understood in the sense of Remark \ref{rk:abscontderiv}. 
  
\end{lemma}
\begin{remark}
    A convenient verifiable sufficient condition for \eqref{eq:master-vertical-move}  and \eqref{eq:master-horizontal-move}  to hold simultaneously is that $h$ and $h'$
have polynomial growth of degree $r$, that
\(\mathbb E|Z|^{j+r+1}<\infty\), and that the displayed Mills-ratio
derivatives are finite and locally absolutely continuous. 
\end{remark}
Identities \eqref{eq:master-vertical-move} and \eqref{eq:master-horizontal-move} are the two  basic steps of the ``kernel calculus'' underpinning our theory. Starting from  \eqref{eq:master-start}, and applying either of these two steps yields everywhere defined representations of the derivatives of the Stein solution as Figure~\ref{fig:kernel-array-calculus} illustrates.

\begin{figure}[htbp]
\centering

\resizebox{\textwidth}{!}{
\begin{tikzpicture}[
  >=Stealth,
  kernel/.style={
    draw=blue!45!black,
    rounded corners=2.5pt,
    fill=gray!4,
    minimum width=1.35cm,
    minimum height=.68cm,
    inner sep=2pt,
    font=\small
  },
  move/.style={
    ->, semithick, draw=blue!55!black
  },
  index/.style={
    font=\scriptsize, text=blue!55!black
  },
  corr/.style={
    font=\scriptsize, text=orange!75!black, inner sep=1.6pt
  },
]

\path (-1.90,2.30) coordinate (leftanchor);

\draw[move] (.6,4.85) -- (5.55,4.85)
  node[midway,above,font=\scriptsize,text=blue!55!black]
  {differentiate \(h\) by \eqref{eq:master-vertical-move}};
\draw[move] (-1.75,4.00) -- (-1.75,-0.20);
\node[font=\scriptsize,text=blue!55!black,rotate=90]
  at (-2,1.95) {differentiate in \(x\) by \eqref{eq:master-horizontal-move}};

\node[index] at (0,4.60)    {$j=0$};
\node[index] at (2.80,4.60) {$j=1$};
\node[index] at (5.60,4.60) {$j=2$};
\node[index] at (8.00,4.60) {$\cdots$};

\node[index] at (-1.25,3.90)  {$i=0$};
\node[index] at (-1.25,2.30)  {$i=1$};
\node[index] at (-1.25,0.70)  {$i=2$};
\node[index] at (-1.25,-0.90) {$\vdots$};

\node[kernel] (k00) at (0,3.90)    {$K_{p,w}^{0,0}$};
\node[kernel] (k01) at (2.80,3.90) {$K_{p,w}^{0,1}$};
\node[kernel] (k02) at (5.60,3.90) {$K_{p,w}^{0,2}$};
\node (r0) at (8.00,3.90) {$\cdots$};

\node[kernel] (k10) at (0,2.30)    {$K_{p,w}^{1,0}$};
\node[kernel] (k11) at (2.80,2.30) {$K_{p,w}^{1,1}$};
\node[kernel] (k12) at (5.60,2.30) {$K_{p,w}^{1,2}$};
\node (r1) at (8.00,2.30) {$\cdots$};

\node[kernel] (k20) at (0,0.70)    {$K_{p,w}^{2,0}$};
\node[kernel] (k21) at (2.80,0.70) {$K_{p,w}^{2,1}$};
\node[kernel] (k22) at (5.60,0.70) {$K_{p,w}^{2,2}$};
\node (r2) at (8.00,0.70) {$\cdots$};

\node (v0) at (0,-0.90)    {$\vdots$};
\node (v1) at (2.80,-0.90) {$\vdots$};
\node (v2) at (5.60,-0.90) {$\vdots$};
\node at (8.00,-0.90) {$\ddots$};

\draw[move] (k00) -- (k01)
  node[midway,above,corr]{$M_{p,w}^{0,0}h$};
\draw[move] (k01) -- (k02)
  node[midway,above,corr]{$M_{p,w}^{0,1}h'$};
\draw[move] (k02) -- (r0)
  node[midway,above,corr]{$M_{p,w}^{0,2}h''$};

\draw[move] (k10) -- (k11)
  node[midway,above,corr]{$M_{p,w}^{1,0}h$};
\draw[move] (k11) -- (k12)
  node[midway,above,corr]{$M_{p,w}^{1,1}h'$};
\draw[move] (k12) -- (r1)
  node[midway,above,corr]{$M_{p,w}^{1,2}h''$};

\draw[move] (k20) -- (k21)
  node[midway,above,corr]{$M_{p,w}^{2,0}h$};
\draw[move] (k21) -- (k22)
  node[midway,above,corr]{$M_{p,w}^{2,1}h'$};
\draw[move] (k22) -- (r2)
  node[midway,above,corr]{$M_{p,w}^{2,2}h''$};

\draw[move] (k00) -- (k10)
  node[midway,right,corr]{$-M_{p,w}^{0,-1}h$};
\draw[move] (k10) -- (k20)
  node[midway,right,corr]{$-M_{p,w}^{1,-1}h$};
\draw[move] (k20) -- (v0)
  node[midway,right,corr]{$-M_{p,w}^{2,-1}h$};

\draw[move] (k01) -- (k11)
  node[midway,right,corr]{$-M_{p,w}^{0,0}h'$};
\draw[move] (k11) -- (k21)
  node[midway,right,corr]{$-M_{p,w}^{1,0}h'$};
\draw[move] (k21) -- (v1)
  node[midway,right,corr]{$-M_{p,w}^{2,0}h'$};

\draw[move] (k02) -- (k12)
  node[midway,right,corr]{$-M_{p,w}^{0,1}h''$};
\draw[move] (k12) -- (k22)
  node[midway,right,corr]{$-M_{p,w}^{1,1}h''$};
\draw[move] (k22) -- (v2)
  node[midway,right,corr]{$-M_{p,w}^{2,1}h''$};

\end{tikzpicture}
}

\vspace{1.1em}

\begin{tikzpicture}
\node[
  draw=gray!55, fill=gray!2, rounded corners=4pt,
  align=center, inner sep=8pt, text width=12.4cm
] {
  \textbf{A path ending at \((n,k)\)}\\[5pt]
  $\displaystyle
  \begin{aligned}
  f_{p,w,h}^{(n)}(x)
  &= \textcolor{blue!65!black}{
       \mathbb E\left[K_{p,w}^{n,k}(x,Z)h^{(k)}(Z)\right]}
  \\[3pt]
  &\quad+
  {\textcolor{orange!75!black}{\text{accumulated pointwise corrections}}}.
  \end{aligned}$
};
\end{tikzpicture}

\caption{The kernel array: Vertical moves differentiate the Stein
solution, horizontal moves transfer a derivative to the test function,
and the array continues in both directions. The orange label on each
arrow is the pointwise term at \(x\) created by that move (the argument
\(x\) is suppressed on the arrows for readability).
At each vertical move, all previously accumulated
pointwise corrections must also be differentiated. Each move requires
the corresponding regularity, integrability and boundary conditions
from Lemma~\ref{lma:Kernel calculus -- integration by parts} and Lemma~\ref{lma:Kernel calculus -- differentiation}.
A path therefore gives the kernel term shown in blue together with
the accumulated corrections. These corrections are independent of the path taken, by Lemma~\ref{lma:path-independent}, and may vanish under certain conditions, for example those of
Corollary~\ref{cor:closed-three-neighbour-calculus}.}
\label{fig:kernel-array-calculus}
\end{figure}

\subsection{From the kernels to the representations}
\label{subsec:all-order-calculus}
In principle, under the appropriate conditions on $p, w$ and $h$, the kernel calculus can be used to obtain representations of $f_{p,w,h}^{(n)}$ with respect to derivatives of the test function $h$ up to order $k$, for arbitrary integers $n$ and $k$. As illustrated schematically in Figure \ref{fig:kernel-array-calculus}, these expressions  (which are obtained by following a path starting from $i=0$, $j=0$ towards line $i=n$ and column $j=k$ of our array) involve accumulated pointwise terms (which we call ``corrections'') which depend on all the lower derivatives of the test functions $h$ (see the panel below the illustration). 

This observation opens a new  research question: given $p$ what choice of  $w$  allows to impose minimal assumptions on $h$ to obtain a useful representation for $f_{p,w,h}^{(n)}$. This is an aesthetic choice which has a practical implication: assumptions on the test function $h$ guide the underlying probability metric that is considered, and thus one would want naturally to impose as few assumptions on $h$ as possible. In this respect, we decide to concentrate our efforts on obtaining one-term representations of $f_{p,w,h}^{(n)}$, i.e\ representations of the form 
\begin{equation*}
    f_{p,w,h}^{(n)}(x)
  = 
       \mathbb E\left[K_{p,w}^{n,k}(x,Z)h^{(k)}(Z)\right]
\end{equation*}
In the Gaussian case, \cite{Gaunt_Li,gaunt2024steinfactorssteinsmethod} showed (recall the last statement of Theorem \ref{thm:known-gaussian-stein-factors}) that the orders $k\in\{n-1,n,n+1\}$ are the only ones for which one can hope for such one-term Stein factors. Our first result shows that this observation extends to all distributions satisfying Assumption~\ref{ass:density-support}.

\begin{lemma}
\label{lem:no-other-orders}
Let $p$ satisfy Assumption~\ref{ass:density-support} and let $n,k\geq0$ be integers.
\begin{enumerate}
\item[(a)]
If $k\geq n+2$ and $\mathbb E|Z|^{n+1}<\infty$,
there exists a polynomial $h$ of degree at most $n+1$ such that
\[
 \|h^{(k)}\|=0
 \qquad\text{but}\qquad
 \|f_{p,w,h}^{(n)}\|>0.
\]
\item[(b)]
If $k\leq n-2$, $\mathbb E|Z|^k<\infty$, and
$a_w$ and $1/w$ are $k$ times continuously differentiable
near some $x_0\in\mathcal I_p$, then the function
\[
 h_0(x)
 =
 \operatorname{sgn}(x-x_0)\frac{(x-x_0)^k}{k!}
\]
satisfies $\|h_0^{(k)}\|=1$ and
\[
 f_{p,w,h_0}^{(k+1)}(x_0+)
 -
 f_{p,w,h_0}^{(k+1)}(x_0-)
 =
 \frac{2}{w(x_0)},
\]
so that $\|f_{p,w,h_0}^{(n)}\|=\infty$ in the sense of
Notation~\ref{not:derivative-norms}.
\end{enumerate}
Consequently, under the respective assumptions of (a) and (b),
there is no finite constant $C$ such that
\[
 \|f_{p,w,h}^{(n)}\|
 \leq C\|h^{(k)}\|
\]
for every $h\in L^1(p)$ with $\|h^{(k)}\|<\infty$,
with the seminorms understood according to
Notation~\ref{not:derivative-norms}.
\end{lemma}
\begin{remark}
   As already mentioned above, for any given $n$ and $k$, one can still bound $f_{p, w, h}^{(n)}$ in terms of  $h^{(k)}$ at the price of carrying additional terms involving $h^{(j)}$ for $j \in \{0, \ldots, k-1\}$. Although the resulting bounds are unlikely to be sharp (see below), we do not claim they can not be useful. Mixed bounds on Stein solutions have already been obtained through an iterative argument,  see  \cite{dobler_gaunt_vollmer_2017} for an overview and applications. It could be interesting to study the connection between their approach and ours in more detail. 
\end{remark}

Before giving our representations, we record a useful  lemma which, in particular, implies that  the accumulated correction terms do not depend on the chosen path. 

\begin{lemma}\label{lma:path-independent}
For $i,j\geq0$,
\begin{equation}
\label{eq:M-derivative-recursion}
 \bigl(M_{p,w}^{i,j}\bigr)'
 =
 M_{p,w}^{i+1,j}-M_{p,w}^{i,j-1}.
\end{equation}
Iterating yields, for every \(i\geq1\) and \(j\geq0\),
\begin{equation}
\label{eq:M-all-order-recursion}
 M_{p,w}^{i,j}
 =
 \bigl(M_{p,w}^{0,j}\bigr)^{(i)}
 +
 \sum_{r=0}^{i-1}
 \bigl(M_{p,w}^{r,j-1}\bigr)^{(i-1-r)}.
\end{equation}
These identities hold on each component of
$\mathcal I_p$ whenever the displayed quantities are finite and
the indicated derivatives exist.
\end{lemma}

Starting from \eqref{eq:master-start} that is the upper-left corner of the array, we can now read off from Figure~\ref{fig:kernel-array-calculus}, at least formally, three representations of
$f_{p,w,h}^{(n)}$. We call them the \emph{lower neighbour},
\emph{diagonal}, and \emph{upper neighbour} representations according
as they involve $h^{(n-1)}$, $h^{(n)}$, or $h^{(n+1)}$, respectively.
To reach the lower neighbour, first move down, then repeat the pair of
steps $(\rightarrow,\downarrow)$.
For the diagonal, repeat $(\rightarrow,\downarrow)$ from the start.
For the upper neighbour, first move right, then repeat
$(\downarrow,\rightarrow)$. By Lemma~\ref{lma:path-independent}, any
two paths leading to the same entry yield kernels that agree almost
everywhere.

\begin{theorem}
\label{thm:kernel-representations-of-derivatives}
Let \(n\geq1\) and $h\in L^1(p)$.
Each representation below is asserted under the regularity,
integrability, and boundary conditions of
Lemma~\ref{lma:Kernel calculus -- integration by parts} at, respectively, $j= \{n-1,n,n+1\}$ and Lemma~\ref{lma:Kernel calculus -- differentiation} at order $i=n$. The identities hold almost everywhere on
$\mathcal I_p$, and pointwise wherever the corresponding
differentiation identities apply.
Then
\begin{align}
\begin{split}
 f_{p,w,h}^{(n)}(x)
 &=
 \mathbb E[K_{p,w}^{n,n-1}(x,Z)h^{(n-1)}(Z)]
 -
 M_{p,w}^{1,0}(x)h^{(n-1)}(x)
 \\
 &\quad+
 \sum_{r=0}^{n-2}
 \left[
  \bigl(M_{p,w}^{r+1,r}-M_{p,w}^{1,0}\bigr)'h^{(r)}
 \right]^{(n-2-r)}(x),
 \end{split}
 \label{eq:general-lower-neighbour-representation}\\
 \begin{split}
 f_{p,w,h}^{(n)}(x)
 &=
 \mathbb E[K_{p,w}^{n,n}(x,Z)h^{(n)}(Z)]
 +
 \bigl(M_{p,w}^{n,n-1}(x)-M_{p,w}^{1,0}(x)\bigr)
 h^{(n-1)}(x)
 \\
 &\quad+
 \sum_{r=0}^{n-2}
 \left[
  \bigl(M_{p,w}^{r+1,r}-M_{p,w}^{1,0}\bigr)'h^{(r)}
 \right]^{(n-2-r)}(x),
 \end{split}
 \label{eq:general-diagonal-representation}
 \\
 \begin{split}
 f_{p,w,h}^{(n)}(x)
 &=
 \mathbb E[K_{p,w}^{n,n+1}(x,Z)h^{(n+1)}(Z)]
 +
 M_{p,w}^{n,n}(x)h^{(n)}(x)
 \\
 &\quad+
 \sum_{r=0}^{n-1}
 \left[
  \bigl(M_{p,w}^{r,r}\bigr)'h^{(r)}
 \right]^{(n-1-r)}(x).
 \end{split}
 \label{eq:general-upper-neighbour-representation}
\end{align}
Empty sums are understood to be zero.
\end{theorem}

These representations give straightforward sufficient conditions
for relating the $n$th derivative to a single derivative of $h$
(of order either $n$, $n+1$, or $n-1$).

\begin{corollary}
\label{cor:closed-three-neighbour-calculus}
Assume that, at all indices needed below,
\begin{equation}
\label{eq:closed-three-neighbour-condition}
 M_{p,w}^{r,r-1}=-\frac1w,
 \qquad
 M_{p,w}^{r,r}=0,
 \qquad r\geq0.
\end{equation}
Then the correction terms in
Theorem~\ref{thm:kernel-representations-of-derivatives}
cancel or are absorbed by centring the lower-neighbour
representation. Under the corresponding regularity,
integrability, and boundary assumptions,
\begin{align}
 f_{p,w,h}^{(n)}(x)
 &=
 \mathbb E\left[
 K_{p,w}^{n,n-1}(x,Z)
 \bigl(h^{(n-1)}(Z)-h^{(n-1)}(x)\bigr)
 \right],
 \qquad n\geq1,
 \label{eq:closed-lower-neighbour}\\
 f_{p,w,h}^{(n)}(x)
 &=
 \mathbb E\left[
 K_{p,w}^{n,n}(x,Z)h^{(n)}(Z)
 \right],
 \qquad n\geq0,
 \label{eq:closed-diagonal}\\
 f_{p,w,h}^{(n)}(x)
 &=
 \mathbb E\left[
 K_{p,w}^{n,n+1}(x,Z)h^{(n+1)}(Z)
 \right],
 \qquad n\geq0.
 \label{eq:closed-upper-neighbour}
\end{align}
\end{corollary}

\begin{remark}
The two identities in
\eqref{eq:closed-three-neighbour-condition} are linked by
\eqref{eq:M-derivative-recursion}:
\[
 \bigl(M_{p,w}^{r,r}\bigr)'
 =
 M_{p,w}^{r+1,r}-M_{p,w}^{r,r-1}.
\]
Thus the first identity is implied by the second
at successive indices, since
$M_{p,w}^{1,0}=-1/w$ for every target $p$ and weight $w$
(and $M_{p,w}^{0,-1}=-1/w$ as well).
The converse fails: for the exponential law
with mean $\beta>0$
and \(w=1\),
\[
 M_{p,1}^{r,r-1}=-1,
 \qquad
 M_{p,1}^{r,r}=-\beta,
 \qquad  {r\geq1},
\]
{whereas $M_{p,1}^{0,0}=0$}
(see Lemma~\ref{lem:exp-Mn-Un}).
\end{remark}

Taking absolute values and using the triangular inequality immediately leads to Stein envelopes, which in turn can be maximised to yield Stein factors, as defined in Definition \ref{def:stein-factors-and-envelopes}. Since the envelopes can be achieved by choosing the appropriate test function, they are sharp. More precisely, under the conditions of Corollary \ref{cor:closed-three-neighbour-calculus}, 
\begin{align}
 |f_{p,w,h}^{(n)}(x)| \le \min\Big\{ 
 U_{p,w}^{n,n-1}(x)
 \mathrm{osc}(h^{(n-1)}),
 U_{p,w}^{n,n}(x)\|h^{(n)}\|,
 U_{p,w}^{n,n+1}(x)\|h^{(n+1)}\|\Big\}
 \label{eq:closed-envelopes}
\end{align}
where the first argument of the minimum is stated for $n\ge 1$ and the other two for $n\ge0$. Taking suprema of the envelopes leads (formally) to a bound of similar aspect to \eqref{eq:gaussian-general-magic-factors}, for general density $p$ and weight $w$. Of course the usefulness of \eqref{eq:closed-envelopes} hinges on the properties 
of $U_{p,w}^{n,k}$, for $k\in \{n-1,n,n+1\}$. We explore this in  Section~\ref{sec:closed-kernel-calculi}.

\section{Consequences}
\label{sec:closed-kernel-calculi}
In this section, we unpack the consequences of the representations from Theorem~\ref{thm:kernel-representations-of-derivatives} to obtain Stein envelopes and factors under various assumptions on $p, w$ and $h$.
For the whole section, identities are asserted  under the
regularity, integrability, and boundary conditions of
Theorem~\ref{thm:kernel-representations-of-derivatives}. All the norms are considered when they exists in the sense of Notation~\ref{not:derivative-norms}.

\subsection{Low-order envelopes and factors}
\label{subsec:universal-low-order}
Using the computations of Example~\ref{ex:low-order-kernel-means}, the representations from Theorem~\ref{thm:kernel-representations-of-derivatives} simplify at order $n=0,1,2$ as follows.

\begin{theorem}
\label{thm:universal-low-order-calculus}
Let $h\in L^1(p)$. Then
\begin{align}
 f_{p,w,h}(x)
 &=
 \mathbb E[K_{p,w}^{0,0}(x,Z)h(Z)],
 \label{eq:ker00}\\
 &=
 \mathbb E[K_{p,w}^{0,1}(x,Z)h'(Z)],
 \label{eq:ker01}\\
 f_{p,w,h}'(x)
 &=
 \mathbb E\left[
 K_{p,w}^{1,0}(x,Z)\bigl(h(Z)-h(x)\bigr)
 \right],
 \label{eq:ker10}\\
 &=
 \mathbb E[K_{p,w}^{1,1}(x,Z)h'(Z)],
 \label{eq:ker11}\\
 &=
 \mathbb E[K_{p,w}^{1,2}(x,Z)h''(Z)]
 -
 \left(\frac{\tau_p(x)}{w(x)}\right)'h'(x),
 \label{eq:ker12}\\
 f_{p,w,h}''(x)
 &=
 \mathbb E[K_{p,w}^{2,1}(x,Z)h'(Z)]
 +
 \frac{1}{w(x)}h'(x),
 \label{eq:ker21}\\
 &=
 \mathbb E[K_{p,w}^{2,2}(x,Z)h''(Z)]
 -
 \left(\frac{\tau_p(x)}{w(x)}\right)''h'(x),
 \label{eq:ker22}
 \\
  & = \mathbb E[K_{p,w}^{2,3}(x,Z)h'''(Z)] + M_{p, w}^{2, 2}(x) h''(x) 
 -
 \left(\frac{\tau_p(x)}{w(x)}\right)''h'(x).
 \label{eq:ker23}
\end{align}
Moreover, \eqref{eq:ker00}, \eqref{eq:ker11} and \eqref{eq:ker21} may equivalently be written as
\begin{align}
\label{eq:ker00-centered}
f_{p,w,h}(x)
 &=
 \mathbb E[K_{p,w}^{0,0}(x,Z)(h(Z)-h(x))]\\
\label{eq:ker11-centered}
f_{p,w,h}'(x)
 &=
 \mathbb E[K_{p,w}^{1,1}(x,Z)(h'(Z)-h'(x))] -  \left(\frac{\tau_p(x)}{w(x)}\right)'h'(x)
\\
\label{eq:ker21-centered}
 f_{p,w,h}''(x)
 & =
 \mathbb E\left[
 K_{p,w}^{2,1}(x,Z)
 \bigl(h'(Z)-h'(x)\bigr)
 \right]
 -
 \left(\frac{\tau_p(x)}{w(x)}\right)''h'(x).
\end{align}
\end{theorem}

Envelopes and factors on the solution follow directly.

\begin{proposition}
\label{prop:stein-solution-low-order-factors}
Let \(h\in L^1(p)\) and \(x\in\mathcal I_p\). Then, the solution admits the following envelopes 
\begin{equation}
\label{eq:general-low-order-solution-envelope}
 |f_{p,w,h}(x)|
 \leq
 \min\left\{
  \frac{P(x)\overline P(x)}
 {|w(x)|p(x)}
  \operatorname{osc}(h),
  \frac{\tau_p(x)}{{|w(x)|}}
  \|h^{(1)}\|
 \right\},
\end{equation}
(the factors follow  automatically by maximizing the envelopes). 
\end{proposition}

\begin{remark}\label{rk:dobler} In the literature, Stein factor type bounds for the solutions $f_{p, w, h}$ are sometimes expressed in terms of either $\|h\|$ or  $\|h - \mathcal Ph\|$ rather than   $\mathrm{osc}(h)$.  There is a priori no hierarchy between these three norms and the corresponding Stein factors.  Sharp bounds in terms of $\|h\|$ follow from \eqref{eq:ker00}. Since $M^{0, 0}_{p, w}(x) = 0$, instead of subtracting $h(x)$ to obtain \eqref{eq:ker00-centered} one can subtract $\mathcal Ph$ to obtain 
\begin{equation*}
    f_{p,w,h}(x)
 =
 \mathbb E[K_{p,w}^{0,0}(x,Z)(h(Z)-\mathcal Ph)] 
\end{equation*} 
so that 
\begin{equation*}
    \|f_{p,w,h}\|
 \le 
 U_{p,w}^{0,0}(x) \|h-\mathcal Ph\|. 
\end{equation*} 
When the support of $p$ is connected, under condition \eqref{eq:cond-first-order-signs-mills}, we have 
\begin{equation*}
    \sup_{x\in (l,u)}U_{p,w}^{0,0}(x) = U^{0,0}_{p,w}(m)= \frac{1}{2 w(m)p(m)}
\end{equation*}
with $m$ the median of $p$, recovering \cite[Proposition~3.9]{dobler_beta_2015}.  
\end{remark}

 Moving to the first derivative, recall from Example~\ref{ex:low-order-kernel-means} that under
\eqref{eq:cond-first-order-signs-mills} the absolute kernel means in
\eqref{eq:general-absolute-kernel-envelope} reduce to
\begin{equation}
\label{eq:first-derivative-absolute-kernel-means}
 U_{p,w}^{1,j}
 =
 -\overline m_{p,w}'\underline P_{j+1}
 +
 \underline m_{p,w}'\overline P_{j+1}
 \geq0,
 \qquad j\geq0,
\end{equation}
wherever the displayed quantities are defined.
In particular, the sign condition gives
$K_{p,w}^{1,0}\leq0$ and hence, since
$M_{p,w}^{1,0}=-1/w$, necessarily $w>0$.
Substitution of these absolute means into the low-order kernel
representations gives the first-derivative envelope.

\begin{proposition}
\label{prop:first-derivative-factors}
Let $h\in L^1(p)$ and
suppose, for the first argument in the minimum to hold, that~\eqref{eq:cond-first-order-signs-mills} is verified.
Then, almost everywhere on \(\mathcal I_p\),
\begin{equation}
\label{eq:first-derivative-envelop}
\begin{split}
 |f_{p,w,h}'(x)|
 \leq
 \min\Bigg\{&
 \frac{1}{w(x)}\operatorname{osc}(h),
 \quad
 U_{p,w}^{1,1}(x)\mathrm{osc}(h^{(1)})+\left|\left(\frac{\tau_p}{w}\right)'(x)\right|\|h^{(1)}\|,
 \\
 &
 \left|
 \left(\frac{\tau_p(x)}{w(x)}\right)'
 \right|\|h^{(1)}\|
 +
 U_{p,w}^{1,2}(x)\|h^{(2)}\|
 \Bigg\}
\end{split}
\end{equation}
(again, the factors follow  automatically). 

\end{proposition}

\begin{remark}
  A similar remark to Remark \ref{rk:dobler} can be made here. In the literature, bounds for $f_{p,w,h}'$ involving no derivatives of $h$ are sometimes expressed in terms of either $\|h\|$ or $\|h-\mathcal{P}h\|$ rather than $\mathrm{osc}$. Since $M_{p,w}^{1,0}= -1/w\neq 0$, one cannot recover one-terms bound better than the one obtained from the inequalities
    \begin{equation*}
        \mathrm{osc}(h)/2 \le \|h-\mathcal{P}h\| \le \mathrm{osc}(h).
    \end{equation*}
\end{remark}
The first term in the minimum appearing in \eqref{eq:first-derivative-envelop} is  uniformly bounded whenever
\(w\) is bounded away from zero. In general, obtaining sharp uniform bounds from the second and third envelopes and factors is a nontrivial analytic problem that generally requires distribution-specific estimates. 
\begin{remark}
\label{rem:stein-kernel-first-derivative}
Choosing \(w=\tau_p\) lead to simplifications in \eqref{eq:first-derivative-envelop}. From \eqref{eq:stein-kernel-mills-identities}, we computed in Example~\ref{ex:low-order-kernel-means} 
\[
 U_{p,\tau_p}^{1,0}=\frac{1}{\tau_p}, \quad \text{ and }\quad U_{p,\tau_p}^{1,1}
 =
 2\frac{\underline P_2\overline P_2}{\tau_p^2p},
\]
so that
\eqref{eq:first-derivative-envelop} simplifies to
\begin{equation}
\label{eq:tau-first-derivative-combined-envelope}
\begin{split}
 |f_{p,\tau_p,h}'(x)|
 \leq
 \min\Bigg\{&
 \frac{1}{\tau_p(x)}\operatorname{osc}(h),
 \quad
 2\frac{\underline P_2(x)\overline P_2(x)}
        {\tau_p^2(x)p(x)}
 \|h^{(1)}\|,
 \\
 &
 \frac{
 \overline P_2(x)\underline P_3(x)
 +
 \underline P_2(x)\overline P_3(x)}
 {\tau_p^2(x)p(x)}
 \|h^{(2)}\|
 \Bigg\},
\end{split}
\end{equation}
the third argument in the minimum being valid only if
$\mathbb E|Z|^2<\infty$.  
The corresponding  Stein factors follow. 
The first factor is uniformly bounded whenever $\tau_p$ is bounded away
from zero;
see, for example, \cite{NourdinViens2009} for examples of distributions satisfying this condition.
\end{remark}

The envelopes and factors for $f''_{p,w,h}$ following from Theorem~\ref{thm:universal-low-order-calculus} do not, without additional assumptions on $p$, enjoy the same kind of simplifications as those for the solution and its first derivative, even for the choice $w=\tau_p$. We chose not to record them in full generality but to present two examples. 

\begin{example}[Symmetrized Maxwell distribution]
\label{subsec:symmetrized-maxwell}

Let \(M\) have the Maxwell distribution with scale \(\sigma>0\), let
\(\varepsilon\) be an independent Rademacher variable, and set
\(Z_\sigma:=\varepsilon M\). Then \(Z_\sigma\) has density
\begin{equation}
\label{eq:symmax-density}
 p_\sigma(x)
 =
 \frac{x^2}{\sqrt{2\pi}\sigma^3}
 \exp\left(-\frac{x^2}{2\sigma^2}\right),
 \qquad x\in\mathbb R.
\end{equation}
Writing  $P^\sigma$ and $\overline{P}^\sigma$ for the associated cdf and survival function, denoting as before the standard Gaussian density,
distribution, and survival functions by \(\varphi\), \(\Phi\), and
\(\overline\Phi\), respectively, we have for \(y=x/\sigma\),
\[
 P^\sigma(x)=\Phi(y)-y\varphi(y),
 \qquad
 \overline P^\sigma(x)=\overline\Phi(y)+y\varphi(y).
\]
The law is centred and symmetric, with moments of every order:
\[
 \mathbb E|Z_\sigma|^k
 =
 \sigma^k2^{k/2}
 \frac{\Gamma((k+3)/2)}{\Gamma(3/2)},
 \qquad k>-3,
\]
so \(\operatorname{Var}(Z_\sigma)=3\sigma^2\). The density-Stein solution with \(w\equiv1\) is
generally singular at \(0\) because the score
\[
 \frac{p_\sigma'(x)}{p_\sigma(x)}
 =
 \frac2x-\frac{x}{\sigma^2},
 \qquad x\neq0,
\]
 is singular at the origin. The Stein kernel,
\[
 \tau_{p_\sigma}(x)
 =
 \sigma^2+\frac{2\sigma^4}{x^2},
 \qquad x\neq0,
\]
removes this singularity.  Indeed, although \(p_\sigma(0)=0\) and \(\tau_{p_\sigma}(0)\) is not finite, the product 
\[
 \tau_{p_\sigma}(x)p_\sigma(x)
 =
 \sigma(y^2+2)\varphi(y)
\]
has a strictly positive continuous extension at the origin. We therefore
work with the Stein-kernel weighted equation (i.e.\ we choose  \(w=\tau_{p_\sigma}\)) which reads 
\begin{equation}
\label{eq:symmax-kernel-equation}
\left(\sigma^2+\frac{2\sigma^4}{x^2}\right)f'(x)-xf(x)
 =
 h(x)-\mathbb E[h(Z_\sigma)],
 \qquad \text{a.e. }x\neq0.
\end{equation}
We apply
the kernel formulas separately on \((-\infty,0)\) and \((0,\infty)\) and
join the canonical weighted solution continuously at \(0\). By \eqref{eq:stein-kernel-mills-identities}
\[
 \underline m_{p_\sigma,\tau_{p_\sigma}}'\geq0,
 \qquad
 \overline m_{p_\sigma,\tau_{p_\sigma}}'\leq0,
\]
therefore the general envelopes give, almost everywhere,
\begin{align}
 |f_{p_\sigma,\tau_{p_\sigma},h}(x)|
 &\leq
 \min\left\{\frac{P^\sigma(x)\overline{P}^\sigma(x)}{\tau_{p_\sigma}(x)p_{\sigma}(x)}\operatorname{osc}(h),
  \|h^{(1)}\|
 \right\},
 \label{eq:symmax-solution-envelope}\\
 |f_{p_\sigma,\tau_{p_\sigma},h}'(x)|
 &\leq
 \min\left\{
  \frac1{\tau_{p_\sigma}(x)}\operatorname{osc}(h),
  U_{p_\sigma,\tau_{p_\sigma}}^{1,1}(x)\|h^{(1)}\|,
  U_{p_\sigma,\tau_{p_\sigma}}^{1,2}(x)\|h^{(2)}\|
 \right\},
 \label{eq:symmax-first-derivative-envelope}\\
 |f_{p_\sigma,\tau_{p_\sigma},h}''(x)|
 &\leq
 \min\left\{U_{p_\sigma,\tau_{p_\sigma}}^{2,1}(x)\mathrm{osc}(h^{(1)}),
  U_{p_\sigma,\tau_{p_\sigma}}^{2,2}(x)\|h^{(2)}\|,
 \right.
 \notag\\
 &\hspace{2.8cm}\left.
  |M_{p_\sigma,\tau_{p_\sigma}}^{2,2}(x)|\|h^{(2)}\|
  +
  U_{p_\sigma,\tau_{p_\sigma}}^{2,3}(x)\|h^{(3)}\|
 \right\}.
 \label{eq:symmax-second-derivative-envelope}
\end{align}
Taking suprema in \(x\) in the first argument of \eqref{eq:symmax-solution-envelope} yields
\begin{align*}
 \|f_{p_\sigma, \tau_{p_\sigma}, h}\|
 &\le
 \frac{\gamma_0^{\mathrm M}}{\sigma}\mathrm{osc}(h), 
 \qquad
 \gamma_0^{\mathrm M}=0.3538352\ldots, \\
  \|f_{p_\sigma, \tau_{p_\sigma}, h}'\|
 &\le
 \frac1{\sigma^2}\mathrm{osc}(h)
\end{align*}
with maximizers at \(x=\pm1.75750\ldots\,\sigma\) for the solution factor. The first-derivative factor is approached in the tails and we have
\[\|\tau_{p_\sigma}f_{p_\sigma,\tau_{p_\sigma},h}'\| \le \mathrm{osc}(h).\]
Numerically,
\[
 \|U_{p_\sigma,\tau_{p_\sigma}}^{1,1}\|
 =
 \frac{\gamma_1^{\mathrm M}}{\sigma},
 \qquad
 \gamma_1^{\mathrm M}=0.530424\ldots,
\]
with maximizers at \(x=\pm1.98480\ldots\,\sigma\). 

{Finally, the second derivative is uniformly bounded for Lipschitz test functions. We show in Lemma~\ref{lem:symmax-U21} that
\(U_{p_\sigma,\tau_{p_\sigma}}^{2,1}\leq 1/\sigma^2\) on \(\mathbb R\), with
\(U_{p_\sigma,\tau_{p_\sigma}}^{2,1}=1/\tau_{p_\sigma}\) for \(|x|\geq\sqrt2\,\sigma\).
Therefore, since  \(1/\tau_{p_\sigma}(x)\to 1/\sigma^2\) as $|x|\to \infty$, the centred representation \eqref{eq:ker21-centered} gives
\begin{equation}
\label{eq:symmax-second-derivative-factor}
 \|f_{p_\sigma,\tau_{p_\sigma},h}''\|
 \leq
 \frac1{\sigma^2}\operatorname{osc}(h')
 \leq
 \frac2{\sigma^2}\|h^{(1)}\|.
\end{equation}
Both constants are approached in the tails and not attained. They are sharp: for \(|x|\geq\sqrt2\,\sigma\) the kernel \(K_{p_\sigma,\tau_{p_\sigma}}^{2,1}(x,\cdot)\) is non-positive, so taking \(h'=1\) on a small neighbourhood of such \(x\) and \(h'=-1\) elsewhere gives the sharpness.}

Higher-order
pointwise bounds are obtained recursively from Theorem~\ref{thm:kernel-representations-of-derivatives}; the correction terms do
not collapse as they do in the Gaussian case.

\begin{remark}
For the two-sided Maxwell law, \cite{mckeague_pekoz_swan_2019} develop a
Wasserstein approximation theory based on a regularized Stein equation,
rather than Kolmogorov factors. In the standard case \(\sigma=1\), their
regularized solution corresponds to
\(f_{p_1,\tau_{p_1},h}\),
after correcting the sign of the upper-tail integral in their
equation~(26), and their Proposition~4.5 gives
\[
 \|f_{p_1,\tau_{p_1},h}\|\leq 3\|h'\|,
 \qquad
 \|f_{p_1,\tau_{p_1},h}'\|\leq 4\|h'\|.
\]
Equations~\eqref{eq:symmax-solution-envelope} and
\eqref{eq:symmax-first-derivative-envelope} improve these constants to
\(1\) and \(0.530424\ldots\), respectively. 
\end{remark}
\end{example}

\begin{example}[Subbotin distributions]
\label{subsec:subbotin-example}

Let \(\beta>1\) and let \(Z_\beta\) have density
\begin{equation}
\label{eq:subbotin-density}
 p_\beta(x)
 =
 C_\beta\exp\left(-\frac{|x|^\beta}{\beta(\beta-1)}\right),
 \qquad
 C_\beta
 =
 \frac{\beta}{2(\beta(\beta-1))^{1/\beta}\Gamma(1/\beta)},
 \qquad x\in\mathbb R.
\end{equation}
Denote by \(P_\beta\) and \(\overline P_\beta\) the distribution and
survival functions associated with \(p_\beta\).
This family contains the standard
Gaussian law at \(\beta=2\), while
\(\beta=4\) gives the quartic density proportional to \(e^{-x^4/12}\), which
appears in the critical Curie--Weiss model; see
Section~\ref{sec:applications}. 

Subbotin distributions are centred, symmetric and log-concave. 
Their Stein kernel is 
\begin{equation}
\label{eq:subbotin-stein-kernel}
 \tau_{p_\beta}(x)
 =
 \frac{(\beta(\beta-1))^{2/\beta}}{\beta}
 \exp\left(\frac{|x|^\beta}{\beta(\beta-1)}\right)
 \Gamma\left(
  \frac2\beta,\frac{|x|^\beta}{\beta(\beta-1)}
 \right),
\end{equation}
where \(\Gamma(a,b)\) is the upper incomplete gamma function. The Stein kernel is unbounded for \(1<\beta<2\), constant for
\(\beta=2\), and bounded for \(\beta>2\). More precisely, for \(\beta\geq2\),
\begin{equation}
\label{eq:subbotin-tau-supremum}
 \|\tau_{p_\beta}\|
 =
 \tau_{p_\beta}(0)
 =
 \frac{(\beta(\beta-1))^{2/\beta}}{\beta}\Gamma\left(\frac2\beta\right).
\end{equation}
The constant weight \(w=1\) leaves the score as the zeroth-order
coefficient, that is
\begin{equation}\label{eq:subbotin-score}
 \rho_{p_\beta}(x)
 :=
 \frac{p_\beta'(x)}{p_\beta(x)}
 =
 -\frac{\operatorname{sgn}(x)|x|^{\beta-1}}{\beta-1}.
\end{equation}
We therefore study
the solution to the $w=1$ Stein equation
\begin{equation}
\label{eq:subbotin-stein-equation}
 f'(x)
 -
\frac{\operatorname{sgn}(x)|x|^{\beta-1}}{\beta-1}f(x)
 =
 h(x)-\mathbb E[h(Z_\beta)],
 \qquad\text{a.e. }x\in\mathbb R.
\end{equation} 
Since \(p_\beta\) is log-concave,
\[
 \underline m_{p_\beta,1}'\geq0,
 \qquad
 \overline m_{p_\beta,1}'\leq0,
\]
so Proposition~\ref{prop:first-derivative-factors} applies.
Moreover \cite[Prop.~6.2, Prop.~8.5]{eichelsbacher_2026} showed 
\[\underline m_{p_\beta,1}'' \ge 0 \qquad  \overline{m}''_{p_\beta,1} \ge 0,\] 
so the representations \eqref{eq:ker21} and \eqref{eq:ker23} yield sharp envelopes as the underlying kernel is negative on $\mathcal{I_P}$. More precisely, we have almost everywhere,
\begin{align}
 |f_{p_\beta,1,h}(x)|
 &\leq
 \min\left\{
 \frac{P_\beta(x)\overline{P}_\beta(x)}{p_\beta(x)}\operatorname{osc}(h),
  \tau_{p_\beta}(x)\|h^{(1)}\|
 \right\},
 \label{eq:subbotin-solution-envelope}\\
 |f_{p_\beta,1,h}'(x)|
 &\leq
 \min\left\{
  \operatorname{osc}(h),
  U_{p_\beta,1}^{1,1}(x)\|h^{(1)}\|,
  U_{p_\beta,1}^{1,2}(x)\|h^{(2)}\|
  +|\tau_{p_\beta}'(x)|\|h^{(1)}\|
 \right\},
 \label{eq:subbotin-first-derivative-envelope}\\
 |f_{p_\beta,1,h}''(x)|
 &\leq
 \min\Bigl\{
  \bigl(U_{p_\beta,1}^{2,1}(x)+1\bigr)\|h^{(1)}\|,
 \notag\\
 &\hspace{1.25cm}
  U_{p_\beta,1}^{2,2}(x)\|h^{(2)}\|
  +|\tau_{p_\beta}''(x)|\|h^{(1)}\|,
 \notag\\
 &\hspace{1.25cm}
  U_{p_\beta,1}^{2,3}(x)\|h^{(3)}\|
  +|M_{p_\beta,1}^{2,2}(x)|\|h^{(2)}\|
  +|\tau_{p_\beta}''(x)|\|h^{(1)}\|
 \Bigr\}.
 \label{eq:subbotin-second-derivative-envelope}
\end{align}
For $\beta\geq2$, taking suprema in the first envelope gives \[\|f_{p_\beta,1,h}\|
 \leq
 \frac{(\beta(\beta-1))^{1/\beta}\Gamma(1/\beta)}{2\beta}\operatorname{osc}(h),\]
and by \eqref{eq:subbotin-tau-supremum} the second envelope is uniformly bounded by
\begin{equation}
\label{eq:subbotin-lipschitz-factor}
 \|f_{p_\beta,1,h}\|
 \leq
 \frac{(\beta(\beta-1))^{2/\beta}}{\beta}\Gamma\left(\frac2\beta\right)\|h^{(1)}\|.
\end{equation}
For \(1<\beta<2\), the uniform Lipschitz factor is infinite:
\(\tau_{p_\beta}\) diverges in the tails, and the admissible
test function \(h(x)=x\) gives \(f_{p_\beta,1,h}=-\tau_{p_\beta}\).
We illustrate the corresponding pointwise kernel coefficients in
Figure~\ref{fig:subbotin-diagonal-envelopes}. 

Note that $K^{2,1}_{p_\beta,1}\le0$ and thus the first term of \eqref{eq:subbotin-second-derivative-envelope} is, for every $\beta>1$ and $x\neq0$, a sharp pointwise envelope (the supremum is not attained). The regimes differ only at the origin: for $\beta>2$, $K^{2,1}_{p_\beta,1}(0,\cdot)\equiv0$ and the envelope equals $1$, attained by $h(z)=z$; for $1<\beta<2$, $U^{2,1}_{p_\beta,1}(x)\to\infty$ as $x\to0$, so $f''$ has no uniform Lipschitz factor.
This can also be seen directly from \(h(x)=x\):
\(f_{p_\beta,1,h}''=-\tau_{p_\beta}''\), and
\(\tau_{p_\beta}''(x)\sim
\tau_{p_\beta}(0)|x|^{\beta-2}\) as \(x\to0\). The same threshold $\beta=2$ (where $V(x) := |x|^{\beta}/(\beta(\beta-1))$, the potential of $p_\beta$, stops being $C^2$ at $0$) separates the indicator results of \cite{eichelsbacher_2026}.

 \begin{remark}\label{rmk:subbotin-comparison-w-litterature}
When \(\beta=2\), these estimates reduce to the Gaussian kernel bounds of Theorem~\ref{thm:known-gaussian-stein-factors}. For the quartic density
\(p_4(x)=C_4 e^{-x^4/12}\), \eqref{eq:subbotin-lipschitz-factor}
and the envelopes \eqref{eq:subbotin-first-derivative-envelope}--\eqref{eq:subbotin-second-derivative-envelope} improve the corresponding Lipschitz bounds from 
\cite{Chatterjee_2011}. 
More precisely, we have the following sharp Stein factors for Lipschitz test function $h$
\[
\|f_{p_4,1,h}\|
\leq\frac{\sqrt{3\pi}}2,
\qquad
\|f_{p_4,1,h}'\|
\leq c^{1,1}_{p_4,1},
\qquad
\|f_{p_4,1,h}''\|
\leq c^{2,1}_{p_4,1},
\]
hold with \(c^{1,1}_{p_4,1}\approx1.02325\) and
\(c^{2,1}_{p_4,1}\approx 2.30554\) (obtained by numerical maximization).
By contrast, the bounds in \cite{Chatterjee_2011} yield
\[
\|f_{p_4,1,h}\|\leq c_2,
\qquad
\|f_{p_4,1,h}'\|
\leq\frac{1+c_2}{C_4},
\qquad
\|f_{p_4,1,h}''\|
\leq2(1+c_2),
\]
with the admissible choice
\(
c_2
=
\frac3{C_4^2}
+
\left(\frac3{C_4^2}\right)^{1/3}
\approx37.39626
\)
from their hypothesis \(\mathrm{(H2)}\). Thus our bounds represents improvements of approximatively factor $24.36$, $126.61$ and $33.31$ (respectively).
\end{remark}
\begin{figure}[p]
 \centering
 \includegraphics[width=\textwidth]{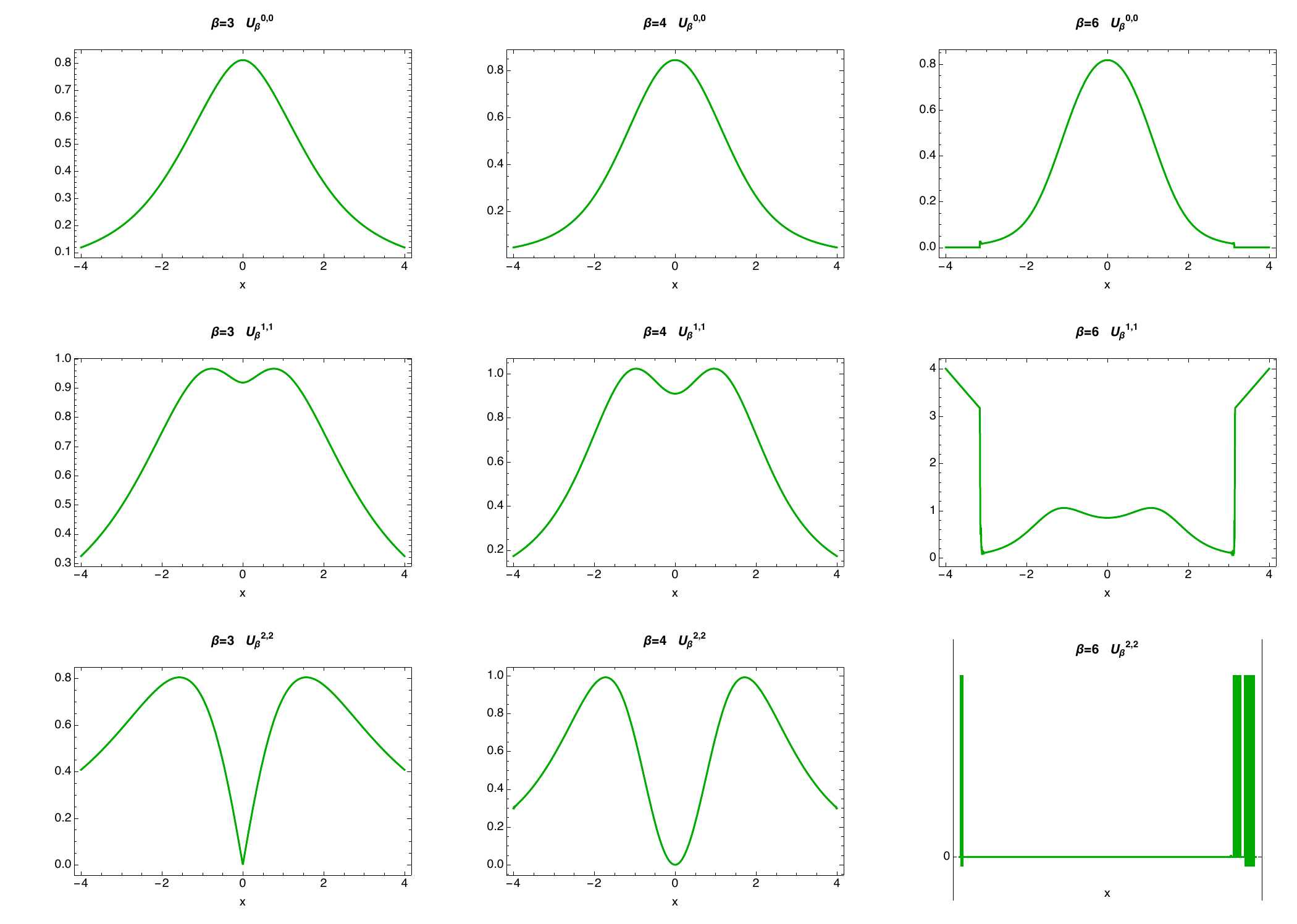}
 \caption{
 Diagonal absolute kernel means for the normalized Subbotin laws
 \eqref{eq:subbotin-density} with \(\beta=3,4,6\). The columns correspond to
 the three values of \(\beta\). The rows display, respectively,
 \(U_{p_\beta,1}^{0,0}
 =P_\beta\overline P_\beta/p_\beta\),
 \(U_{p_\beta,1}^{1,1}\), and
 \(U_{p_\beta,1}^{2,2}\), the
 coefficients appearing in
 \eqref{eq:subbotin-solution-envelope},
 \eqref{eq:subbotin-first-derivative-envelope}, and
 \eqref{eq:subbotin-second-derivative-envelope}.  They are the bounded-test
 coefficient for the normalization \(\|h\|\leq1\)
 for \(f_{p_\beta,1,h}\); the Lipschitz coefficient
 for \(f_{p_\beta,1,h}'\), and the diagonal \(h''\)-coefficient
 for \(f_{p_\beta,1,h}''\), respectively.
 The first is twice the coefficient of
 \(\operatorname{osc}(h)\) in
 \eqref{eq:subbotin-solution-envelope}.
 The third is only one part of the second-derivative bound:
 the corresponding estimate also contains
 \(|\tau_{p_\beta}''|\|h^{(1)}\|\).
 The Lipschitz coefficient for the second derivative is instead
 \(1+U_{p_\beta,1}^{2,1}\). For
 \(\beta>2\), the higher-order diagonal envelopes develop symmetric
 off-centre maxima, showing that the Gaussian simplifications no longer
 persist in the Subbotin family. As in the previous examples, the apparent blow-up in some sub-figures is due to numerical
instability rather than to the behaviour of the quantities
themselves.} 
 \label{fig:subbotin-diagonal-envelopes}
\end{figure}
\end{example}

\subsection{Half-line indicator test functions}\label{sec:kolmogorov-solutions}
Throughout this section, we will denote, for  $z\in \mathbb R$, $h_z(x)= \mathbf{1}_{(-\infty,z]}(x)$ the indicator function of the lower half-line $(-\infty,z]$. These test functions are linked to the Kolmogorov distance via the Stein transfer principle \eqref{eq:stein-transfer}. The specific form of these functions make the representations from Theorem~\ref{thm:kernel-representations-of-derivatives} explicit.  
This section therefore provides
exact envelopes at every order under
the stated assumptions for such test functions\footnote{To our knowledge,
these results have not been recorded in the literature at this
level of generality, despite requiring only elementary computations.}.

For simplicity, we strengthen
Assumption~\ref{ass:density-support} by requiring $wp$ to be
strictly positive and differentiable throughout $(l,u)$.
In particular, $\mathcal I_p=(l,u)$ and $w>0$
on this interval.
Additional smoothness will be imposed when higher derivatives
are considered. It is classical (and immediate from \eqref{eq:master-start}) that for $x,z\in (l,u)$, the solution to the Kolmogorov Stein equation is
\begin{equation}f_{p,w,h_z}(x) =
    \frac{P(x\wedge z)\overline P(x\vee z)}{w(x)p(x)}
    =
    \begin{cases}
        \overline P(z)\,\underline m_{p,w}(x),&x\leq z,\\[2pt]
        P(z)\,\overline m_{p,w}(x),&x>z.
    \end{cases}
    \label{eq:kolmo-solution}
    \end{equation}

The regularity condition on $p$ and $w$ for $f_{p,w,h}$ to have derivatives of higher order are best described through the following function
\[
    a_w(x):=-\frac{s_w(x)}{w(x)}
    =-\frac{(w(x)p(x))'}{w(x)p(x)},
    \qquad x\in(l,u).
\]

\begin{theorem}\label{thm:kolmo-rep}
    Let $z\in(l,u)$ and fix an integer $n\geq1$ and assume that
$a_w,1/w\in C^{n-1}((l,u))$. Then $f_{p,w,h_z}$ is of
class $C^n$ on $(l,z)$ and $(z,u)$, and its derivatives
through order $n$ admit finite one-sided limits at $z$.
For every $x\neq z$,
\begin{equation}
\label{eq:general-kolmogorov-all-derivatives}
    f_{p,w,h_z}^{(n)}(x)
    =
    \begin{cases}
        \overline P(z)\,\underline m_{p,w}^{(n)}(x),&x<z,\\[2pt]
        P(z)\,\overline m_{p,w}^{(n)}(x),&x>z.
    \end{cases}
\end{equation}
At $z$, the one-sided derivative limits satisfy
\begin{equation}
\label{eq:general-kolmogorov-derivative-jumps}
\begin{aligned}
    f_{p,w,h_z}^{(k)}(z+)-f_{p,w,h_z}^{(k)}(z-)
    &=
    P(z)\overline m_{p,w}^{(k)}(z)
    -\overline P(z)\underline m_{p,w}^{(k)}(z)
\end{aligned}
\end{equation}
at every order
$1\leq k\leq n$.
In particular, $f_{p,w,h_z}$ is continuous at $z$, whereas
its first derivative has the jump
\[
    f_{p,w,h_z}'(z+)-f_{p,w,h_z}'(z-)
    =-\frac{1}{w(z)}.
\]
\end{theorem}
\begin{remark}
For the standard Gaussian density Stein operator, taking
$w\equiv1$ yields $s_1(x)=-x$ and $a_1(x)=x$,
thus retrieving the results of Theorem~\ref{theo:kolmosteif}.
\end{remark}

The classical Mills ratios have well-known differentiation recurrences. Our weighted Mills ratios obey similar relations. For $n\geq0$,
whenever the required derivatives exist,
define the coefficients $A_{p,w,n}$ and $B_{p,w,n}$ recursively by
\begin{align}
    A_{p,w,0}=1&,\quad
    A_{p,w,n+1}=A_{p,w,n}'+a_wA_{p,w,n},
    \label{eq:generalized-laplace-coefficients-An}\\
    B_{p,w,0}=0&,\quad
    B_{p,w,n+1}=B_{p,w,n}'+w^{-1}A_{p,w,n}.
    \label{eq:generalized-laplace-coefficients-Bn}
\end{align}
The first coefficients are
\[
    A_{p,w,1}=a_w,\quad B_{p,w,1}=w^{-1},
    \qquad
    A_{p,w,2}=a_w'+a_w^2,\quad
    B_{p,w,2}=(w^{-1})'+a_ww^{-1},
\]
and the coefficients $A_{p,w,n}$ admit the explicit representation
\begin{equation*}
\label{eq:generalized-laplace-A-explicit}
    A_{p,w,n}(x)
    =
    w(x)p(x)
    \left(\frac{1}{wp}\right)^{(n)}(x).
\end{equation*}
Hence the following precision of Theorem~\ref{thm:kolmo-rep}.
\begin{corollary}\label{corr:kolmo-recursion}
    Under the condition of Theorem~\ref{thm:kolmo-rep}, 
    \begin{equation}
     \underline m_{p,w}^{(n)}
    =A_{p,w,n}\,\underline m_{p,w}+B_{p,w,n},
    \qquad
    \overline m_{p,w}^{(n)}
    =A_{p,w,n}\,\overline m_{p,w}-B_{p,w,n},\label{eq:weighted-mills-recursion}
\end{equation}
and the following representation of $f^{(n)}_{p,w,h}$ holds
    \begin{equation*}
    f_{p,w,h_z}^{(n)}
=A_{p,w,n}f_{p,w,h_z}+B_{p,w,n}(h_z-P(z))
\end{equation*}
away from $z$. Moreover
\begin{equation*}
    f_{p,w,h_z}^{(k)}(z+)-f_{p,w,h_z}^{(k)}(z-)=-B_{p,w,k}(z)
\end{equation*}
\end{corollary}

The explicit representation
\eqref{eq:general-kolmogorov-all-derivatives} also yields
the exact pointwise envelopes and, upon
taking the supremum over $x$, the corresponding
Kolmogorov Stein factors for derivatives away from the threshold.

\begin{corollary}
\label{cor:general-kolmogorov-envelopes}
Under the assumptions of Theorem~\ref{thm:kolmo-rep}, for every $x\in(l,u)$,
\begin{align}
 \sup_{\substack{z\in(l,u)\\z\neq x}}
    \left|f_{p,w,h_z}^{(n)}(x)\right|
    =
    \max\Bigl\{&
        P(x)\bigl|\overline m_{p,w}^{(n)}(x)\bigr|,
        \overline P(x)\bigl|\underline m_{p,w}^{(n)}(x)\bigr|
    \Bigr\}\label{eq:general-kolmogorov-exact-envelope-mills} \\
\begin{split}
    =
    \max\Bigg\{&
    P(x)
    \left|
        \frac{\overline P(x)A_{p,w,n}(x)}{w(x)p(x)}
        -B_{p,w,n}(x)
    \right|,
    \\
    &
    \overline P(x)
    \left|
        \frac{P(x)A_{p,w,n}(x)}{w(x)p(x)}
        +B_{p,w,n}(x)
    \right|
    \Bigg\}.
\end{split}
\label{eq:general-kolmogorov-exact-envelope}
\end{align}
The formulas also hold for $n=0$,
and the supremum may then be taken over all $z\in(l,u)$.
At order zero, it is attained at $z=x$.
\end{corollary}
\begin{example}\label{ex:explicit-kolmo-envelope}
For ease of future reference, we record the pointwise Kolmogorov
envelopes of orders $0$, $1$, and $2$, first in compact form and
then in an expanded form suitable for direct application.
At each order, we impose the corresponding
smoothness assumptions of
Corollary~\ref{cor:general-kolmogorov-envelopes}.

At order $0$,
\begin{equation}
\label{eq:general-kolmogorov-zero-envelope}
    \sup_{z\in(l,u)}
    \left|f_{p,w,h_z}(x)\right|
    =
    P(x)\,\overline m_{p,w}(x)
    =
    \overline P(x)\,\underline m_{p,w}(x)
    =
    \frac{P(x)\overline P(x)}{w(x)p(x)}.
\end{equation}

At order $1$,
\begin{equation}
\begin{split}
    \sup_{\substack{z\in(l,u)\\z\neq x}}
    \left|f_{p,w,h_z}'(x)\right|
    &=
    \max\Bigl\{
        P(x)\bigl|\overline m_{p,w}'(x)\bigr|,
        \overline P(x)\bigl|\underline m_{p,w}'(x)\bigr|
    \Bigr\} \\
    &=
    \max\Bigg\{
    P(x)
    \left|
        -\frac{1}{w(x)}
        -\frac{\overline P(x)(wp)'(x)}
              {(w(x)p(x))^2}
    \right|,
    \\
    &\hspace{2cm}
    \overline P(x)
    \left|
        \frac{1}{w(x)}
        -\frac{P(x)(wp)'(x)}
              {(w(x)p(x))^2}
    \right|
    \Bigg\}.
    \label{eq:general-kolmogorov-first-envelope}
\end{split}
\end{equation}

At order $2$,
\begin{equation}
\begin{split}
    \sup_{\substack{z\in(l,u)\\z\neq x}}
    \left|f_{p,w,h_z}''(x)\right|
    &=
    \max\Bigl\{
        P(x)\bigl|\overline m_{p,w}''(x)\bigr|,
        \overline P(x)\bigl|\underline m_{p,w}''(x)\bigr|
    \Bigr\}
    \\
    &=
    \max\Bigg\{
    P(x)
    \Bigg|
        -\frac{p'(x)}{w(x)p(x)}
        +\frac{2(wp)'(x)}{w(x)^2p(x)}
    \\
    &\hspace{2cm}
        +\overline P(x)
        \frac{
            2\bigl((wp)'(x)\bigr)^2
            -w(x)p(x)(wp)''(x)
        }{
            \bigl(w(x)p(x)\bigr)^3
        }
    \Bigg|,
    \\
    &\hspace{1.6cm}
    \overline P(x)
    \Bigg|
        \frac{p'(x)}{w(x)p(x)}
        -\frac{2(wp)'(x)}{w(x)^2p(x)}
    \\
    &\hspace{2cm}
        +P(x)
        \frac{
            2\bigl((wp)'(x)\bigr)^2
            -w(x)p(x)(wp)''(x)
        }{
            \bigl(w(x)p(x)\bigr)^3
        }
    \Bigg|
    \Bigg\}.
    \label{eq:general-kolmogorov-second-envelope}
\end{split}
\end{equation}
\end{example}

For any weight satisfying the assumptions of this section, taking
the supremum of the exact envelope over $x\in(l,u)$ gives the
optimal Kolmogorov Stein factor
for the piecewise derivatives.
More precisely, under the regularity assumptions for order $n\geq1$,
    \[
    \sup_{z\in(l,u)}
    \operatorname*{ess\,sup}_{x\in(l,u)\setminus\{z\}}
    \left|f_{p,w,h_z}^{(n)}(x)\right|
    =
    \sup_{x\in(l,u)}
    \sup_{\substack{z\in(l,u)\\z\neq x}}
    \left|f_{p,w,h_z}^{(n)}(x)\right|.
\]
Indeed, for each $z$, the derivative is continuous on either
side of the threshold, so its supremum away from $z$ agrees
with its essential supremum. The same conclusion holds at
order zero, with no exclusion of $z=x$.
For $n\geq2$, these essential suprema concern the classical
derivatives calculated separately on $(l,z)$ and $(z,u)$.
They do not require the derivatives to extend continuously
across $z$. These piecewise factors therefore differ from
the derivative norms of Notation~\ref{not:derivative-norms},
which also impose regularity across the threshold. The Kolmogorov envelopes \eqref{eq:general-kolmogorov-zero-envelope}--\eqref{eq:general-kolmogorov-second-envelope} were illustrated in Figure~\ref{fig:kolmogorov-factors} for the Gaussian distribution.

\begin{remark}
        Figures~\ref{fig:ip-kolmogorov-gamma}--\ref{fig:ip-kolmogorov-student} give the first Kolmogorov  envelopes for the Gamma, Beta, and Student distributions, weighted by $w=\tau_p$ if needed. Indeed the envelopes of Example~\ref{ex:explicit-kolmo-envelope} can be unbounded over the support (if the weight $w$ vanishes at the boundary); in this case  $w f_{p,w,h}^{(n)}$ may be bounded. 
    
\end{remark}

\begin{figure}[p]
 \centering
 \includegraphics[width=\textwidth]
 {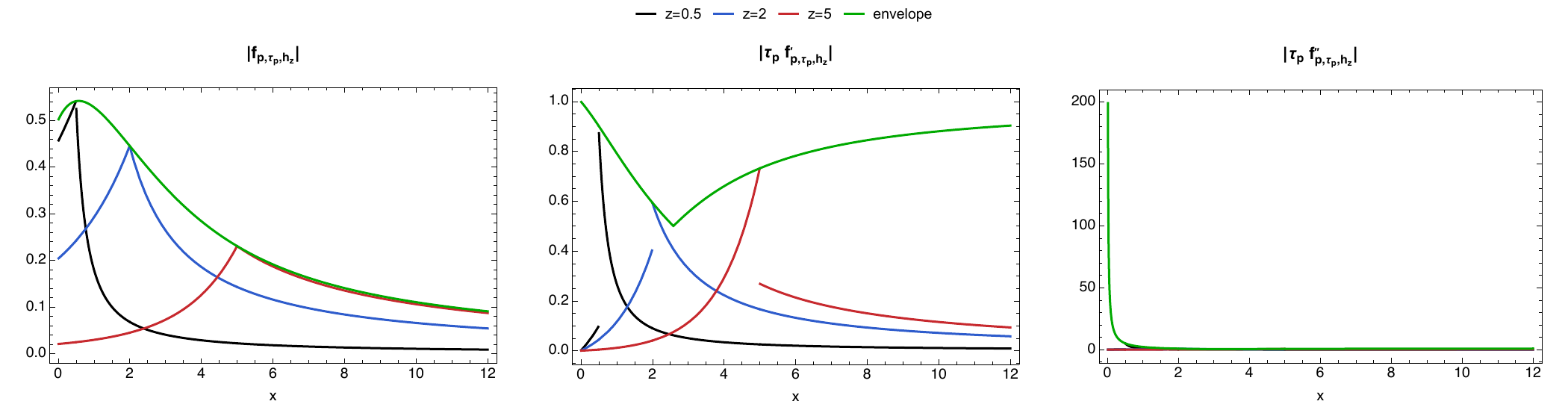}
 \caption{
 Kolmogorov Stein solutions for the Gamma distribution with
 shape \(2\) and scale \(1\). 
 The coloured curves correspond to \(z=0.5,2,5\);
 the green curves are the exact envelopes from
 \eqref{eq:general-kolmogorov-exact-envelope}  for
 \(|f_{p,\tau_p,h_z}|\),
 \(|\tau_p f_{p,\tau_p,h_z}'|\), and
 \(|\tau_p f_{p,\tau_p,h_z}''|\), with $\tau_p(x)=x$.
 The second-derivative envelope is unbounded near the origin,
 even after multiplication by \(\tau_p\).}
 \label{fig:ip-kolmogorov-gamma}
\end{figure}

\begin{figure}[p]
 \centering
 \includegraphics[width=\textwidth]
 {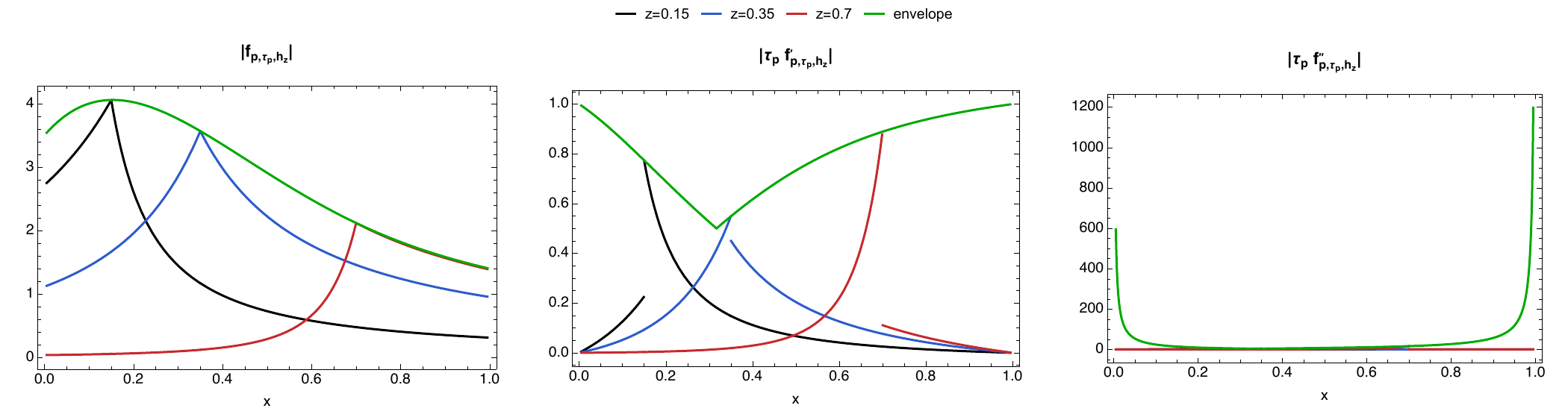}
 \caption{
 Kolmogorov Stein solutions for the Beta distribution with
 parameters \(\alpha=2\) and \(\beta=5\). 
 The coloured curves correspond to \(z=0.15,0.35,0.7\);
 the green curves are the exact envelopes from
 \eqref{eq:general-kolmogorov-exact-envelope}  for
 \(|f_{p,\tau_p,h_z}|\),
 \(|\tau_p f_{p,\tau_p,h_z}'|\), and
 \(|\tau_p f_{p,\tau_p,h_z}''|\).
 Here \(\tau_p(x)=x(1-x)/7\).
 Multiplying by \(\tau_p\) controls the first-derivative
 envelope, but the second-derivative envelope remains
 unbounded at the boundary.}
 \label{fig:ip-kolmogorov-beta}
\end{figure}

\begin{figure}[p]
 \centering
 \includegraphics[width=\textwidth]
 {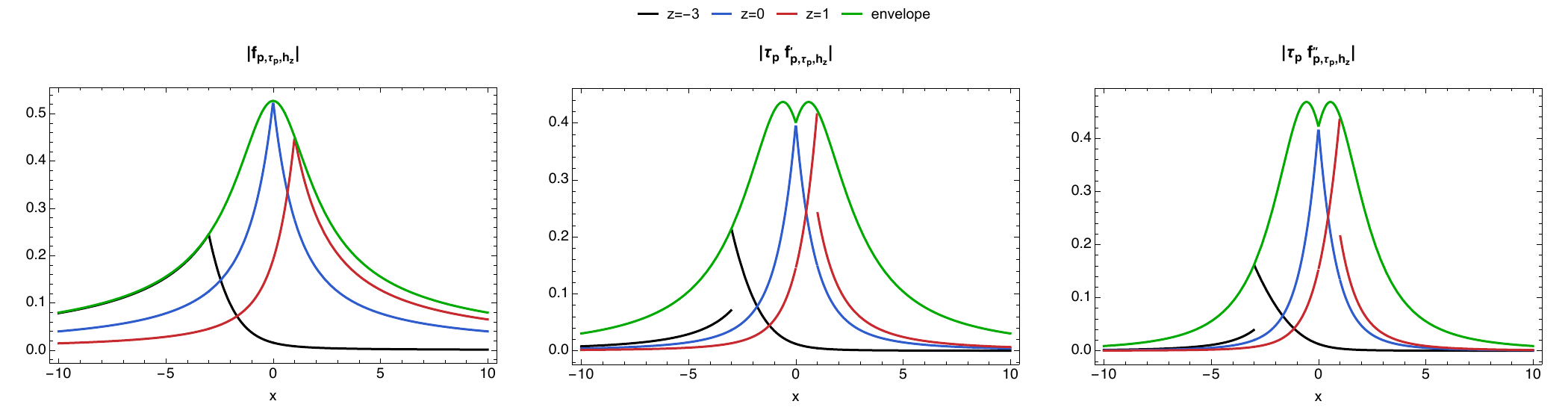}
 \caption{
 Kolmogorov Stein solutions for the Student distribution with
 \(\nu=5\) degrees of freedom. 
 The coloured curves correspond to \(z=-3,0,1\);
 the green curves are the exact envelopes from
 \eqref{eq:general-kolmogorov-exact-envelope} for
 \(|f_{p,\tau_p,h_z}|\),
 \(|\tau_p f_{p,\tau_p,h_z}'|\), and
 \(|\tau_p f_{p,\tau_p,h_z}''|\).
 Here the Stein kernel is
 \(\tau_p(x)=(x^2+5)/4\).}
 \label{fig:ip-kolmogorov-student}
\end{figure}

\begin{example}[Symmetrized Maxwell distribution]\label{ex:symmax-kolmo}
Recall the Maxwell distribution $p_\sigma$ introduced in Example~\ref{subsec:symmetrized-maxwell}, and consider the Kolmogorov solution. Using the continuous extension of \(\tau_{p_\sigma}p_\sigma\) at \(0\), the general Kolmogorov formula \eqref{eq:kolmo-solution}, applied with
\(w=\tau_{p_\sigma}\), gives
\begin{equation}
\label{eq:symmax-kolmogorov-solution}
 f_{p_\sigma, \tau_{p_\sigma}, h_z}(x)
 =
 \frac{
  P^\sigma(x\wedge z)\overline P^\sigma(x\vee z)
 }{
  \tau_{p_\sigma}(x)p_\sigma(x)
 }.
\end{equation}
At \(x=0\), the denominator is understood through its
continuous extension. Using
\eqref{eq:general-kolmogorov-zero-envelope},
\eqref{eq:general-kolmogorov-first-envelope}, and
\eqref{eq:general-kolmogorov-second-envelope}, and then symmetry, the
first three exact envelopes are
\begin{align}
 \sup_z |f_{p_\sigma, \tau_{p_\sigma}, h_z}(x)|
 &=
 \frac{P^\sigma\overline P^\sigma}
      {\sigma (y^2+2)\varphi(y)},
 \label{eq:symmax-kolmogorov-zero-envelope}\\
 \sup_{z\neq x}|f_{p_\sigma, \tau_{p_\sigma}, h_z}'(x)|
 &=
 \frac{y^2}{\sigma^2(y^2+2)}
 \max\left\{
 P^\sigma\left(
 1-\frac{y\overline P^\sigma}{(y^2+2)\varphi(y)}
 \right),
 \right.
 \notag\\
 &\hspace{3.1cm}\left.
 \overline P^\sigma\left(
 1+\frac{yP^\sigma}{(y^2+2)\varphi(y)}
 \right)
 \right\},
 \label{eq:symmax-kolmogorov-first-envelope}\\
 \sup_{z\neq x}|f_{p_\sigma, \tau_{p_\sigma}, h_z}''(x)|
 &=
 \frac{|y|}{\sigma^3(y^2+2)^2}
 \max\left\{
 P^\sigma
 \left|
 \frac{y(y^4+y^2+6)\overline P^\sigma}
      {(y^2+2)\varphi(y)}
 -(y^4+4)
 \right|,
 \right.
 \notag\\
 &\hspace{2.4cm}\left.
 \overline P^\sigma
 \left|
 \frac{y(y^4+y^2+6)P^\sigma}
      {(y^2+2)\varphi(y)}
 +(y^4+4)
 \right|
 \right\}.
 \label{eq:symmax-kolmogorov-second-envelope}
\end{align}
The derivatives are understood away
from the threshold.
The absolute value has disappeared from the first-derivative
envelope because
\[
 y\overline P^\sigma\leq (y^2+2)\varphi(y),
\]
which follows from
\(\tau_{p_\sigma}(x)p_\sigma(x)=\int_x^\infty t p_\sigma(t)\,dt\) for 
\(x>0\). The other bracket is nonnegative by symmetry.
Similarly as Example~\ref{subsec:symmetrized-maxwell}, taking suprema in \(x\) yields
\[
 \sup_z\|f_{p_\sigma, \tau_{p_\sigma}, h_z}\|
 =
 \frac{\gamma_0^{\mathrm M}}{\sigma},
 \qquad
 \gamma_0^{\mathrm M}=0.3538352\ldots,
\]
with maximizers at \(x=\pm1.75750\ldots\,\sigma\), and
\[
 \sup_z\|f_{p_\sigma, \tau_{p_\sigma}, h_z}'\|
 =
 \frac1{\sigma^2},
 \qquad
 \sup_z\|f_{p_\sigma, \tau_{p_\sigma}, h_z}''\|
 =
 \infty.
\]
The first-derivative factor is approached in the tails, while the
second-derivative envelope is unbounded. Figure~\ref{fig:symmax-kolmogorov-solutions}
illustrates these envelopes together with the corresponding solutions.
\begin{remark}
    To the best of our knowledge, these exact Kolmogorov envelopes and
factors in
\eqref{eq:symmax-kolmogorov-zero-envelope}--\eqref{eq:symmax-kolmogorov-second-envelope}
were not previously recorded in the literature.
\end{remark}

    \begin{figure}
 \centering
 \includegraphics[width=\textwidth]{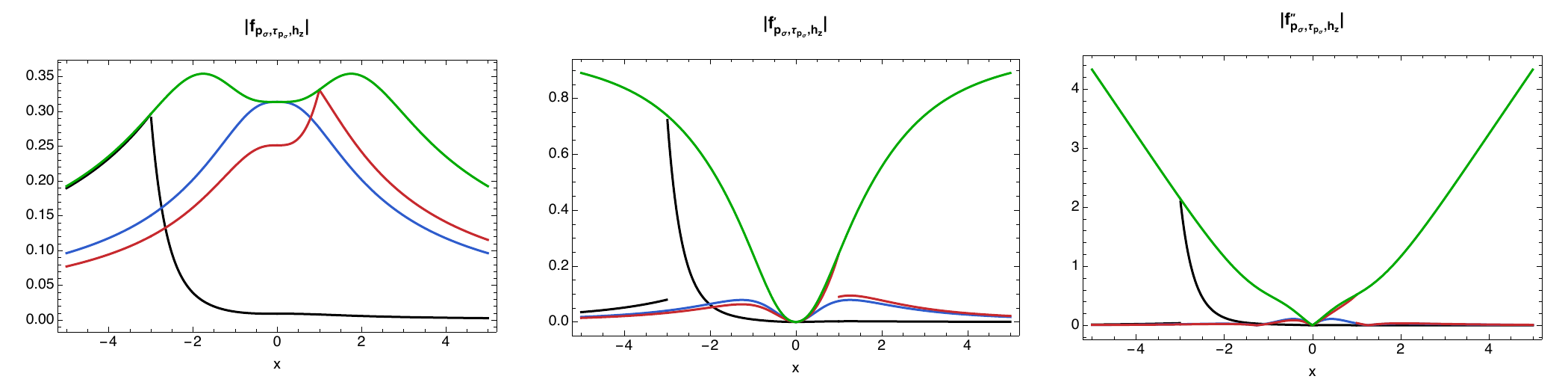}
 \caption{
 Kolmogorov Stein solutions and their first two derivatives for the
 symmetrized Maxwell law with \(\sigma=1\), plotted for \(z=-3\text{ (black)},0\text{ (blue)},1\text{ (red)}\).
 The coloured curves are
 \(|f_{p_\sigma, \tau_{p_\sigma}, h_z}|\), \(|f_{p_\sigma, \tau_{p_\sigma}, h_z}'|\), and
 \(|f_{p_\sigma, \tau_{p_\sigma}, h_z}''|\), while the green curves are the exact Kolmogorov
 envelopes
 \eqref{eq:symmax-kolmogorov-zero-envelope},
 \eqref{eq:symmax-kolmogorov-first-envelope}, and
 \eqref{eq:symmax-kolmogorov-second-envelope}. 
 The order-\(0\) envelope is bounded and has symmetric off-centre
 maxima, the order-\(1\) envelope tends to the sharp factor
 \(1/\sigma^2\) in the tails, and the order-\(2\) envelope is unbounded.}
 \label{fig:symmax-kolmogorov-solutions}
\end{figure}

\end{example}

\begin{example}[Subbotin distributions]\label{ex:subbotin-kolmo}

Let \(\beta>1\) and recall the family of densities $p_\beta$ and notations introduced in Example~\ref{subsec:subbotin-example}. Denote \(Z_\beta\) a random variable having density $p_\beta$. 
The representation \eqref{eq:kolmo-solution} gives
\[
 f_{p_\beta, 1, h_z}(x)
 =
 \frac{P_\beta(x\wedge z)\overline P_\beta(x\vee z)}{p_\beta(x)}.
\]
Using
\eqref{eq:general-kolmogorov-zero-envelope},
\eqref{eq:general-kolmogorov-first-envelope}, and
\eqref{eq:general-kolmogorov-second-envelope}, the envelopes are
\begin{align}
 \sup_z |f_{p_\beta, 1, h_z}(x)|
 &=
 \frac{P_\beta(|x|)\overline P_\beta(|x|)}{p_\beta(|x|)},
 \label{eq:subbotin-kolmogorov-zero-envelope}\\
 \sup_{{z\neq x}}|f_{p_\beta, 1, h_z}'(x)|
 &=
 \max\left\{
 P_\beta(|x|)
 \left(
 1-
 \frac{\overline P_\beta(|x|)|x|^{\beta-1}}
      {(\beta-1)p_\beta(|x|)}
 \right),
 \right.
 \notag\\
 &\hspace{2.7cm}\left.
 \overline P_\beta(|x|)
 \left(
 1+
 \frac{P_\beta(|x|)|x|^{\beta-1}}
      {(\beta-1)p_\beta(|x|)}
 \right)
 \right\},
 \label{eq:subbotin-kolmogorov-first-envelope}\\
 \sup_{z\neq x}|f_{p_\beta, 1, h_z}''(x)|
 &=
 \max\left\{
 P_\beta(|x|)
 \left|
 \frac{\overline P_\beta(|x|)}{p_\beta(|x|)}
 \left(
 |x|^{\beta-2}
 +
 \frac{|x|^{2\beta-2}}{(\beta-1)^2}
 \right)
 -
 \frac{|x|^{\beta-1}}{\beta-1}
 \right|,
 \right.
 \notag\\
 &\hspace{1.2cm}\left.
 \overline P_\beta(|x|)
 \left[
 \frac{P_\beta(|x|)}{p_\beta(|x|)}
 \left(
 |x|^{\beta-2}
 +
 \frac{|x|^{2\beta-2}}{(\beta-1)^2}
 \right)
 +
 \frac{|x|^{\beta-1}}{\beta-1}
 \right]
 \right\}.
 \label{eq:subbotin-kolmogorov-second-envelope}
\end{align}
For \(1<\beta<2\), the second-derivative formula is
asserted only for \(x\neq0\). For \(\beta\geq2\), it also holds at
\(x=0\), with \(|x|^{\beta-2}\) interpreted as \(1\) when \(\beta=2\).
When $\beta \ge 2$, taking suprema in \(x\)  yields the sharp factors
\begin{equation}\label{eq:subbotin-kolmogorov-Stein-factor}
    \sup_z\|f_{p_\beta, 1, h_z}\|
 =
 \frac{1}{4p_\beta(0)}
 =
 \frac{(\beta(\beta-1))^{1/\beta}\Gamma(1/\beta)}{2\beta},
 \qquad
 \sup_z\|f_{p_\beta, 1, h_z}'\|
 =
 1,
\end{equation}
whereas
\[
 \sup_z\|f_{p_\beta, 1, h_z}''\|=\infty.
\]
Here the piecewise second-derivative envelope itself
is unbounded, independently of the jump at the threshold. A weighted envelope for the solution can be given by seeing that 
\[
\frac{P_\beta \overline{P}_\beta}{p_\beta}\le \frac{\max\{P_\beta ,\overline{P}_\beta\}}{|\rho_\beta|}.
\]
Indeed, at $x=0$ the score vanishes and the bound is trivial. By symmetry, consider $x>0$. Then
$P_\beta(x)\geq\overline P_\beta(x)$, and after division by $P_\beta(x)>0$
it suffices to show \[|\rho_\beta(x)|\,\overline P_\beta(x)\leq p_\beta(x),\]
where $|\rho_\beta(t)|=t^{\beta-1}/(\beta-1)$ for $t>0$. Since $\beta>1$,
the map $t\mapsto t^{\beta-1}$ is increasing and we obtain
\[
 |\rho_\beta(x)|\,\overline P_\beta(x)
 =\int_x^\infty|\rho_\beta(x)|\,p_\beta(t)\,dt
 \leq\int_x^\infty|\rho_\beta(t)|\,p_\beta(t)\,dt
 =-\int_x^\infty p_\beta'(t)\,dt
 =p_\beta(x).
\]
Multiplying by $P_\beta(x)/\bigl(|\rho_\beta(x)|\,p_\beta(x)\bigr)$ gives the claim. Thus
\begin{equation}\label{eq:subbotin-score-weighted-kolmogorov-factor}
    \sup_z |\rho_{p_\beta}(x) f_{p_\beta, 1, h_z}(x)|
 \le 
 1.
\end{equation}
We illustrate the corresponding pointwise solutions and envelopes in
Figure~\ref{fig:subbotin-kolmogorov-solutions}.
\begin{remark}
For \(\beta=2\), equations
\eqref{eq:subbotin-kolmogorov-Stein-factor} and
\eqref{eq:subbotin-score-weighted-kolmogorov-factor} recover the Gaussian
Kolmogorov factors of
Theorem~\ref{theo:kolmosteif}. The even-power subfamily
\(\beta=2k\) overlaps with the class treated by
\cite{eichelsbacher_lowe_2010}, who consider densities proportional to
\(\exp\{-a_kx^{2k}\}\). On this common subfamily, their zeroth-order
Kolmogorov constant is larger than ours by a factor \(2\), whereas their
first-derivative factor agrees with ours and is sharp. In the quartic case
\(p_4(x)=C_4 e^{-x^4/12}\),
the bounded-test constants from Lemmas~4.1 and 4.2
of \cite{Chatterjee_2011}, specialized to indicators,
\[
\|f_{p_4,1,h_z}\|\leq\frac{2}{C_4},
\qquad
\|\rho_{p_4}f_{p_4,1,h_z}\|\leq2,
\qquad
\|f_{p_4,1,h_z}'\|\leq4,
\]
are improved by the exact indicator calculus by factors of
\(8\), \(2\), and \(4\), respectively.

Equations \eqref{eq:subbotin-kolmogorov-Stein-factor} and
\eqref{eq:subbotin-score-weighted-kolmogorov-factor} were obtained independently by a Mills-ratio and one-crossing argument valid for general symmetric potentials in \cite[Thm.~6.4 and \S8]{eichelsbacher_2026}. For $1<\beta<2$ the supremum of $P_\beta\overline P_\beta/p_\beta$ is attained off the origin and the exact zeroth-order factor is given in \cite[Prop.~8.5]{eichelsbacher_2026}; the first-order factors are unchanged.
\end{remark}
\begin{figure}[p]
 \centering
 \includegraphics[width=\textwidth]{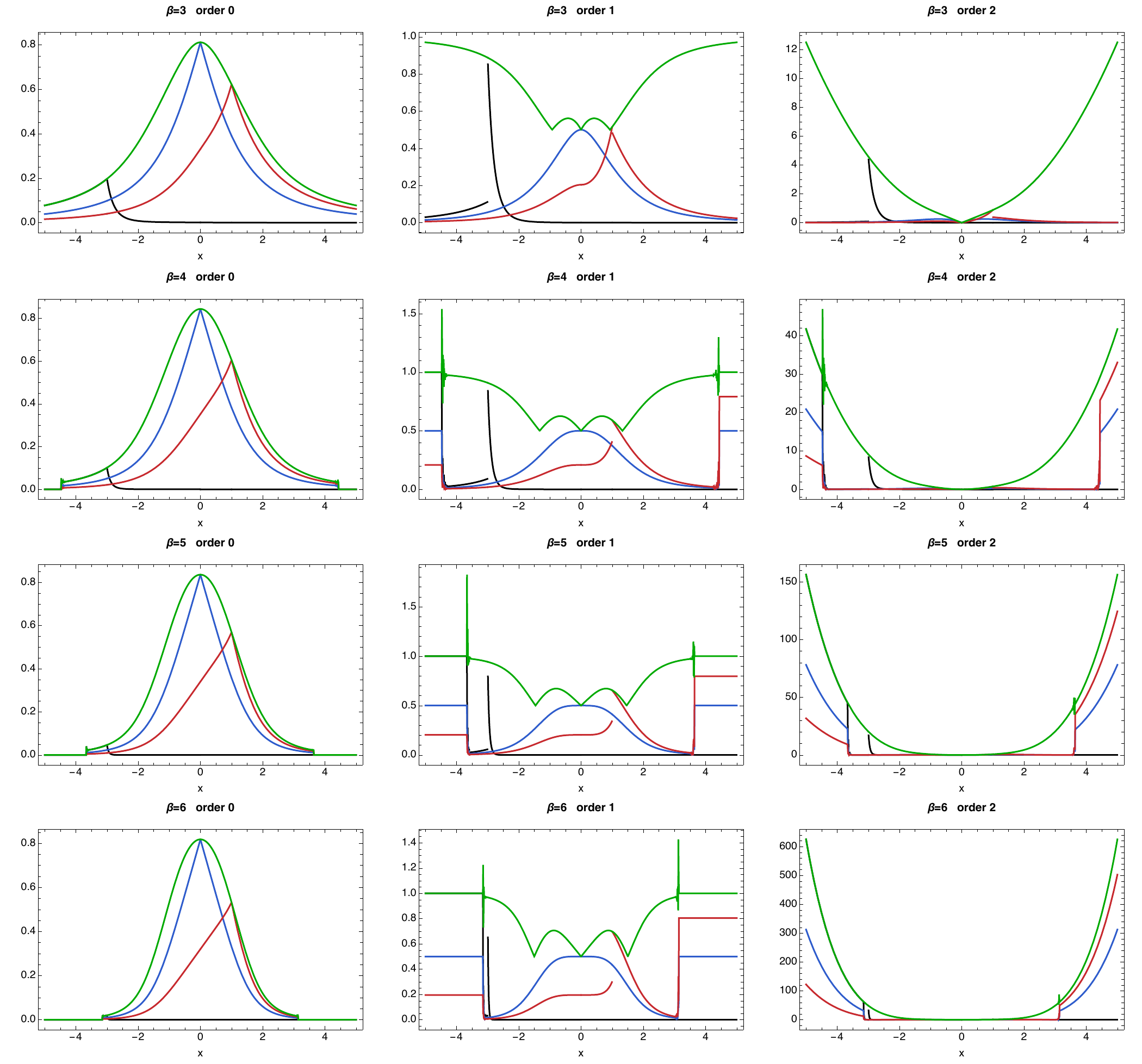}
 \caption{
 Kolmogorov Stein solutions and their first two derivatives for the
 Subbotin laws with \(\beta=3,4,5,6\), plotted for \(z=-3,0,1\). In each row,
 the coloured curves are
 \(|f_{p_\beta, 1, h_z}|\), \(|f_{p_\beta, 1, h_z}'|\), and \(|f_{p_\beta, 1, h_z}''|\), while the green
 curves are the exact Kolmogorov envelopes
 \eqref{eq:subbotin-kolmogorov-zero-envelope},
 \eqref{eq:subbotin-kolmogorov-first-envelope}, and
 \eqref{eq:subbotin-kolmogorov-second-envelope}. The order-\(0\) envelope
 is symmetric and maximized at the origin. The order-\(1\) envelope
 equals \(1/2\) at the origin, develops two symmetric off-centre local
 maxima, and tends to the sharp factor \(1\) in the tails. The order-\(2\)
 envelope is unbounded, reflecting the absence of a uniform Kolmogorov
 factor for \(f_{p_\beta, 1, h_z}''\). We note that the apparent blow-up in some sub-figures is due to numerical
instability rather than to the behaviour of the quantities
themselves.  }
 \label{fig:subbotin-kolmogorov-solutions}
\end{figure}

\end{example}

As a final remark, we note that  the factor \eqref{eq:subbotin-score-weighted-kolmogorov-factor} is not specific to the Subbotin family and comes from a more general property of log-concave distributions.

\begin{proposition}
\label{prop:score-weighted-kolmogorov-factor}
Under the assumptions of this section, suppose in addition that
$p$ is differentiable and log-concave on $(l,u)$, and write
$\rho_p=p'/p$ for its score function. At every $x\in(l,u)$
such that $\rho_p(x)\neq0$,
\[
    \sup_{z\in(l,u)}|f_{p,w,h_z}(x)|
    \leq
    \begin{cases}
        \displaystyle
        \frac{\overline P(x)}{w(x)\rho_p(x)},
        &\rho_p(x)>0,\\[6pt]
        \displaystyle
        \frac{P(x)}{w(x)|\rho_p(x)|},
        &\rho_p(x)<0.
    \end{cases}
\]
In particular,
\[
    \sup_{z\in(l,u)}|f_{p,w,h_z}(x)|
    \leq
    \frac{\max\{P(x),\overline P(x)\}}
         {w(x)|\rho_p(x)|},
    \qquad \rho_p(x)\neq0,
\]
and, for every $x\in(l,u)$,
\[
    \sup_{z\in(l,u)}
    |f_{p,1,h_z}(x)\rho_p(x)|
    \leq1.
\]
Moreover, if $w=1$, condition\eqref{eq:cond-first-order-signs-mills} holds and 
\[
\sup_{\substack{z\in(l,u)\\z\neq x}}
    \left|f_{p,1,h_z}'(x)\right|
    =
    \max\Bigg\{
    P(x)
    \left(
        -1
        -\rho_p(x)\frac{\overline P(x)}
              {p(x)}
    \right),
    \overline P(x)
    \left(
        1
        -\rho_p(x)\frac{P(x)}
              {p(x)}
    \right)
    \Bigg\}.
\]
\end{proposition}

\begin{remark}
The bounds on the solution in
Proposition~\ref{prop:score-weighted-kolmogorov-factor}
are generally not exact pointwise envelopes. The constant $1$
in the score-weighted bound is nevertheless optimal. Indeed,
for the exponential density $p(x)=\lambda e^{-\lambda x}$,
$x>0$, with $\lambda>0$, we have $\rho_p(x)=-\lambda$ and
\[
    \sup_{z>0}
    \left|f_{p,1,h_z}(x)\rho_p(x)\right|
    =1-e^{-\lambda x}.
\]
This quantity is strictly less than $1$ at every finite $x$,
but tends to $1$ as $x\to\infty$. Thus the constant is sharp,
although equality is never attained at a finite point.
\end{remark}

\subsection{The integrated Pearson family}
\label{sec:integrated-pearson}

 Proofs for this section are deferred to Appendix \ref{sn:proofsectionsecin}.

\begin{definition} 
\label{def:integrated-pearson}
Let \(Z\) be absolutely continuous, with finite mean \(\mu\),
density \(p\), and Stein kernel \(\tau_p\).
We say that \(Z\) belongs to the integrated Pearson family if
\[
 \tau_p(x)=\kappa_2x^2+\kappa_1x+\kappa_0,
 \qquad x\in\mathcal I_p.
\]
We then write
\(Z\sim\mathrm{IP}(\mu;\kappa_2,\kappa_1,\kappa_0)\).
\end{definition}

\begin{example}
Up to affine changes of variable, the integrated Pearson family
contains the Gaussian, Gamma, Beta, and Student distributions
(the latter with more than one degree of freedom); see
\cite{Integrated_pearson}. It is a proper subclass of the ordinary Pearson family. 
For example, the half-normal density satisfies $p'(x)/p(x)=-x$
on $(0,\infty)$, but its Stein kernel is not quadratic.
\end{example}

The quadratic form of the Stein kernel drives an extension of the general identities 
\[\underline P_2 = - s P +\tau_p p \qquad \text{ and } \qquad \overline{P}_2=s\overline{P}+\tau_p p\]
for $\underline P_k$ and $\overline{P}_k$, whenever they are finite. The latter condition is easy to check in the integrated Pearson setting. Set 
\[
 s(x):=s_{\tau_p}(x)=\mu-x,
 \qquad
 \lambda_k(x):=s(x)+(k-1)\tau_p'(x),
 \qquad
 q_k:=k\bigl(1-(k-1)\kappa_2\bigr).
\]

\begin{definition}  
\label{def:admissible}
Let \(Z\sim\mathrm{IP}\), and
let \(k\geq0\) be an integer.
We say that \(k\) is \emph{admissible} for \(Z\) if 
\(\underline P_{k+1}(x)\) and
 \(\overline P_{k+1}(x)\) are finite for some, equivalently
 every, \(x\in\mathcal I_p\).

\end{definition}

Admissibility is simply a moment condition. It also ensures
that the boundary terms in the forthcoming integrations by
parts vanish.

\begin{lemma} 
\label{lem:admissible}
Let \(Z\sim\mathrm{IP}\) and let \(k\geq1\).
The following are equivalent:
\begin{enumerate}
 \item[(i)] \(q_j>0\) for every \(1\leq j\leq k\);
 \item[(ii)] \(\mathbb E|Z|^k<\infty\);
 \item[(iii)] k is admissible
\end{enumerate}
Moreover, if \(k\) is admissible, then
\[
 \lim_{x\downarrow l}\lambda_j(x)\underline P_j(x)
 =
 \lim_{x\downarrow l}\tau_p(x)\underline P_{j-1}(x)
 =
 0
\]
for every \(1\leq j\leq k\), and symmetrically at \(u\)
for the upper integrated tails.
\end{lemma}

\begin{remark}
\label{rem:admissible-1/kappa2}
If \(\kappa_2\leq0\), as for the Gaussian, Gamma, and Beta
distributions, all orders are admissible.
If \(\kappa_2>0\), the order \(k\) is admissible if and only if
\(k<1+1/\kappa_2\), that is,
\[
 k\in\{0,1,\ldots,\lceil1/\kappa_2\rceil\}.
\]
In particular, orders zero and one are always admissible.
\end{remark}

\begin{lemma} 
\label{lem:integrated-pearson-tail-recursions}
Let \(Z\sim\mathrm{IP}\) and let \(k\geq1\) be admissible.
Then, for \(x\in(l,u)\),
\begin{align}
 q_k\underline P_{k+1}(x)
 &=
 -\lambda_k(x)\underline P_k(x)
 +
 \tau_p(x)\underline P_{k-1}(x),\\
 q_k\overline P_{k+1}(x)
 &=
 \lambda_k(x)\overline P_k(x)
 +
 \tau_p(x)\overline P_{k-1}(x).
\end{align}
\end{lemma}

These recursions yield the cancellation identities \eqref{eq:closed-three-neighbour-condition} when $w=\tau_p$.
The representations then follow directly from
Corollary~\ref{cor:closed-three-neighbour-calculus}.
As before, each representation is understood under the
corresponding regularity, integrability, and boundary
assumptions on the test function.

\begin{theorem} 
\label{thm:integrated-pearson-kernel-representations}
Let \(Z\sim\mathrm{IP}\) and let \(n\geq1\) be admissible.
Then
\[
 M_{p,\tau_p}^{n,n-1}=-\frac1{\tau_p},
 \qquad
 M_{p,\tau_p}^{n,n}=0.
\]
Consequently,
\begin{align}
 f_{p,\tau_p,h}^{(n)}(x)
 &=
 \mathbb E\left[
 K_{p,\tau_p}^{n,n-1}(x,Z)
 \bigl(h^{(n-1)}(Z)-h^{(n-1)}(x)\bigr)
 \right],
 \label{eq:IP-representation-n-n-1}\\
 f_{p,\tau_p,h}^{(n)}(x)
 &=
 \mathbb E\left[
 K_{p,\tau_p}^{n,n}(x,Z)h^{(n)}(Z)
 \right].
 \label{eq:IP-representation-n-n}
\end{align}
If \(n+1\) is also admissible, then
\begin{equation}
\label{eq:IP-representation-n-n+1}
 f_{p,\tau_p,h}^{(n)}(x)
 =
 \mathbb E \left[ K_{p,\tau_p}^{n,n+1}(x,Z)h^{(n+1)}(Z) \right].
\end{equation}
\end{theorem}
 \begin{remark}\label{rk:IP-other-weights}
Other weights also allow some simplifications.
Throughout this remark, \(n\geq1\) is admissible and the
corresponding regularity, integrability, and boundary assumptions
are understood.
If \((\tau_p/w)''\) is constant, then
\[
 M_{p,w}^{n,n-1}=-\frac1w+C_{n,w},
\]
where \(C_{n,w}\) is the constant given in
Proposition~\ref{prop:summary-M-tau-over-w}.
The lower-neighbour representation becomes
\[
 f_{p,w,h}^{(n)}(x)
 =
 \mathbb E\left[
 K_{p,w}^{n,n-1}(x,Z)
 \bigl(h^{(n-1)}(Z)-h^{(n-1)}(x)\bigr)
 \right]
 +
 C_{n,w}h^{(n-1)}(x).
\]
If \((\tau_p/w)'\) is constant, then \(C_{n,w}=0\), and
both the lower-neighbour and diagonal representations hold
without additional terms, with \(w\) in place of \(\tau_p\).
If \(n+1\) is also admissible, the upper-neighbour representation
then has the correction
\[
 -\frac{n}{q_n}
 \left(\frac{\tau_p}{w}\right)'(x)h^{(n)}(x).
\]
Thus choosing \(w=\tau_p\), or a nonzero constant multiple of it,
removes all three corrections.
\end{remark}

Now that we have proven the one term representations, we shall pursue to bound the resulting envelopes from \eqref{eq:closed-envelopes}. That requires control over the kernel absolute means. While the low-order means are provided in
Example~\ref{ex:low-order-kernel-means}, the higher-order means and absolute means 
can also be computed explicitly in the IP case, using the following
identities.

\begin{proposition}
\label{prop:equivalence-m-tau-P}
Let \(Z\sim\mathrm{IP}\).
For every admissible \(k\geq0\),
\begin{align}
 p(x)\tau_p(x)^{k+1}
 \underline m_{p,\tau_p}^{(k)}(x)
 &=
 \left(\prod_{i=1}^kq_i\right)\underline P_{k+1}(x),
 \label{eq:IP-left-tail-mills-equivalence}\\
 p(x)\tau_p(x)^{k+1}
 \overline m_{p,\tau_p}^{(k)}(x)
 &=
 (-1)^k
 \left(\prod_{i=1}^kq_i\right)\overline P_{k+1}(x).
 \label{eq:IP-right-tail-mills-equivalence}
\end{align}
\end{proposition}
\begin{remark}
    In the Gaussian case, this equivalence was observed by \cite[Equation (2.1)--(2.2) and Lemma 2.1]{daly_2008} to prove representation \eqref{eq:gaussian-daly-rep-fn}.
\end{remark}

Since the integrated tails are nonnegative and \(q_i>0\) at
every admissible order, these identities immediately determine
the signs of the derivatives of the weighted Mills ratios.

\begin{corollary}
\label{cor:IP-Mills-monotonicity}
Let \(Z\sim\mathrm{IP}\).
For every admissible \(k\),
\[
 \underline m_{p,\tau_p}^{(k)}\geq0,
 \qquad
 (-1)^k\overline m_{p,\tau_p}^{(k)}\geq0.
\]
In particular, for the Gaussian, Gamma, and Beta distributions,
the lower and upper weighted Mills ratios are respectively
absolutely and completely monotone.
\end{corollary}

Combining these sign properties with the integrated-tail
recursions gives explicit expressions for the absolute kernel
means at the three neighbouring orders.
\begin{proposition}
\label{prop:IP-absolute-expectations}
Let \(Z\sim\mathrm{IP}\) and let \(n\geq0\) be admissible.
Then 
\[
 U_{p,\tau_p}^{n,n-1}
 =
 -M_{p,\tau_p}^{n,n-1}
 =
 \frac1{\tau_p}.
\]
Moreover
\[
 U_{p,\tau_p}^{n,n}
 =
 \frac{2}{\underline R_n+\overline R_n},
\]
with 
\[
\underline R_n
 :=
 \frac{\tau_p\underline P_n}{\underline P_{n+1}},
 \qquad
 \overline R_n
 :=
 \frac{\tau_p\overline P_n}{\overline P_{n+1}}.
\]
If \(n+1\) is admissible, then
\[
 U_{p,\tau_p}^{n,n+1}
 =
 -M_{p,\tau_p}^{n,n+1}
 =
 \frac1{q_{n+1}}.
\]
\end{proposition}

Thus the envelopes and factors at the lower- and upper-neighbour absolute means are already explicit. The next theorem gives the exact value of the supremum of the diagonal envelope in several cases and an explicit upper bound in the others.

\begin{theorem}
\label{thm:IP-diagonal}
Let \(Z\sim\mathrm{IP}(\mu;\kappa_2,\kappa_1,\kappa_0)\) and \(\mathcal I_p=(l,u)\). Let \(n\geq0\)
be admissible.
Then, at every finite endpoint \(e\) of $\mathcal{I}_p$,
\[
 U_{p,\tau_p}^{n,n}(x)
 \longrightarrow
 \frac{2}{|\lambda_{n+1}(e)|}
 \qquad\text{as }x\to e.
\]
If \(n+1\) is admissible, then
\begin{equation}
\label{eq:IP-global-envelope}
 U_{p,\tau_p}^{n,n}(x)
 <
 \frac{2}{|\lambda_{n+1}(x)|},
 \qquad x\in(l,u),
\end{equation}
where the right-hand side is interpreted as \(+\infty\)
when \(\lambda_{n+1}(x)=0\).
This bound is sharp at every finite endpoint.

Set
\[
 \theta_n
 :=
 q_n+q_{n+1}
 =
 2n+1-2n^2\kappa_2,
 \qquad
 g_n(x)
 :=
 2\theta_n\tau_p(x)+\lambda_{n+1}(x)^2.
\]
When \(\kappa_2>-1/2\), also set
\[
 x_n^*
 :=
 \mu-\frac{(n+1)\tau_p'(\mu)}{1+2\kappa_2}.
\]

\begin{enumerate}
\item[(i)]
If \(\kappa_2\leq-1/2\), or if
\(\kappa_2>-1/2\) and \(x_n^*\notin(l,u)\), then
\[
 \sup_{x\in\mathcal I_p}U_{p,\tau_p}^{n,n}(x)
 =
 \max_{e\in\{l,u\}\cap\mathbb R}
 \frac{2}{|\lambda_{n+1}(e)|}
 =
 \max_{e\in\{l,u\}\cap\mathbb R}
 \frac{2}{|\mu-e|+n\sqrt{D_{\tau_p}}},
\]
where
\(D_{\tau_p}=\kappa_1^2-4\kappa_0\kappa_2\).
The supremum is approached at the finite endpoint where
\(|\lambda_{n+1}|\) is smallest.
It is not attained in \(\mathcal I_p\) unless \(g_n\)
is constant.

\item[(ii)]
Suppose that \(\kappa_2>-1/2\) and \(x_n^*\in(l,u)\).
If \(\theta_n>0\), in particular if \(n+1\) is admissible,
then
\[
 \sup_{x\in\mathcal I_p}U_{p,\tau_p}^{n,n}(x)
 \leq
 \frac{2}{\sqrt{g_n(x_n^*)}},
\]
where
\[
 g_n(x_n^*)
 =
 \theta_n
 \frac{
 2(1+2\kappa_2)\tau_p(\mu)-\tau_p'(\mu)^2
 }{1+2\kappa_2}.
\]

\item[(iii)]
If \(p\) is symmetric about \(\mu\) and
\(\kappa_2>-1/2\), then \(x_n^*=\mu\) and the supremum
is attained at \(\mu\).
For admissible \(n\geq1\),
\[
 \sup_{x\in\mathcal I_p}U_{p,\tau_p}^{n,n}(x)
 =
 U_{p,\tau_p}^{n,n}(\mu)
 =
 \frac{\mathbb E|Z-\mu|^n}
 {n\tau_p(\mu)\mathbb E|Z-\mu|^{n-1}}.
\]
At order zero,
\[
 U_{p,\tau_p}^{0,0}(\mu)
 =
 \frac1{2\tau_p(\mu)p(\mu)}.
\]
Since \(\mathbb E|Z-\mu|=2\tau_p(\mu)p(\mu)\), we also have
\[
 U_{p,\tau_p}^{1,1}(\mu)
 =
 2p(\mu).
\]
\end{enumerate}
\end{theorem}

\begin{remark}
\label{rem:IP-diagonal}
We refer to cases (i), (ii) and (iii) as the boundary, interior and symmetric
regimes, respectively.
In the interior regime, the proof shows that
\(U_{p,\tau_p}^{n,n}\) has no local minimum away from
\(x_n^*\).
There is therefore at most one local maximum on each side
of \(x_n^*\), so the supremum can be found by two unimodal
numerical searches.

The case \(\kappa_2\leq-1/2\) occurs only for Beta
distributions, up to affine changes of variable, with
\(\alpha+\beta\leq2\).
When \(\kappa_2>-1/2\), the position of \(x_n^*\) can be
checked at the finite endpoints:
\[
 x_n^*\leq l
 \quad\Longleftrightarrow\quad
 (1-2n\kappa_2)(\mu-l)
 \leq
 (n+1)\tau_p'(l),
\]
and
\[
 x_n^*\geq u
 \quad\Longleftrightarrow\quad
 (1-2n\kappa_2)(u-\mu)
 \leq
 (n+1)|\tau_p'(u)|.
\]

In our numerical examples, the upper bound in the interior
regime improves as \(n\) increases. It can be rather loose
at order zero, where direct numerical maximization is
often preferable.
\end{remark}

Proposition~\ref{prop:IP-absolute-expectations}, Theorem~\ref{thm:IP-diagonal} together with Theorem~\ref{thm:integrated-pearson-kernel-representations} easily yield one term envelopes and factors. 
For integers \(n,k\geq0\), we write the Stein factors associated with the class of test functions whose $k$th derivative exists and is bounded
\[
 c_{p,w}^{n,k}
 :=
 \sup_{\substack{h\in L^1(p)\\\|h^{(k)}\|\leq1}}
 \|f_{p,w,h}^{(n)}\|,
\]
with the derivative seminorms and regularity requirements understood
according to Notation~\ref{not:derivative-norms}.

\begin{corollary}
\label{cor:integrated-pearson-envelopes}
Let \(Z\sim\mathrm{IP}\) and let \(n\geq0\) be admissible.
Then, with each term used only when the corresponding representation
applies and the test-function seminorm is finite. 
\begin{equation}
\label{eq:tau-factor-n-1nn+1}
 |f_{p,\tau_p,h}^{(n)}(x)|
 \leq
 \min\left\{
 \frac{\operatorname{osc}(h^{(n-1)})}{\tau_p(x)},
 \quad
 U_{p,\tau_p}^{n,n}(x)\|h^{(n)}\|,
 \quad
 \frac{\|h^{(n+1)}\|}{q_{n+1}}
 \right\}.
\end{equation}
The first term is omitted when \(n=0\); at that order, the
middle term may be replaced by
\[
 \frac12U_{p,\tau_p}^{0,0}(x)\operatorname{osc}(h).
\]
The last term requires \(n+1\) to be admissible.
The associated Stein factors are
\[
 c_{p,\tau_p}^{n,n-1}
 =
 2\|\frac1{\tau_p}\|,
 \qquad
 c_{p,\tau_p}^{n,n}
 =
 \sup_{x\in\mathcal I_p}U_{p,\tau_p}^{n,n}(x),
 \qquad
 c_{p,\tau_p}^{n,n+1}
 =
 \frac1{q_{n+1}},
\]
at the corresponding admissible orders, with \(n\geq1\)
for the first identity.

In particular, for admissible \(n\geq1\),
\[
 \|\tau_pf_{p,\tau_p,h}^{(n)}\|
 \leq
 \operatorname{osc}(h^{(n-1)})
 \leq
 2\|h^{(n-1)}\|.
\]
\end{corollary}

\begin{remark}
    We stress that these envelopes are sharp. For the diagonal and upper-neighbour
representations, one chooses the relevant test derivative to have
the sign of the kernel. For the centred lower-neighbour
representation, one takes \(h^{(n-1)}(x)\) at one end of its
range and lets \(h^{(n-1)}\) approach the other end outside a
shrinking neighbourhood of \(x\).
When necessary, these choices are approximated by smooth
test derivatives with the same norm bound. Their antiderivatives
have at most the polynomial growth allowed by the corresponding
admissibility assumption. Sharpness is meant separately for each
test-function seminorm.
\end{remark}

At orders zero and one, this recovers the bounds obtained in
Section~\ref{subsec:universal-low-order}, with the further
simplification \(U_{p,\tau_p}^{1,2}=1/q_2\) when order \(2\)
is admissible:
\begin{align}
 |f_{p,\tau_p,h}(x)|
 &\leq
 \min\left\{
 \frac{P(x)\overline P(x)}{\tau_p(x)p(x)}
 \operatorname{osc}(h),
 \quad
 \|h^{(1)}\|
 \right\},
 \label{eq:ip-template-order-zero}\\
 |f_{p,\tau_p,h}'(x)|
 &\leq
 \min\left\{
 \frac{\operatorname{osc}(h)}{\tau_p(x)},
 \quad
 2\frac{\underline P_2(x)\overline P_2(x)}
        {\tau_p(x)^2p(x)}\|h^{(1)}\|,
 \quad
 \frac1{q_2}\|h^{(2)}\|
 \right\}.
 \label{eq:ip-template-order-one}
\end{align}
The last term in
\eqref{eq:ip-template-order-one} is included only when order \(2\)
is admissible.
As already observed at low orders, the lower-neighbour estimates
give finite uniform factors whenever \(\inf_x\tau_p(x)>0\).
For Gamma and Beta laws, where \(\tau_p\) vanishes at the
boundary, they remain useful in weighted form.

Table~\ref{tab:ip-factors} collects the factors obtained above
for the Stein-kernel weight.

\begin{sidewaystable}[p]
\centering
\small
\renewcommand{\arraystretch}{2.4}
\begin{tabular}{@{}lccc@{}}
\hline
Target
&
$c_{p,\tau_p}^{n,n-1}$
&
$c_{p,\tau_p}^{n,n}$
&
$c_{p,\tau_p}^{n,n+1}$
\\
\hline
$\mathcal N(\mu,\sigma^2)$
&
$\dfrac{2}{\sigma^2}$
&
$\dfrac{\Gamma(\frac{n+1}{2})}
 {\sqrt2\,\sigma\,\Gamma(\frac n2+1)}$
&
$\dfrac1{n+1}$
\\[4pt]
$\mathrm{Exp}(\beta)$
&
$\infty$ (weighted: $2$)
&
$\dfrac{2}{(n+1)\beta}$
&
$\dfrac1{n+1}$
\\[4pt]
$\Gamma(\alpha,\beta)$
&
$\infty$ (weighted: $2$)
&
$\begin{cases}
 \dfrac{2}{\beta(\alpha+n)},
 & \alpha\leq n+1,\\[4pt]
 \leq\dfrac{2}{\beta\sqrt{(2n+1)(2\alpha-1)}},
 & \alpha>n+1
\end{cases}$
&
$\dfrac1{n+1}$
\\[10pt]
$\mathrm B(\alpha,\beta)$
&
$\infty$ (weighted: $2$)
&
$\begin{cases}
 \frac{2(\alpha+\beta)}{\min\{\alpha,\beta\}+n}.,
 & \min\{\alpha,\beta\}
 \leq\dfrac{(\alpha+\beta)(n+1)}{\alpha+\beta+2n},\\[4pt]
 \frac{2\alpha\,
 \Gamma\bigl(\frac{n+1}{2}\bigr)
 \Gamma\bigl(\alpha+\frac n2\bigr)}{\Gamma\bigl(\frac n2+1\bigr)
 \Gamma\bigl(\alpha+\frac{n+1}{2}\bigr)},
 & \alpha=\beta>1,\\[4pt]
 \le 2(\alpha+\beta)
 \sqrt{\frac{(\alpha+\beta)(\alpha+\beta-2)}
 {(2\alpha\beta-\alpha-\beta)
  \bigl((2n+1)(\alpha+\beta)+2n^2\bigr)}},
 & \text{otherwise}
\end{cases}$
&
$\dfrac{\alpha+\beta}
 {(n+1)(\alpha+\beta+n)}$
\\[4pt]
$t_\nu$
&
$\dfrac{2(\nu-1)}{\nu}$ ($n<\nu$)
&
$c^{n,n}_{p,\tau_p} = \frac{\nu-1}{2\sqrt\nu}\,\frac{\Gamma\!\left(\frac{n+1}{2}\right)\Gamma\!\left(\frac{\nu-n}{2}\right)}{\Gamma\!\left(\frac n2+1\right)\Gamma\!\left(\frac{\nu-n+1}{2}\right)}$, $0\le n<\nu
$&
$\dfrac{\nu-1}{(n+1)(\nu-n-1)}$ ($n+1<\nu$)
\\
\hline
\end{tabular}

\caption{Neighbouring-order Stein factors for $\mathrm{IP}$ distribution and \(w=\tau_p\).
The lower-neighbour column applies for \(n\geq1\);
the other columns include \(n=0\).
At every admissible order \(n\geq1\),
\(
 \|\tau_pf_{p,\tau_p,h}^{(n)}\|
 \leq
 \operatorname{osc}(h^{(n-1)})
 \leq
 2\|h^{(n-1)}\|,
\)
with sharp constants. \\ The entries marked "weighted" indicate cases where the
unweighted lower-neighbour factor is infinite.
The diagonal entries are exact except where an upper bound
is indicated; in those cases, the supremum of the envelope
can be computed numerically.}
\label{tab:ip-factors}
\end{sidewaystable}

\begin{example}[Beta distribution]
\label{ex:beta}
Consider the Beta distribution with parameters \(\alpha=2\) and \(\beta=5\).
To obtain diagonal factors for \(n=0,1,2\), the upper bounds of the interior regime in Theorem~\ref{thm:IP-diagonal} give  
\[
 \begin{aligned}
 c_{p,\tau_p}^{0,0}&\le 8.6824\ldots,
 &
 c_{p,\tau_p}^{1,1}&\le 4.7899\ldots, &
 c_{p,\tau_p}^{2,2}&\le 3.5031\ldots,
 \end{aligned}
\]
whereas numerical maximization gives
\[
 \begin{aligned}
 c_{p,\tau_p}^{0,0} &\approx 8.1132878\ldots,
 &
 c_{p,\tau_p}^{1,1} &\approx 4.7677731\ldots, &
 c_{p,\tau_p}^{2,2} &\approx 3.5027178\ldots,
 \end{aligned}
\]
with maximizers at approximately
\(0.15368\), \(0.09189\), and \(0.02359\), all to the
left of \(x_n^*=(7-3n)/35\).
From order three onward, we are in the boundary regime and the diagonal factor for $n=3$ is
\[
 \frac{14}{5}
 =
 \lim_{x\downarrow0}U_{p,\tau_p}^{3,3}(x).
\]

The literature generally uses the weight
\[
 \eta(x)=x(1-x)=(\alpha+\beta)\tau_p(x).
\]
Our factors have to be divided by $\alpha + \beta$ to be compared. For instance, our bound
\begin{equation*}
    \|f_{p,\tau_p,h}\| \le \| h^{(1)}\| \Leftrightarrow \|f_{p,\eta,h}\| \le \frac{1}{\alpha+\beta}\|h^{(1)}\|
\end{equation*}
This agrees with
\cite[Proposition~3.13(a)]{dobler_beta_2015}
and improves by a factor \(2\) the bound in
\cite[Lemma~3.2]{goldstein_reinert_beta_2013}.
\end{example}

\begin{example}[Exponential distribution]
\label{ex:expon}
Let \(Z\) have the exponential distribution with scale
parameter \(\beta>0\). Then \(\kappa_2=0\), \(q_n=n\), and all orders are admissible.

Both the Stein kernel $\tau_p(x) =\beta x$ and the score $\rho_p(x)=-\frac1\beta$ have particularly simple forms, so we may consider the two weights \(w=\tau_p\) and
\(w=1\). Their Stein equations are
\[
 \beta xf'(x)-(x-\beta)f(x)
 =
 h(x)-\mathbb E[h(Z)],
\]
and
\[
 f'(x)-\frac1\beta f(x)
 =
 h(x)-\mathbb E[h(Z)]
\]
respectively.

For \(w=\tau_p\), the lower- and upper-diagonal factors are given by 
Corollary~\ref{cor:integrated-pearson-envelopes} and the diagonal factor is exact by the boundary regime. Indeed, in that setup, \(\mu=\beta\) and \(\tau_p'=\beta\), so
\[
 x_n^*=-n\beta\leq0,
 \qquad
 \lambda_{n+1}(0)=(n+1)\beta.
\]
Thus, for every \(n\geq0\)
\[
 c_{p,\tau_p}^{n,n}
 =
 \lim_{x\downarrow0}U_{p,\tau_p}^{n,n}(x)
 =
 \frac{2}{(n+1)\beta}.
\]

For the weight \(w=1\), the condition \eqref{eq:closed-three-neighbour-condition} does not hold as Lemma~\ref{lem:exp-Mn-Un} shows
\[
 M_{p,1}^{n,n-1}=-1,
 \qquad
 M_{p,1}^{n,n}=-\beta,
 \qquad n\geq1.
\]
Therefore only the lower-neighbour and diagonal correction terms vanish. In the upper-neighbour representation, the correction sum vanishes but the correction term
\(-\beta h^{(n)}\) remains.
Using the absolute kernel means from
Lemma~\ref{lem:exp-Mn-Un}, we obtain
\[
 \|f_{p,1,h}^{(n)}\|
 \leq
 \min\left\{
 \operatorname{osc}(h^{(n-1)}),
 \quad
 \beta\|h^{(n)}\|,
 \quad
 \beta\|h^{(n)}\|+\beta^2\|h^{(n+1)}\|
 \right\},
 \qquad n\geq1.
\]
Since
\(\operatorname{osc}(h^{(n-1)})
 \leq2\|h^{(n-1)}\|\),
the first term recovers the sharp factor \(2\) of
\cite{daly_2008}.
The third term is included to show what happens when one
transfers another derivative to \(h\): a contribution from
\(h^{(n)}\) remains, and the resulting bound is worse than
the diagonal one.
\end{example}

\begin{remark}[Weighted factors]
    The last equality of Corollary \ref{cor:integrated-pearson-envelopes} is about the function $\tau_pf_{p,\tau_p,h}^{(n)}$. In \cite[Proposition~2.11]{germain_swan_2025} the authors showed that if an integrated Pearson distributions had a finite endpoint or was symmetric, then
    \begin{equation}
         (-1)^n
 \bigl(\tau_p\overline m_{p,\tau_p}'\bigr)'
 \leq0,
 \qquad
 \bigl(\tau_p\underline m_{p,\tau_p}'\bigr)
 \geq0,
 \qquad
 {\text{a.e. on }\mathcal I_p},
    \end{equation}
from which they conclude that    
    \begin{equation}\label{eq:tau-f'-prime}
    \|(\tau_pf_{p,\tau_p,h}')'\| \le 2 \| h'\|.\end{equation}
\end{remark}
    The preceding Remark can be generalised to higher order. For instance, we have the following result about the Gamma distribution.
    \begin{lemma}\label{lem:gamma-weighted-monotonicity}
Let \(p\) be the Gamma density with shape parameter \(\alpha\geq1\) and
scale parameter \(\beta>0\). For every \(n\ge 0\),
\begin{equation}
(-1)^n(\tau_p\overline{m}_{p,\tau_p}^{(n)})'\leq0,
\qquad
(\tau_p\underline m_{p,\tau_p}^{(n)})'\geq0,
\end{equation}
and
\[
 \|(\tau_pf_{p,\tau_p,h}^{(n)})'\|
 \leq
 2\|h^{(n)}\|.
\]
\end{lemma}

This is an instance of a more general result about weighted derivatives of the Stein solution that is immediate from our kernel calculus. Let $n\geq0$ and
$\omega:\mathcal I_p\to(0,\infty)$ locally absolutely
continuous, and define the weighted kernel
\begin{align}
\label{eq:Lambdan}
 \mathrm {K}_{p,w,\omega}^{n}(x,v)
 &:=
 \omega'(x)K_{p,w}^{n,n}(x,v)
 +
 \omega(x)K_{p,w}^{n+1,n}(x,v)
 \notag\\
 &=
 (-1)^n
 \bigl(\omega\overline m_{p,w}^{(n)}\bigr)'(x)
 \frac{\underline P_n(v)}{p(v)}
 \mathbf 1_{\{v\leq x\}}
 -
 \bigl(\omega\underline m_{p,w}^{(n)}\bigr)'(x)
 \frac{\overline P_n(v)}{p(v)}
 \mathbf 1_{\{v>x\}},
\end{align}
and its absolute mean
\[
 \mathrm U_{p,w,\omega}^{n}
 :=
 \left|
 \bigl(\omega\overline m_{p,w}^{(n)}\bigr)'
 \right|\underline P_{n+1}
 +
 \left|
 \bigl(\omega\underline m_{p,w}^{(n)}\bigr)'
 \right|\overline P_{n+1}.
\]
wherever the displayed derivatives exist.
Here \(w\) specifies the Stein operator and \(\omega\) specifies
the derivative to be controlled, and the canonical choice is
\(\omega=w\), when $w>0$. 

\begin{proposition}
\label{prop:all-order-weighted-factor}
Fix $n\geq0$ and let $h\in L^1(p)$.
Assume that the hypotheses of
Corollary~\ref{cor:closed-three-neighbour-calculus},
including \eqref{eq:closed-three-neighbour-condition},
hold for the diagonal representation at order $n$.
Assume that $\omega:\mathcal I_p\to(0,\infty)$ is locally
absolutely continuous, and that the regularity and integrability
conditions of Proposition~\ref{lma:basis-prop-of-kernels} hold at the
indices $(n,n)$ and $(n+1,n)$.
Then, for almost every $x\in\mathcal I_p$,
\begin{equation}
\label{eq:weighted-n}
\begin{split}
 \bigl(\omega f_{p,w,h}^{(n)}\bigr)'(x)
 &=
 \mathbb E\Bigl[
 \mathrm{K}_{p,w,\omega}^{n}(x,Z)\bigl(h^{(n)}(Z)-h^{(n)}(x)\bigr)
 \Bigr],
\end{split}
\end{equation}
and,
\begin{equation}
\label{eq:weighted-n-crude}
 \left|
 \bigl(\omega f_{p,w,h}^{(n)}\bigr)'(x)
 \right|
 \leq
 \mathrm U_{p,w,\omega}^{n}(x)
    \mathrm{osc}(h^{(n)}).
\end{equation}
If moreover
\begin{equation}
\label{eq:weighted-all-order-sign-condition}
 (-1)^n
 \bigl(\omega\overline m_{p,w}^{(n)}\bigr)'
 \leq0,
 \qquad
 \bigl(\omega\underline m_{p,w}^{(n)}\bigr)'
 \geq0,
 \qquad
 {\text{a.e. on }\mathcal I_p},
\end{equation}
then $\mathrm U_{p,w,\omega}^{n}=\omega/w$
and this sign condition in particular forces
$w>0$ almost everywhere. Hence
\begin{equation}
\label{eq:weighted-n-osc}
\begin{split}
 \left|
 \bigl(\omega f_{p,w,h}^{(n)}\bigr)'(x)
 \right|
 &\leq
 \frac{\omega(x)}{w(x)}
 \operatorname{osc}\bigl(h^{(n)}\bigr),
 \\
 \left\|
 \bigl(\omega f_{p,w,h}^{(n)}\bigr)'
 \right\|
 &\leq
  \left\|\frac{\omega}{w}\right\|
 \operatorname{osc}\bigl(h^{(n)}\bigr)
 \leq
 2\left\|\frac{\omega}{w}\right\|\|h^{(n)}\|.
\end{split}
\end{equation}
\end{proposition}

Thus, choosing $\omega = w$ in Proposition~\ref{prop:all-order-weighted-factor}, implies that if
\[
 (-1)^n
 \bigl(w\overline m_{p,w}^{(n)}\bigr)'\leq0,
 \qquad
 \bigl(w\underline m_{p,w}^{(n)}\bigr)'\geq0
 \qquad
\text{a.e. on }\mathcal I_p,
\]
then
\[
 \left\|
 \bigl(wf_{p,w,h}^{(n)}\bigr)'
 \right\|
 \leq
 \operatorname{osc}\bigl(h^{(n)}\bigr)
 \leq
 2\|h^{(n)}\|.
\]
\begin{remark}
\label{rem:case-by-case}
When $\mathcal I_p=(l,u)$, {\cite[Lemma~2.10]{germain_swan_2025}} showed the following sufficient condition for \eqref{eq:weighted-all-order-sign-condition} to hold for $n=1$, hence the bound \eqref{eq:tau-f'-prime} to apply:
\begin{equation}
\label{eq:weighted-factor-sufficient-condition}
 2\tau_p(x)+(\mu-x)\tau_p'(x)\geq0,
 \qquad
 \text{for a.e. }x\in(l,u).
\end{equation}
\end{remark}

\section{Applications}
\label{sec:applications}
The representations and factors obtained through the kernel calculus
can strengthen existing approximation arguments in several ways.

\subsection{Improvements of known quantitative results}
\label{sec:Improvement-known-distances}

As already mentioned in Section~\ref{sec:closed-kernel-calculi},
most papers applying Stein's method use Stein factors that were
not claimed to be sharp. One can now repeat the same computations
with the factors obtained through our kernel representations,
directly improving known quantitative results. Moreover, many
of these computations involve terms of the form
\[
 \int_D
 \bigl(
 (\tau_p f_{p,\tau_p,h}')(t)
 -
 (\tau_p f_{p,\tau_p,h}')(x)
 \bigr)\,dt
\]
for a certain region $D\subseteq\mathbb R$.
If \eqref{eq:weighted-all-order-sign-condition} holds for $n=1$, the weighted bound
from \eqref{eq:weighted-n-osc} for $\omega=w=\tau_p$ gives
\[
 \left|
 \int_D
 \bigl(
 (\tau_p f_{p,\tau_p,h}')(t)
 -
 (\tau_p f_{p,\tau_p,h}')(x)
 \bigr)\,dt
 \right|
 \leq
 \|(\tau_p f_{p,\tau_p,h}')'\|
 \int_D|t-x|\,dt
 \leq
 2\|h^{(1)}\|\int_D|t-x|\,dt,
\]
often yielding further improvements.

\paragraph{Beta approximation.}
Let $W_n=L_{2n}/(2n)$, where $L_{2n}$ is the last return to zero
of a simple symmetric random walk during its first $2n$ steps,
as considered by Goldstein and Reinert
\cite{goldstein_reinert_beta_2013}.
Let $Z\sim\mathrm{Beta}(1/2,1/2)$.
Their argument gives, for Lipschitz test functions $h$,
\[
 \left|\mathbb E[h(W_n)]-\mathbb E[h(Z)]\right|
 \leq
 \frac{1}{2n}\|h'\|
 +
 \frac{1}{2n}\|f_{p,\eta,h}\|
 +
 \left(\frac{3}{2n}+\frac{1}{n^2}\right)
 \|f_{p,\eta,h}'\|,
\]
where $\eta=(\alpha+\beta)\tau_p$.
Since $\alpha=\beta=1/2$, we have $\eta=\tau_p$ and can
directly use our factors for $f_{p,\tau_p,h}$ and
$f_{p,\tau_p,h}'$ in terms of $\|h'\|$, namely $1$ and $4/3$,
in place of their $2$ and $8$.
Taking the supremum over $\|h'\|\leq1$ yields
\[
 d_{\mathrm W}(W_n,Z)
 \leq
 \frac{3}{n}+\frac{4}{3n^2},
\]
improving the bound
\[
 \frac{27}{2n}+\frac{8}{n^2}
\]
obtained in \cite[Section~4]{goldstein_reinert_beta_2013}.
The same kind of improvement applies to the occupation-time and
first-passage statistics of \cite{goldstein_reinert_beta_2013}. For Beta approximation in the  P\'olya--Eggenberger urn scheme (see Section~\ref{subsec:beta-edgeworth} for details),
the same argument with $c^{1,1}_{p,\eta}$ gives
\[d_{\mathrm W}(W_n,Z)\leq\frac1n\{1+\left(\frac{ab}{a+b}+\frac12\right)c^{1,1}_{p,\eta}\},\]
which improves the constants of
\cite[Theorem~1.1]{goldstein_reinert_beta_2013} and
\cite[Remark~12]{germain_swan_2025}.

\paragraph{Quartic approximation in the critical Curie--Weiss model.}
Let $\sigma=(\sigma_1,\ldots,\sigma_n)\in\{-1,1\}^n$ be
distributed according to the Curie--Weiss Gibbs measure at
the critical inverse temperature,
\[
 \mathbb P(\sigma)
 =
 \frac1{\mathcal Z_n}
 \exp\left(
 \frac{1}{2n}
 \left(\sum_{i=1}^n\sigma_i\right)^2
 \right),
 \qquad \sigma\in\{-1,1\}^n.
\]
Define
\begin{equation}
 W_n=n^{-3/4}\sum_{i=1}^n\sigma_i.
 \label{eq:criricuriewei}
\end{equation}
It is classical \cite{EllisNewman1978Statistics} that $W_n$
converges in distribution to a random variable $Z_4$ with the quartic
Subbotin density $p_4$, and \cite[Theorem~2.1]{Chatterjee_2011} shows that
$d_{\mathrm K}(W_n,Z_4)\leq Cn^{-1/2}$ for an unspecified absolute constant $C$.
Their proof uses the exchangeable pair $(W_n,W_n')$ given by one step of
the Glauber dynamics. Write $\Delta_n=W_n-W_n'$, $c_0=n^{3/2}$, and define
$r$ by $\mathbb E[\Delta_n\mid W_n]=c_0^{-1}W_n^3/3+r(W_n)$.
For every Lipschitz $h$, the proof of
\cite[Theorem~1.1]{Chatterjee_2011} gives
\begin{equation}
\label{eq:CS-plug-in}
\begin{split}
 \bigl|\mathbb E[h(W_n)]-\mathbb E[h(Z_4)]\bigr|
 &\leq
 \|f_{p_4,1,h}'\|\,\mathbb E\Bigl|1-\frac{c_0}2\mathbb E[\Delta_n^2\mid W_n]\Bigr|
 +\frac{c_0}4\|f_{p_4,1,h}''\|\,\mathbb E|\Delta_n|^3\\
 &\qquad
 +c_0\|f_{p_4,1,h}\|\,\mathbb E|r(W_n)|,
\end{split}
\end{equation}
and \cite[Lemma~5.1]{Chatterjee_2011} controls the terms that do not
involve the Stein solution:
\[
 \mathbb E\Bigl|1-\frac{c_0}2\mathbb E[\Delta_n^2\mid W_n]\Bigr|
 \leq\frac{15}{2}n^{-1/2},
 \quad
 c_0\mathbb E|r(W_n)|\leq15n^{-1/2},
 \quad
 |\Delta_n|\leq2n^{-3/4},
 \quad
 \mathbb E|W_n|^3\leq15.
\]
With the Stein factors of \cite[Lemmas~4.1 and~4.2]{Chatterjee_2011},
\eqref{eq:CS-plug-in} gives
\[d_{\mathrm W}(W_n,Z_4)\leq1533n^{-1/2}+154n^{-3/4},\]
and the same inputs
in the proof of \cite[Theorem~1.2]{Chatterjee_2011} give
$d_{\mathrm K}(W_n,Z_4)\leq124n^{-1/2}+814n^{-3/4}$.
Here and below, constants are rounded upwards; the exact expressions are
given in Appendix~\ref{app:CW-edgeworth}.
Replacing their Stein factors by those of
Examples~\ref{subsec:subbotin-example} and~\ref{ex:subbotin-kolmo}
improves these bounds to
\begin{equation}
\label{eq:improv-distances-CW}
 d_{\mathrm K}(W_n,Z_4)
 \leq
 28n^{-1/2}+113n^{-3/4},\qquad
 d_{\mathrm W}(W_n,Z_4)
 \leq
 31n^{-1/2}+5n^{-3/4}.
\end{equation}
As the prescribed error tends to zero, these bounds reduce
the sufficient value of $n$ by a factor tending to
$(124/28)^2$, roughly $20$, for the Kolmogorov distance,
and $(1533/31)^2$, roughly $2450$, for the Wasserstein
distance.
\subsection{Edgeworth corrections}
\label{subsec:edgeworth-corrections}
We now use the kernel calculus to obtain explicit Edgeworth corrections
for the Beta and quartic Subbotin approximations.

\subsubsection{Beta approximation for the P\'olya--Eggenberger urn}
\label{subsec:beta-edgeworth}

This section illustrates how the factors of
Section~\ref{sec:integrated-pearson} combine with an exchangeable pair.
In the P\'olya--Eggenberger urn, each drawn ball is returned together
with $c>0$ additional balls of the same colour. If the urn initially
contains $\tilde a$ red and $\tilde b$ blue balls, set $a=\tilde a/c$ and
$b=\tilde b/c$, and let $W_n$ be the proportion of red balls among the
first $n$ draws. Then $W_n\Rightarrow Z\sim\operatorname{Beta}(a,b)$
\cite{dobler_beta_2015}. Write $p$ for the density of $Z$ and
\[
 \eta(x)=x(1-x)=(a+b)\tau_p(x),
 \qquad
 \gamma(x)=a-(a+b)x,
 \qquad
 \ell(x)=\frac{a+(b-a)x}{2}.
\]
Since $\eta=(a+b)\tau_p$, \eqref{eq:IP-representation-n-n} gives, for every
$r\geq1$,
\begin{equation}
\label{eq:urn-diagonal-factor}
 f_{p,\eta,h}^{(r)}(x)=\mathbb E\bigl[K_{p,\eta}^{r,r}(x,Z)h^{(r)}(Z)\bigr],
 \qquad
 \|f_{p,\eta,h}^{(r)}\|\leq c_{p,\eta}^{r,r}\|h^{(r)}\|,
 \qquad
 c_{p,\eta}^{r,r}=\frac{c_{p,\tau_p}^{r,r}}{a+b},
\end{equation}
with the values of Table~\ref{tab:ip-factors}. Proofs are in
Appendix~\ref{app:beta-edgeworth}.

Conditionally on $Z$, $nW_n$ is $\operatorname{Bin}(n,Z)$, so Taylor's
formula gives
\begin{equation}
\label{eq:beta-app-smooth-approximation}
 \bigl|\mathbb E[h(W_n)]-\mathbb E[h(Z)]\bigr|
 \leq
 \frac{ab}{2(a+b)(a+b+1)}\,\frac{\|h''\|}{n},
\end{equation}
with equality for $h(x)=x^2$. One can gain an order by subtracting the first correction. Doing so via conditional Taylor expansion would need
$h^{(3)}$ Lipschitz. Our representatiosn and factors save one derivative of $h$.

\begin{theorem}[First-order Beta Edgeworth expansion]
\label{thm:edgeworth-one-term}
Let $h\in C^{2,1}([0,1])$ and $\mathrm e_1(x)=\{(a+b)x-a\}/2$. For every
$n\geq1$,
\[
 \left|
 \mathbb E[h(W_n)]-\mathbb E[h(Z)]
 -\frac1n\mathbb E[\mathrm e_1(Z)h'(Z)]
 \right|
 \leq
 \frac{C_2\|h^{(2)}\|+C_3\|h^{(3)}\|}{n^2},
\]
where $h^{(3)}$ is understood almost everywhere and
\begin{equation}
\label{eq:beta-app-explicit-constants}
 C_2=\frac{ab\,|b-a|}{2(a+b)(a+b+1)}\,c_{p,\eta}^{2,2},
 \qquad
 C_3=\frac{ab}{a+b}\left(\frac1{12}+\frac{\max(a,b)}{4(a+b+1)}\right)c_{p,\eta}^{3,3}.
\end{equation}
\end{theorem}

Since $\mathbb E[\mathrm e_1(Z)h'(Z)]=\frac12\mathbb E[\eta(Z)h''(Z)]$, the
correction is the Beta average of the first term of Voronovskaya's
formula for Bernstein polynomials \cite[Section~1.6]{lorentz_1986}. The coefficients in \eqref{eq:beta-app-explicit-constants} are easy to compute numerically. 
For example, with  $\operatorname{Beta}(2,5)$, one gets $C_2\approx0.1340$ and $C_3=23/168\approx0.1369$.

\begin{remark}[Higher orders]
\label{thm:beta-edgeworth-all-orders}
Since $\mathbb E[h(W_n)\mid Z=x]$ is the $n$th Bernstein polynomial of
$h$ at $x$, Bernstein's complete asymptotic expansion
\cite[Theorem~B]{tachev_2012} extends
Theorem~\ref{thm:edgeworth-one-term} to all orders. Explicitly, write
\[\mathbb E[(W_n-Z)^r\mid Z]=\sum_{j=\lceil r/2\rceil}^{r-1}n^{-j}b_{r,j}(Z)\]
for $r\geq2$, with polynomials $b_{r,j}$, and set $b_{r,j}=0$ otherwise.
For $m\geq1$, $h\in C^{2m+1,1}([0,1])$ and every $n\geq1$, Taylor's
formula through degree $2m+1$ gives
\[
 \Biggl|\mathbb E[h(W_n)]-\mathbb E[h(Z)]
 -\sum_{j=1}^m n^{-j}\sum_{r=j+1}^{2j}
 \frac{\mathbb E[b_{r,j}(Z)h^{(r)}(Z)]}{r!}\Biggr|
 \leq n^{-(m+1)}\sum_{r=m+2}^{2m+2} C_{m,r}\|h^{(r)}\|,
\]
with $C_{m,r}=\frac1{r!}\sum_{j=m+1}^{r-1}\mathbb E|b_{r,j}(Z)|$. Since
$b_{2,1}=\eta$, the first correction is that of
Theorem~\ref{thm:edgeworth-one-term}, but -- as mentioned just above the statement of the theorem --  the bounding factor at $m=1$ requires one derivative
too many. It may be possible to extend the approach behind Theorem~\ref{thm:edgeworth-one-term} to weaken the assumptions on the regularity oh $h$ in the all order bound. 
\end{remark}

\begin{example}[Second order]
\label{ex:e2}
Since $b_{3,2}=\eta\eta'$ and $b_{4,2}=3\eta^2$, integrating by parts
turns the coefficient of $n^{-2}$ in
Remark~\ref{thm:beta-edgeworth-all-orders} into
\[\mathcal E_2(h)=\mathbb E[\mathrm e_2(Z)h''(Z)], \qquad
\mathrm e_2=-(p\eta\eta')'/(6p)+(p\eta^2)''/(8p).\] It can also be read
through the kernels:
\begin{equation}
\label{eq:E2-C}
 \mathcal E_2(h)=\mathbb E[Q_2(Z)f_{p,\eta,h}''(Z)],
 \qquad
 Q_2=-\frac{b-a}{4}\eta+\gamma\left(\frac\ell2-\frac1{12}\right),
\end{equation}
so that, by \eqref{eq:urn-diagonal-factor} and Fubini's theorem,
\begin{equation}
\label{eq:e2-general}
 \mathrm e_2(z)
 =\int_0^1Q_2(x)K_{p,\eta}^{2,2}(x,z)p(x)\,dx
 =\frac{3\gamma(z)^2-(1-2z)\gamma(z)-\{3(a+b)-2\}\eta(z)}{24}.
\end{equation}
For $(a,b)=(2,5)$,
$\mathrm e_2(z)=\frac5{12}-\frac{23}6z+\frac{19}3z^2$.
\end{example}

\subsubsection{A quartic Subbotin Edgeworth correction}
\label{subsec:critical-curie-weiss-edgeworth}

Let $Z_4\sim p_4$ be the quartic Subbotin distribution.
By symmetry, $\mathbb E[Z_4]=0$, while
\[
 \mathbb E[Z_4^2]
 =
 \sqrt{12}\,\frac{\Gamma(3/4)}{\Gamma(1/4)},
 \qquad
 \mathbb E[Z_4^6]=9\,\mathbb E[Z_4^2].
\]
The second identity follows from the quartic Stein identity
applied to $f(x)=x^3$. Thus
\[
 e_1(x)
 =
 \frac{x^2}{2}-\frac{x^6}{30}
 -\frac15\mathbb E[Z_4^2]
\]
satisfies $\mathbb E[e_1(Z_4)]=0$. This is the first-order correction
identified by Can and Pham \cite[Theorem~1.4]{can_pham_2017}: with
$p_1(t)=e^{-t^4/12}$ and $p_2(t)=(t^2/2-t^6/30)p_1(t)$, the
$n^{-1/2}$ term of their expansion is, once normalised, the signed
density $e_1p_4$, because
$\mathbb E[Z_4^2/2-Z_4^6/30]=\mathbb E[Z_4^2]/5$. In particular,
\cite[Corollary~1.5]{can_pham_2017} reads, in our notation,
\[
 \lim_{n\to\infty}\sqrt n\,\bigl(\mathbb P(W_n\leq x)-\mathbb P(Z_4\leq x)\bigr)
 =
 \mathbb E\bigl[e_1(Z_4)\mathbf 1_{\{Z_4\leq x\}}\bigr],
 \qquad x\in\mathbb R,
\]
and their Theorem~1.4 quantifies it as a Cram\'er-type moderate
deviation result for $0\leq x\leq n^{1/12}$, through a Laplace-method
analysis of the exact law of $W_n$. The following theorem gives the
same correction for Lipschitz test functions, with a remainder of
order $n^{-3/4}$, using the quartic Stein factors of
Example~\ref{subsec:subbotin-example} and exchangeable pairs. This order is optimal, see Remark~\ref{rem:CW-optimal-rate}. The proof is given in
Appendix~\ref{app:CW-edgeworth}.

\begin{theorem}
\label{thm:critical-CW-edgeworth}
Let $W_n$ be as in \eqref{eq:criricuriewei}.
There exists $C<\infty$ such that, for every $n\geq1$
and every $1$-Lipschitz function $h$,
\[
 \left|
 \mathbb E[h(W_n)]-\mathbb E[h(Z_4)]
 -\frac1{\sqrt n}\mathbb E[e_1(Z_4)h(Z_4)]
 \right|
 \leq Cn^{-3/4}.
\]
\end{theorem}
\begin{remark}
\label{rem:CW-optimal-rate}
The exponent \(3/4\) in Theorem~\ref{thm:critical-CW-edgeworth} cannot
be improved. Indeed, \(W_n\) takes values in the lattice
\(L_n=n^{-3/4}(n+2\mathbb Z)\), of mesh \(\delta_n=2n^{-3/4}\).
The function \(h_n(x)=\operatorname{dist}(x,L_n)\) is \(1\)-Lipschitz,
vanishes on \(L_n\) and satisfies \(0\leq h_n\leq\delta_n/2\).
Hence \(\mathbb E[h_n(W_n)]=0\), and the correction term is at most
\(\delta_n\mathbb E|e_1(Z_4)|/(2\sqrt n)=O(n^{-5/4})\).
On the other hand, the integral of \(h_n\) over each cell of \(L_n\) equals
\(\delta_n^2/4\), and \(p_4\) is bounded below on \([-1,1]\), so that
\(\mathbb E[h_n(Z_4)]\geq c\,\delta_n\) for some \(c>0\) and all \(n\)
large enough. The left-hand side of Theorem~\ref{thm:critical-CW-edgeworth}
with \(h=h_n\) is therefore at least \(c\,\delta_n-O(n^{-5/4})\), of
exact order \(n^{-3/4}\).
\end{remark}

\section{Use of generative AI}
Generative AI tools were used throughout the preparation of this
paper, for language editing, organisation, \LaTeX{} preparation,
symbolic and numerical computations, and the development and
checking of mathematical arguments. They also assisted in writing
the \texttt{Mathematica} code used for several figures.  Contributions to the mathematical arguments
include the recursions
\eqref{eq:generalized-laplace-coefficients-An}--
\eqref{eq:generalized-laplace-coefficients-Bn},
Lemma~\ref{lem:no-other-orders}(a),
Theorem~\ref{thm:IP-diagonal}, the development of the
Edgeworth corrections and their proofs in
Section~\ref{subsec:edgeworth-corrections}, and Lemma~\ref{lem:symmax-U21}. GenAI also was used to write this current section. 

The paper is also intended for use with generative AI tools.
Once the vocabulary, kernel identities, and recursions are
established, many applications of the calculus reduce to systematic
computations that readily lend themselves to automation.
To illustrate this, we have developed the \href{https://ferdinand.rapin.web.ulb.be/stein}{Stein Atlas}. This webpage contains two useful tools: first, a series of companion papers
with the assistance of generative AI tools.
These papers have two purposes: to provide usable representations,
envelopes, and factors for particular distributions, and to show
how the method can be applied and its computations automated once
the general framework is in place. Each paper follows a common template organized by families of test functions, collecting the envelopes and factors useful to use Stein's method for the associated integral probability metric.
Second, a dynamic Stein calculator which computes the Stein envelopes and factors for a density given in closed form, possibly with parameters, for the weights $w = 1, w = \tau_p$ or a weight chosen by the user. The result of the calculator computations are labelled to assess the level of confidence (closed form proved in this paper, exact supremum obtained symbolically, or numerical supremum). These computations run in the browser (SymPy via Pyodide).

Generative AI assistance does not replace mathematical verification.
The authors take full responsibility for the statements, proofs,
computations, numerical results, and references in this paper.

\section*{Acknowledgements}

We thank Robert Gaunt for helpful early discussions, particularly for pointing us to reference \cite{daly_2008} which is the starting point of this work.

\printbibliography
\appendix

\section{Proofs from Section \ref{sec:density-approach}}
\label{app:proof-of-no-other-orders}

\begin{proof}[Proof of Proposition \ref{prop:stein-solution}]
Since $h\in L^1(p)$, the function $p\mathcal S_ph$, extended
across the interior zeros of $p$ by its integral representation,
is locally absolutely continuous on $(l,u)$, tends to zero
at $l$ and $u$, and is continuous at every interior point $e_r$.
On each interval
$(e_{r-1},e_r)$, the function $wp$ is locally absolutely continuous
and nowhere zero, so its reciprocal is locally absolutely
continuous. Hence $f_{p,w,h}= \mathcal S_ph/(w)$ is locally absolutely
continuous there. The identity
\(
    wf_{p,w,h}p=p \mathcal S_ph
\)
also gives the required boundary and matching conditions.
Thus $f_{p,w,h}\in\mathcal F_{p,w}$. Differentiating almost
everywhere on each interval yields  \eqref{eq:w-stein-equation}.

Conversely, let $f\in\mathcal F_{p,w}$ with
$\mathcal A_{p,w}f\in L^1(p)$. Integrating
$(wfp)'=(\mathcal A_{p,w}f)p$ over each component of
$\mathcal I_p$, the internal boundary terms cancel by the
matching conditions, while the terms at $l$ and $u$ vanish.
It follows that $\mathcal P(\mathcal A_{p,w}f)=0$.
The same argument, integrating only up to $x\in\mathcal I_p$,
gives
 the  claim.

 Finally, if $f\in\mathcal F_{p,w}$ solves
\eqref{eq:w-stein-equation}, then
$\mathcal A_{p,w}f=h-\mathcal Ph$ almost everywhere, so
$\mathcal A_{p,w}f\in L^1(p)$. The inverse relation just proved
therefore gives
\[
    w(x)f(x)
    =\mathcal S_p(\mathcal A_{p,w}f)(x)
    =\mathcal S_p(h-\mathcal Ph)(x)
    =\mathcal S_ph(x),
    \qquad x\in\mathcal I_p.
\]
Since $w$ is nonzero on $\mathcal I_p$, this proves uniqueness.
\end{proof}

\begin{proof}[Proof of Lemma \ref{lma:Kernel calculus -- integration by parts}]
Stein integration by parts (see \cite[Lemma~2.18]{ernst_reinert_swan_2020}) gives
\begin{equation}\label{eqapp:SIBP}
\E[-\mathcal S_pK_{p,w}^{i,j}(x,Z)h'(Z)] = \E[\bigl(K_{p,w}^{i,j}(x,Z)-M^{i,j}_{p,w}(x)\bigr)h(Z)].
\end{equation}
Applying \eqref{eq:caninv} to
\(K_{p,w}^{i,j}(x,\cdot)\) gives
\begin{equation}\label{eqapp:integraloperator-on-kernel}
 -\mathcal S_pK_{p,w}^{i,j}(x,\cdot)(v)
 =
 \begin{cases}
 K_{p,w}^{i,j+1}(x,v)
 +M_{p,w}^{i,j}(x)P(v)/p(v), & v\leq x,\\
 K_{p,w}^{i,j+1}(x,v)
 -M_{p,w}^{i,j}(x)\overline P(v)/p(v), & v>x,
 \end{cases}
\end{equation}
By integration by parts, one has 
\[
\E[\frac{P(Z)}{p(Z)}h'(Z)\mathbb I_{(l,x]}(Z)-\frac{\overline{P}(Z)}{p(Z)}h'(Z)\mathbb I_{(x,u)}(Z)] = h(x) - \mathcal{P}h
\]
so that taking expectation in \eqref{eqapp:integraloperator-on-kernel} yields
\[
\E[-\mathcal S_pK_{p,w}^{i,j}(x,Z)h'(Z)]=\E[K_{p,w}^{i,j+1}(x,Z)h'(Z)]+ M^{i,j}_{p,w}(x) (h(x)-\mathcal{P}h)
\]

Putting this in \eqref{eqapp:SIBP} is \eqref{eq:master-vertical-move}.
\end{proof}
\begin{proof}[Proof of Lemma~\ref{lma:Kernel calculus -- differentiation}]
Under the stated assumption, $x\mapsto \E[K^{i,j}_{p,w}(x,Z)h(Z)]$ is absolutely continuous. Write the above as 
\begin{equation*}
    \E[K^{i,j}_{p,w}(x,Z)h(Z)]=(-1)^{j}\overline{m}^{(i)}_w(x) \int_l^x\underline P_j(t)h(t) \mathrm{d}t-\underline m_w^{(i)}(x)\int_x^u\overline{P}_j(t)h(t) \mathrm{d}t.
\end{equation*}
From the product rule of $\mathcal{A}_{p,w}$ (see \cite[Equation (15)]{ernst_reinert_swan_2020}) 
\[
\mathcal{A}_{p,w}(\phi \, \psi) = \mathcal{A}(\phi)\psi + w\,\phi \,\psi '.
\]
with $\phi = \overline{m}_w^{(i)}, \underline m_w^{(i)}$ and $\psi$ the corresponding integral, we have that
\begin{align*}
\mathcal{A}_{p,w}\Bigl(\E[K^{i,j}_{p,w}(x,Z)h(Z)]\Bigr)=&\,w(x) \Bigl(\E[K^{i+1,j}_{p,w}(x,Z)h(Z)]-M^{i,j-1}_{p,w}(x)h(x)\Bigr)\\
&+s_w(x)\E[K^{i,j}_{p,w}(x,Z)h(Z)]
\end{align*}
Thus $x\mapsto \E[K^{i+1,j}_{p,w}(x,Z)h(Z)]-M^{i,j-1}_{p,w}(x)h(x)$ is the version of the derivative, from Remark~\ref{rk:abscontderiv}, of $x\mapsto \E[K^{i,j}_{p,w}(x,Z)h(Z)$.
\end{proof}

\begin{proof}[Proof of Lemma~\ref{lem:no-other-orders}(a)]
Recall the notations from Assumption \ref{ass:density-support}. Let $J=(e_0,e_1)$ and let $\mathcal{Q}$ be the space of polynomials of degree at
most $n+1$. This is a space of dimension $n+2$, contained in $L^1(p)$ by assumption;
its quotient $\mathcal Q/\mathbb R$ by the constants has dimension $n+1$. By {\eqref{eq:caninv}}--{\eqref{eq:canonical-weighted-stein-solution}},
the map $h\mapsto f_{p,w,h}|_J$ is linear on $\mathcal{Q}$. If $h, g \in \mathcal Q$ are such that 
$f_{p,w,h}=f_{p,w,g}$ on $J$, then {\eqref{eq:w-stein-equation}} gives $h-\mathcal Ph=g-\mathcal Pg$ almost
everywhere on $J$, hence $h-g$ is constant on $\mathbb R$. The map is therefore injective on
$\mathcal{Q}$ modulo constants. Suppose that
$\|f^{(n)}_{p,w,h}\|=0$ for every $h\in\mathcal{Q}$. By {Notation~\ref{not:derivative-norms}},
$f^{(n-1)}_{p,w,h}$ is then locally absolutely continuous with vanishing
derivative on $J$, so $f_{p,w,h}|_J$ is a polynomial of degree at most $n-1$
(and vanishes if $n=0$). An injective linear map would thus send a space of
dimension $n+1$ into one of dimension $n$, a contradiction.
\end{proof}
\begin{proof}[Proof of Lemma~\ref{lem:no-other-orders}(b)]
We adapt \cite[Proof of Theorem~1.1(ii)]{gaunt2024steinfactorssteinsmethod}. First note that $h_0\in L^1(p)$, $h_0^{(k)}=\operatorname{sgn}(\cdot-x_0)$, and
$h_0\in C^{k-1}$ if $k\ge 1$. Let $V\subset \mathcal I_p$ be an open interval containing
$x_0$ on which $a_w$ and $1/w$ are of class $C^k$. By {\eqref{eq:w-stein-equation}},
\[
  f'_{p,w,h_0}=a_wf_{p,w,h_0}+\frac{h_0-\mathcal{P}h_0}{w}\qquad\text{a.e.\ on } V .
\]
Since $f_{p,w,h_0}$ is continuous, induction on $j$ in this identity gives $f_{p,w,h_0}\in C^{j}(V)$
for $0\le j\le k$. On $V\setminus\{x_0\}$, $h_0-\mathcal{P}h_0$ is a polynomial on each side of
$x_0$, so $f_{p,w,h_0}$ is $k+1$ times differentiable there and, by the Leibniz rule,
\[
  f^{(k+1)}_{p,w,h_0}=\frac{h_0^{(k)}}{w}
  +\sum_{r=0}^{k-1}\binom{k}{r}\Bigl(\frac1w\Bigr)^{(k-r)} (h_0-\mathcal{P}h_0)^{(r)}
  +\sum_{j=0}^{k}\binom{k}{j}a_w^{(k-j)}f_{p,w,h_0}^{(j)} .
\]
Both sums are continuous on $V$. Hence
\[
  f^{(k+1)}_{p,w,h_0}(x_0+)-f^{(k+1)}_{p,w,h_0}(x_0-)
  =\frac{h_0^{(k)}(x_0+)-h_0^{(k)}(x_0-)}{w(x_0)}=\frac{2}{w(x_0)}\neq 0 ,
\]
which extends {\eqref{eq:general-kolmogorov-derivative-jumps}}. Thus $f^{(k+1)}_{p,w,h_0}$ has no continuous version near
$x_0$. Since $n-1\ge k+1$, $f_{p,w,h_0}\notin C^{n-1}(V)$ hence
$\|f^{(n)}_{p,w,h_0}\|=\infty$ by {Notation~\ref{not:derivative-norms}}.
\end{proof}

\section{Proofs from Section~\ref{sec:closed-kernel-calculi}}
\label{sn:proofsectionsecin}

\paragraph{Low order}
\begin{proof}[Proof of Theorem \ref{thm:universal-low-order-calculus}]
The identities follow by applying
\eqref{eq:master-horizontal-move}--\eqref{eq:master-vertical-move}
from \((0,0)\), together with
\[
 M_{p,w}^{0,1}=-\frac{\tau_p}{w},
 \qquad
 M_{p,w}^{1,1}=-\left(\frac{\tau_p}{w}\right)',
 \qquad
 M_{p,w}^{2,1}-M_{p,w}^{1,0}
 =-\left(\frac{\tau_p}{w}\right)''.
\]
The centred form follows from \eqref{eq:low-order-mean-21}.
\end{proof}
\begin{proof}[Proof of Proposition \ref{prop:stein-solution-low-order-factors}]
Let
\[
 c:=\frac12\left(
 \operatorname*{ess\,sup}h+\operatorname*{ess\,inf}h
 \right),
\]
when $\operatorname{osc}(h)<\infty$,
so that
\(\|h-c\|=\tfrac12\operatorname{osc}(h)\).
Since \(K_{p,w}^{0,0}(x,\cdot)\) is centred,
\(f_{p,w,h}=f_{p,w,h-c}\), and the first identity in
\eqref{eq:low-order-abs-mean-01} gives
\begin{align*}
 |f_{p,w,h}(x)|
 &=
 \left|
 \mathbb E\left[
 K_{p,w}^{0,0}(x,Z)(h(Z)-c)
 \right]
 \right|\\
 &\leq
 \mathbb E\left[|K_{p,w}^{0,0}(x,Z)|\right]
 \frac{\operatorname{osc}(h)}2\\
 &=
 \frac{P(x)\overline P(x)}
      {{|w(x)|}p(x)}
 \operatorname{osc}(h).
\end{align*}
On the other hand,
under the assumptions for \eqref{eq:ker01},
that representation and the second identity in
\eqref{eq:low-order-abs-mean-01} give
\[
 |f_{p,w,h}(x)|
 \leq
 \frac{\tau_p(x)}{|w(x)|}
 \|h^{(1)}\|.
\]
\end{proof}

\begin{proof}[Proof of Proposition \ref{prop:first-derivative-factors}]
Representation~\eqref{eq:ker10}, together with
\eqref{eq:absolute-mean-K10-simplified} and \(|h(Z)-h(x)|\leq\operatorname{osc}(h)\), gives
\[
 |f_{p,w,h}'(x)|
 \leq
 \frac{1}{w(x)}\operatorname{osc}(h).
\]
The second estimate follows from \eqref{eq:ker11} and the definition of
\(U_{p,w}^{1,1}\), while the third follows similarly from
\eqref{eq:ker12}.
\end{proof}

\begin{lemma}
\label{lem:symmax-U21}
Consider the symmetrized Maxwell density $p_\sigma$ introduced in Example~\ref{subsec:symmetrized-maxwell}. Write \(a^-=\max\{-a,0\}\).
Then \(U^{2,1}_{p_\sigma,\tau_{p_\sigma}}\) is even, \(\underline m_{p_\sigma,\tau_{p_\sigma}}''>0\) on \((0,\infty)\),
\(\overline m_{p_\sigma,\tau_{p_\sigma}}''>0\) on \([\sqrt2\,\sigma,\infty)\), and for \(x>0\),
\[
 U^{2,1}_{p_\sigma,\tau_{p_\sigma}}(x)
 =
 \frac1{\tau_{p_\sigma}(x)}+2\bigl(\overline m''_{p_\sigma,\tau_{p_\sigma}}(x)\bigr)^-\underline P_2^\sigma(x)
 \leq
 \frac1{\sigma^2}.
\]
\end{lemma}

\begin{proof}
By scaling, \(U^{2,1}_{p_\sigma,\tau_{p_\sigma}}(x)=\sigma^{-2}U^{2,1}_{p_1,\tau_{p_1}}(x/\sigma)\). We then take \(\sigma=1\). Moreover, by symmetry,  \(U^{2,1}_{p_1,\tau_{p_1}}\) is even. Therefore we only consider \(x>0\).

Recall from \eqref{eq:stein-kernel-mills-identities}
\[
 \underline m_{p_1,\tau_{p_1}}'=\frac{\underline P_2^1}{\tau_{p_1}^2p_1},
 \qquad
 \overline m_{p_1,\tau_{p_1}}'=-\frac{\overline P_2^1}{\tau_{p_1}^2p_1}.
\]
Since \(p_1'/p_1=2/x-x\) and \(1/\tau_{p_1}=x^2/(x^2+2)\), we have \((1/(\tau_{p_1}^2p_1))'=k\,1/(\tau_{p_1}^2p_1)\) with
\(k(x)=p_1'/p_1+2x/\tau_{p_1}=(x^4+4)/(x(x^2+2))\), so that
\[
 \underline m_{p_1,\tau_{p_1}}''=\frac{\bigl(P^1+k\,\underline P^1_2\bigr)}{\tau_{p_1}^2p_1}>0,
 \qquad
 \overline m_{p_1,\tau_{p_1}}''=\frac{\bigl(\overline P^1-k\,\overline P_2^1\bigr)}{\tau_{p_1}^2p_1}
 =\frac{x\,N(x)}{(x^2+2)^3},
\]
where \(N(x)=x(x^4+x^2+6)R(x)-(x^4-2x^2+8)\), using \(\overline P^1=(R+x)\varphi\), with \(R=\overline\Phi/\varphi\). .
By \eqref{eq:low-order-mean-21} with \(w=\tau_p\),
\(\overline m_{p_1,\tau_{p_1}}''\,\underline P_2^1+\underline m_{p_1,\tau_{p_1}}''\,\overline P_2^1=-M^{2,1}_{p_1,\tau_{p_1}}=1/\tau_{p_1}\). Now since \(\underline m_{p_1,\tau_{p_1}}''>0\),
\[
 U^{2,1}_{p_1,\tau_{p_1}}=|\overline m_{p_1,\tau_{p_1}}''|\,\underline P_2^1+\underline m_{p_1,\tau_{p_1}}''\,\overline P_2^1
 =\frac1{\tau_{p_1}}+2(\overline m_{p_1,\tau_{p_1}}'')^-\underline P_2^1 .
\]
Gordon's inequality \cite{10.1214/aoms/1177731721}, \(R(x)>x/(x^2+1)\) yields \(N(x)>(2x^4-8)/(x^2+1)\), so \(\overline m_{p_1,\tau_{p_1}}''>0\) on \([\sqrt2,\infty)\). On this interval, $$U^{2,1}_{\tau_{p_1},p_1} = x^2/(x^2+2) \le 1.$$ For \(0<x<\sqrt2\), the desired inequality holds when $(\overline m_{p_1,\tau_{p_1}}'')^-\underline P_2^1 \le 1/(x^2+2)$. Using \(\underline P_2^1\leq x+\sqrt{2/\pi}<x+4/5\), we have
\[
 (\overline m_{p_1,\tau_{p_1}}'')^-\underline P_2^1
 <
 \frac{x(x+4/5)(8-2x^4)}{(x^2+1)(x^2+2)^3}
\]
Comparing to $1/(x^2+2)$, we have
\begin{align*}
    \frac{x(x+4/5)(8-2x^4)}{(x^2+1)(x^2+2)^3} \le \frac{1}{x^2+2} &\Leftrightarrow (x^2+1)(x^2+2)^2-x(x+4/5)(8-2x^4)\ge 0 \\
    &\Leftrightarrow 3x^6+\frac85x^5+5x^4-\frac{32}5x+4\ge 0
\end{align*}
The minimum of \(5x^4-\tfrac{32}5x+4\) on \([0,\infty)\) is attained at \(x_*=(8/25)^{1/3}\) and equals \(4-\tfrac{24}5x_*>0\). Therefore
\(U^{2,1}_{p_1,\tau_{p_1}}(x)\leq x^2/(x^2+2)+2/(x^2+2)=1\).
\end{proof}

\paragraph{Kolmogorov}

\begin{proof}[Proof of Theorem~\ref{thm:kolmo-rep}]
Both branches of \eqref{eq:kolmo-solution} are continuous and agree at $z$, so
$f_{p,w,h_z}$ is continuous on $(l,u)$.
Fix $n\geq1$ and suppose that
$a_w,1/w\in C^{n-1}((l,u))$.
Since $p$ is continuous on $(l,u)$, the weighted Mills ratios
are differentiable there and satisfy
\[
    \underline m_{p,w}'
    =w^{-1}+a_w\underline m_{p,w},
    \qquad
    \overline m_{p,w}'
    =-w^{-1}+a_w\overline m_{p,w}.
\]
These identities imply inductively that both weighted Mills
ratios belong to $C^n((l,u))$.  Since $P(z)$ and $\overline P(z)$ do not depend on $x$,
differentiating the two branches of
\eqref{eq:kolmo-solution} proves
\eqref{eq:general-kolmogorov-all-derivatives} and its
counterparts at orders $1\leq k\leq n$.

Continuity of the Mills-ratio derivatives gives the finite
one-sided limits
\[
    f_{p,w,h_z}^{(k)}(z-)
    =\overline P(z)\underline m_{p,w}^{(k)}(z),
    \qquad
    f_{p,w,h_z}^{(k)}(z+)
    =P(z)\overline m_{p,w}^{(k)}(z).
\]

This proves \eqref{eq:general-kolmogorov-derivative-jumps}
at every order $1\leq k\leq n$. At order one,
$B_{p,w,1}=1/w$, giving the stated jump.
\end{proof}

\begin{proof}[Proof of Corollary~\ref{corr:kolmo-recursion}]
    Differentiating and using
\eqref{eq:generalized-laplace-coefficients-An} and
\eqref{eq:generalized-laplace-coefficients-Bn} gives, for
$1\leq k\leq n$,
\[
    \underline m_{p,w}^{(k)}
    =A_{p,w,k}\underline m_{p,w}+B_{p,w,k},
    \qquad
    \overline m_{p,w}^{(k)}
    =A_{p,w,k}\overline m_{p,w}-B_{p,w,k},
\]
which is \eqref{eq:weighted-mills-recursion}. The representation of $f^{(n)}_{p,w,h}$ directly follow from \ref{eq:general-kolmogorov-all-derivatives}.
Now, using
$P\,\overline m_{p,w}=\overline P\,\underline m_{p,w}$,
we obtain
\begin{align*}
    &f_{p,w,h_z}^{(k)}(z+)
    -f_{p,w,h_z}^{(k)}(z-)\\ 
    &\quad=
    A_{p,w,k}(z)
    \bigl(
        P(z)\overline m_{p,w}(z)
        -\overline P(z)\underline m_{p,w}(z)
    \bigr)
    -B_{p,w,k}(z)\bigl(P(z)+\overline P(z)\bigr)\\
    &\quad=-B_{p,w,k}(z).
\end{align*}
\end{proof}

\begin{proof}[Proof of Proposition \ref{prop:score-weighted-kolmogorov-factor}]
By concavity of $\log p$,
\[
    p(t)\leq p(x)\exp\bigl(\rho_p(x)(t-x)\bigr),
    \qquad t,x\in(l,u).
\]
If $\rho_p(x)>0$, integration over $(l,x)$ gives
\[
    P(x)\leq\frac{p(x)}{\rho_p(x)}.
\]
If $\rho_p(x)<0$, integration over $(x,u)$ gives
\[
    \overline P(x)\leq\frac{p(x)}{|\rho_p(x)|}.
\]
Combining these inequalities with
\eqref{eq:general-kolmogorov-zero-envelope} proves the
sign-dependent bounds and hence the bound involving
$\max\{P(x),\overline P(x)\}$.
Taking $w=1$ and multiplying by $|\rho_p(x)|$ gives the
last assertion when $\rho_p(x)\neq0$; at a zero of the score,
it is immediate.
\end{proof}

\paragraph{Integrated Pearson}

\begin{proof}[Proof of Lemma~\ref{lem:admissible}]
Since $\underline P_{k+1}(x)=\mathbb{E}[(x-Z)_+^{k}]/k!$ and
$\overline P_{k+1}(x)=\mathbb{E}[(Z-x)_+^{k}]/k!$, (ii) and (iii) are equivalent, and
the value of $x$ is immaterial. For (i)$\Leftrightarrow$(ii), note first that
$q_j=j(1-(j-1)\kappa_2)>0$ for all $j$ when $\kappa_2\le0$. If $\kappa_2<0$ then
$\tau_p$ is a concave quadratic, nonnegative only on a bounded interval, so $\mathcal I_p$
is bounded and all moments are finite. If $\kappa_2=0$, then $(\log\tau_pp)'=(\mu-x)/\tau_p$
with $\tau_p$ affine, giving exponential (or Gaussian) tails, and again all moments
are finite. If $\kappa_2>0$, at least one endpoint is infinite: the positivity
interval is a connected component of $\{\tau_p>0\}$;
see \cite[Corollary~2.1]{Integrated_pearson}.
On each infinite tail,
\[
 \frac{(\tau_pp)'(x)}{\tau_p(x)p(x)}
 =-\frac1{\kappa_2x}+O(|x|^{-2}).
\]
Integrating the remainder, which is integrable at infinity, gives
\[
 \tau_p(x)p(x)
 =d_\pm |x|^{-1/\kappa_2}
   \bigl(1+O(|x|^{-1})\bigr),
 \qquad
 p(x)=c_\pm |x|^{-2-1/\kappa_2}
   \bigl(1+O(|x|^{-1})\bigr),
\]
with positive constants on the corresponding tails.
Thus $\mathbb E|Z|^r<\infty$ if and only if
$r<1+1/\kappa_2$ for $r\geq0$, with divergence also at equality.
For the integer $k$, this is equivalent to $q_j>0$ for every
$1\leq j\leq k$.

For the boundary statement, fix $1\le j\le k$; since $j$ is admissible, $\mathbb{E}|Z|^j<\infty$ by (ii).
We treat the endpoint $l$, the endpoint $u$ being symmetric.
For $j=1$ the claim follows from $\tau_p(x)p(x)=\int_l^x(\mu-t)p(t)\,dt\to0$ and $(\mu-x)P(x)\to0$;
when $l=-\infty$, the latter follows from $|x|P(x)\le\mathbb{E}[|Z|\mathbf 1_{\{Z<x\}}]$ for $x<0$.
Let now $j\ge2$. If $l$ is finite, $\lambda_j$ and $\tau_p$ are bounded near $l$, while $\underline P_j(x),\underline P_{j-1}(x)\to0$.
If $l=-\infty$, then for $x<0$ we have $|x|\le|Z|$ and $0<x-Z\le|Z|$ on $\{Z<x\}$. Hence
\[
|x|\underline P_j(x)\le\frac{\mathbb{E}\big[|Z|^{j}\mathbf 1_{\{Z<x\}}\big]}{(j-1)!},
\qquad
x^2\underline P_{j-1}(x)\le\frac{\mathbb{E}\big[|Z|^{j}\mathbf 1_{\{Z<x\}}\big]}{(j-2)!}.
\]
Since $\lambda_j(x)=O(|x|)$ and $\tau_p(x)=O(x^2)$, both $\lambda_j\underline P_j$ and $\tau_p\underline P_{j-1}$ tend to $0$
as $x\downarrow-\infty$ by dominated convergence.
\end{proof}

\begin{proof}[Proof of Lemma \ref{lem:integrated-pearson-tail-recursions}]
For \(k=1\), these are the identities
\[
 \underline P_2=-sP+\tau_pp,
 \qquad
 \overline P_2=s\overline P+\tau_pp,
\]
already obtained for a general density.
For the induction step, suppose that \(k+1\) is admissible.
Integrating the lower-tail identity from \(l\) to \(x\) and
integrating by parts gives
\[
 q_k\underline P_{k+2}
 =
 -(\lambda_k+\tau_p')\underline P_{k+1}
 +\tau_p\underline P_k
 +(\lambda_k'+\tau_p'')\underline P_{k+2}.
\]
The boundary terms vanish by Lemma~\ref{lem:admissible}.
Since
\[
 \lambda_k+\tau_p'=\lambda_{k+1},
 \qquad
 q_k-\lambda_k'-\tau_p''=q_{k+1},
\]
this is the lower-tail identity at order \(k+1\).
The upper-tail identity follows in the same way by integrating
from \(x\) to \(u\).
\end{proof}

\begin{proof}[Proof of Theorem~\ref{thm:integrated-pearson-kernel-representations}]
We prove the cancellation identities
\[
 M_{p,\tau_p}^{n,n-1}=-\frac1{\tau_p},
 \qquad
 M_{p,\tau_p}^{n,n}=0.
\]
The three representations then follow from
Corollary~\ref{cor:closed-three-neighbour-calculus}. Recall the recursion
\begin{equation}
\label{eq:repeated-M-derivative-recursion}
 (M_{p,w}^{i,j})'
 =
 M_{p,w}^{i+1,j}-M_{p,w}^{i,j-1},
\end{equation}
and the initial identities
\[
 M_{p,w}^{0,0}=0,
 \qquad
 M_{p,w}^{1,0}=-\frac1w,
 \qquad
 M_{p,w}^{1,-1}
 =
 -\left(\frac1w\right)'+\frac{\rho_p}{w}.
\]
Suppose inductively that
\[
 M_{p,\tau_p}^{r,r-1}=-\frac1{\tau_p},
 \qquad 1\leq r\leq k,
\]
where \(k\) is admissible.
Repeated use of
\eqref{eq:repeated-M-derivative-recursion} gives
\[
 M_{p,\tau_p}^{k,k-2}
 =
 -k\left(\frac1{\tau_p}\right)'
 +\frac{\rho_p}{\tau_p}.
\]
The integrated-tail recursions of
Lemma~\ref{lem:integrated-pearson-tail-recursions}
therefore yield
\begin{align*}
 q_k M_{p,\tau_p}^{k,k}
 &=
 \lambda_k M_{p,\tau_p}^{k,k-1}
 +\tau_p M_{p,\tau_p}^{k,k-2}\\
 &=
 -\frac{\lambda_k}{\tau_p}
 -k\tau_p\left(\frac1{\tau_p}\right)'
 +\rho_p\\
 &=
 \frac{-\lambda_k+k\tau_p'+\tau_p\rho_p}{\tau_p}
 =0,
\end{align*}
where the last equality follows from
\[
 \lambda_k=\mu-x+(k-1)\tau_p',
 \qquad
 \tau_p\rho_p=\mu-x-\tau_p'.
\]
Since \(q_k>0\), we obtain \(M_{p,\tau_p}^{k,k}=0\).
The derivative recursion now gives
\[
 M_{p,\tau_p}^{k+1,k}
 =
 (M_{p,\tau_p}^{k,k})'
 +M_{p,\tau_p}^{k,k-1}
 =
 -\frac1{\tau_p},
\]
which completes the induction through the admissible orders.

Thus the cancellation conditions of
Corollary~\ref{cor:closed-three-neighbour-calculus}
hold at all required indices. That corollary gives
\eqref{eq:IP-representation-n-n-1} and
\eqref{eq:IP-representation-n-n}, and also
\eqref{eq:IP-representation-n-n+1} when \(n+1\) is
admissible, under the stated assumptions on the test function.
\end{proof}

The following proposition precises Remark~\ref{rk:IP-other-weights}.

\begin{proposition}\label{prop:summary-M-tau-over-w}
Let \(Z\sim\mathrm{IP}\), let \(n\geq1\) be admissible. If \((\tau_p/w)''\) is constant, then
\[
 M_{p,w}^{n,n-1}(x)=-\frac1{w(x)}+C_{n,w},
 \qquad
 \bigl(M_{p,w}^{n,n-1}\bigr)'(x)=-\left(\frac1w\right)'(x),
\]
where \(C_{1,w}=0\) and
\[
 C_{n+1,w}=C_{n,w}+\frac1{q_n}
 \left[
 -n\left(\frac{\tau_p}{w}\right)''
 -C_{n,w}(1-(n-1)\tau_p'')
 \right].
\]
If \((\tau_p/w)'\) is constant, then
\[
 M_{p,w}^{n,n-1}(x)=-\frac1{w(x)},
 \qquad
 M_{p,w}^{n,n}(x)
 =-\frac{n}{q_n}\left(\frac{\tau_p}{w}\right)'.
\]
In particular, if \(w=\tau_p\) , then
\[
 M_{p,\tau_p}^{n,n-1}=-\frac1{\tau_p},
 \qquad
 M_{p,\tau_p}^{n,n}=0.
\]
\end{proposition}

\begin{proof}
Write \(r:=\tau_p/w\). The induction starts from
\(M_{p,w}^{1,0}=-1/w\), so \(C_{1,w}=0\).
Suppose that \(M_{p,w}^{j,j-1}=-1/w+C_{j,w}\), with constants
\(C_{j,w}\), for \(1\leq j\leq n\).
The derivative recursion gives
\[
 M_{p,w}^{n,n-2}
 =
 -n\left(\frac1w\right)'+\frac{\rho_p}{w}.
\]
Combining this with the integrated-tail recursions yields
\begin{align*}
 q_nM_{p,w}^{n,n}
 &=
 \lambda_nM_{p,w}^{n,n-1}
 +\tau_pM_{p,w}^{n,n-2}\\
 &=
 -nr'+\lambda_nC_{n,w},
\end{align*}
where we used \(\lambda_n=n\tau_p'+\tau_p\rho_p\).
Since \(\lambda_n'=-1+(n-1)\tau_p''\), differentiation and
\eqref{eq:repeated-M-derivative-recursion} give
\[
 M_{p,w}^{n+1,n}
 =
 -\frac1w+C_{n,w}
 +\frac{-nr''-C_{n,w}(1-(n-1)\tau_p'')}{q_n}.
\]
If \(r''\) is constant, the last two terms define the stated
constant \(C_{n+1,w}\), completing the induction through the
admissible orders. If \(r'\) is constant, then \(r''=0\) and
\(C_{1,w}=0\) imply \(C_{n,w}=0\) at every such order.
The formula for \(M_{p,w}^{n,n}\) follows from the identity above.
Finally, \(w=\tau_p\) gives \(r=1\).
\end{proof}

\begin{proof}[Proof of Proposition \ref{prop:equivalence-m-tau-P}]
For \(k=0\), these are the definitions of the weighted Mills
ratios. Write \(C_k:=\prod_{i=1}^kq_i\), with \(C_0=1\),
and suppose that the identities hold at order \(k-1\), where
\(k\geq1\) is admissible. Since
\[
 \frac{(p\tau_p^k)'}{p\tau_p^k}
 =
 \rho_p+k\frac{\tau_p'}{\tau_p}
 =
 \frac{\lambda_k}{\tau_p},
\]
differentiating the lower-tail identity gives
\[
 p\tau_p^{k+1}\underline m_{p,\tau_p}^{(k)}
 =
 C_{k-1}
 \left(\tau_p\underline P_{k-1}-\lambda_k\underline P_k\right)
 =
 C_k\underline P_{k+1}.
\]
Similarly, since \(\overline P_k'=-\overline P_{k-1}\),
\[
 p\tau_p^{k+1}\overline m_{p,\tau_p}^{(k)}
 =
 (-1)^kC_{k-1}
 \left(\tau_p\overline P_{k-1}+\lambda_k\overline P_k\right)
 =
 (-1)^kC_k\overline P_{k+1}.
\]
Both final equalities follow from
Lemma~\ref{lem:integrated-pearson-tail-recursions}.
This completes the induction.
\end{proof}

\begin{proof}[Proof of Proposition \ref{prop:IP-absolute-expectations}]
Corollary~\ref{cor:IP-Mills-monotonicity} shows that the
lower- and upper-neighbour kernels are non-positive.
For the lower neighbour, the claim therefore follows from
\(M_{p,\tau_p}^{n,n-1}=-1/\tau_p\).

For the other two identities, first note that the tail recursions
give, for every admissible \(k\geq0\),
\begin{equation}
\label{eq:wronskian equality}
 \left(\prod_{i=1}^kq_i\right)
 \left(
 \underline P_k\overline P_{k+1}
 +
 \overline P_k\underline P_{k+1}
 \right)
 =
 \tau_p^kp.
\end{equation}
Indeed, the identity holds at \(k=0\), and substituting the
recursions at order \(k\) cancels the terms involving
\(\lambda_k\), leaving \(\tau_p\) times the expression at
order \(k-1\).

By Proposition~\ref{prop:equivalence-m-tau-P},
\[
 U_{p,\tau_p}^{n,n}
 =
 \frac{
 2\left(\prod_{i=1}^nq_i\right)
 \underline P_{n+1}\overline P_{n+1}
 }{p\tau_p^{n+1}}.
\]
Applying \eqref{eq:wronskian equality} at order \(n\)
gives the expression in terms of
\(\underline R_n+\overline R_n\).
Similarly,
\[
 U_{p,\tau_p}^{n,n+1}
 =
 \frac{\prod_{i=1}^nq_i}{p\tau_p^{n+1}}
 \left(
 \overline P_{n+1}\underline P_{n+2}
 +
 \underline P_{n+1}\overline P_{n+2}
 \right),
\]
which equals \(1/q_{n+1}\) by
\eqref{eq:wronskian equality} at order \(n+1\).
\end{proof}

\begin{proof}[Proof of Theorem~\ref{thm:IP-diagonal}]
Write $U_n:=U^{n,n}_{p,\tau_p}$ and recall $\underline R_n=\tau_p\underline P_n/\underline P_{n+1}$, $\overline R_n=\tau_p\overline P_n/\overline P_{n+1}$ for admissible $n$, so that $U_n=2/(\underline R_n+\overline R_n)$ by Proposition~\ref{prop:IP-absolute-expectations}. By \cite[Proposition~2.1]{Integrated_pearson}, $p$ is smooth and $\tau_p$ is positive on $\mathcal I_p$, so that $\underline R_k$, $\overline R_k$ and $U_n$ are smooth there. Moreover $\tau_p$ vanishes at every finite endpoint of $\mathcal I_p$.

The envelope \eqref{eq:IP-global-envelope} follows from Lemma~\ref{lem:integrated-pearson-tail-recursions}. Indeed, one has
\begin{align*}
    \underline R_n&=\lambda_{n+1}+q_{n+1}\underline P_{n+2}/\underline P_{n+1}>\lambda_{n+1}, \\ 
    \overline R_n&=-\lambda_{n+1}+q_{n+1}\overline P_{n+2}/\overline P_{n+1}>-\lambda_{n+1},
\end{align*}
and since $\underline R_n,\overline R_n>0$, $\underline R_n+\overline R_n>|\lambda_{n+1}|$, which is \eqref{eq:IP-global-envelope}.

Now to obtain the results about $\sup U_n$, we shall study the behaviour of $U_n$ at the endpoints of the support $(l,u)$ then characterize the critical points before proving claims $(i)$, $(ii)$ and $(iii)$.

Let us first consider a finite endpoint $l>-\infty$ (the endpoint $u$ is symmetric). 
Remark that $\tau_p'^2-4\kappa_2\tau_p \equiv D_{\tau_p}$, the discriminant of $\tau_p$. Since $\tau_p(l)=0$ and $\tau_p >0$ on $\mathcal I_p$, we get $\tau_p'(l)=\sqrt{ D_{\tau_p}}$, then $\lambda_{n+1}(l)=\mu-l+n\sqrt{D_{\tau_p}}>0$.

We now show that $U_n(x)\to2/\lambda_{n+1}(l)$ as $x\downarrow l$. First, $\overline R_n\to0$, because $\tau_p(l)=0$, $\overline P_{n+1}(l)=\mathbb E[(Z-l)^n]/n!>0$, and either $\overline P_n(l)<\infty$ ($n\ge1$) or $\tau_p\overline P_0=\tau_pp\to0$ ($n=0$). Next, we prove by induction that $\underline R_k\to\lambda_{k+1}(l)$ for $0\le k\le n$. Both $\tau_p\underline P_k$ and $\underline P_{k+1}$ vanish at $l$:
for $k=0$, use $\tau_pp\to0$; for $k\geq1$,
use $\underline P_k\to0$ and boundedness of $\tau_p$ near the finite
endpoint. Also, $\underline P_{k+1}'=\underline P_k$. By L'H\^opital's rule it suffices to compute the limit of
\[
\frac{(\tau_p\underline P_0)'}{\underline P_0}=\mu-x,\qquad
\frac{(\tau_p\underline P_k)'}{\underline P_k}=\tau_p'+\underline R_{k-1}\ \ (k\ge1).
\]
These tend to $\mu-l=\lambda_1(l)$ and $\tau_p'(l)+\lambda_k(l)=\lambda_{k+1}(l)$ respectively. Hence $U_n(x) = 2/ (\overline{R}_n(x) + \underline R_n(x)) \to 2/\lambda_{n+1}(l)$ as $x \downarrow l$.

Second, we consider the infinite endpoint $u=+\infty$ (the case $l=-\infty$ is symmetric). We show that $U_n\to0$. If $\kappa_2\le0$, every order is admissible and $1-2n\kappa_2>0$, so this follows from
\eqref{eq:IP-global-envelope} since $\lambda_{n+1}(x)=\mu+n\kappa_1-(1-2n\kappa_2)x$.
Otherwise $\kappa_2>0$, and we consider $\overline{R}_n = (\tau_p/x)(x\overline{P}_n/\overline{P}_{n+1})$. About the first factor we have \[\tau_p/x \sim \kappa_2x \to \infty.\]
Now, we show by induction that $x\overline P_k/\overline P_{k+1}\to1/\kappa_2+1-k$ for all $0\le k\le n$; note that $1/\kappa_2+1-k>0$ because $k$ is admissible (Remark~\ref{rem:admissible-1/kappa2}). 
For $1\leq k\leq n$, admissibility gives
\[
 0\leq x\overline P_k(x)
 \leq
 \frac{\mathbb E\bigl[|Z|^k\mathbf 1_{\{Z>x\}}\bigr]}
      {(k-1)!}
 \longrightarrow0.
\]
Also $\overline P_{k+1}(x)\to0$, by dominated convergence and
$\mathbb E|Z|^k<\infty$.
For $k=0$, the tail asymptotic in the proof of
Lemma~\ref{lem:admissible} gives $xp(x)\to0$, and
$\overline P_1(x)\to0$.
 For $k=0$, since $xp'/p=x(\mu-x-\tau_p')/\tau_p\to-(1/\kappa_2+2)$, L'H\^opital's rule gives
\[
\frac{\overline P_1}{xp}\sim\frac{-p}{p+xp'}=\frac{-1}{1+xp'/p}\to\frac{1}{1/\kappa_2+1}.
\]
For $k\ge1$, since $\overline P_k'=-\overline P_{k-1}$, by induction,
\[
\frac{\overline P_{k+1}}{x\overline P_k}\sim\frac{-\overline P_k}{\overline P_k-x\overline P_{k-1}}=\frac{1}{x\overline P_{k-1}/\overline P_k-1}\to\frac{1}{1/\kappa_2+1-k}.
\]
Hence $\overline R_n=(\tau_p/x)\,(x\overline P_n/\overline P_{n+1})\sim\kappa_2(1/\kappa_2+1-n)\,x\to+\infty$, and $U_n<2/\overline R_n\to0$.

Now we will characterize the possible interior critical points of $U_n$. For $n\ge1$, the quotient rule, Lemma~\ref{lem:integrated-pearson-tail-recursions} with $k=n$ and $\lambda_n+\tau_p'=\lambda_{n+1}$ give
\[
\tau_p\underline R_n'=q_n\tau_p+\lambda_{n+1}\underline R_n-\underline R_n^2,\qquad
\tau_p\overline R_n'=-q_n\tau_p+\lambda_{n+1}\overline R_n+\overline R_n^{\,2}.
\]
For $n=0$ the same equations, with $q_0:=0$, follow from $(\tau_pp)'=(\mu-x)p$. Adding and subtracting them, with $\lambda_{n+1}'=-1+2n\kappa_2$,
\[
\tau_pU_n'=U_n (\underline R_n-\overline R_n-\lambda_{n+1})
\]
At a critical point $x_0$ of $U_n$ we have $\underline R_n(x_0)-\overline R_n(x_0)=\lambda_{n+1}(x_0)$. Denote $S:=\underline R_n- \overline{R}_n - \lambda_{n+1}$ and compute \[U_n''=U_n'S/\tau_p +U_nS'/\tau_p - U_n S \tau_p'/\tau_p^2, \qquad \text{and} \qquad
\tau_pS'=\theta_n\tau_p+\lambda_{n+1}(\underline R_n-\overline R_n)-\underline R_n^2-\overline R_n^{\,2}.\] 
Together with $(\underline R_n+\overline R_n)^2=4/U_n^2$, this gives $2\tau_p(x_0)S'(x_0)=g_n(x_0)-4/U_n(x_0)^2$. Hence, by the second order condition at a critical point, 
\begin{equation}\label{eq:IP-crit}
U_n(x_0)^2g_n(x_0)\le4\ \text{ at a local maximizer},\qquad
U_n(x_0)^2g_n(x_0)\ge4\ \text{ at a local minimizer};
\end{equation}
in particular $g_n>0$ at every local minimizer, and if $U_n$ is constant on an open interval, then $g_n$ is too and $g_n=4/U_n^2$.

Recall $g_n(x)=2\theta_n\tau_p(x)+\lambda_{n+1}(x)^2$ is a quadratic, hence its Taylor expansion at $\mu$ is exact, and a direct computation gives
\[
g_n(x)=(1+2\kappa_2)(x-\mu)^2+2(n+1)\tau_p'(\mu)(x-\mu)+2\theta_n\tau_p(\mu)+n^2\tau_p'(\mu)^2 .
\]
If $\kappa_2>-1/2$, $g_n$ is strictly convex with minimum at $x_n^*$, and completing the square gives the value of $g_n(x_n^*)$ in (ii), since $n^2(1+2\kappa_2)-(n+1)^2=-\theta_n$. If $\kappa_2\le-1/2$, then $\mathcal I_p$ is bounded and $g_n$ is either strictly concave or affine.

\emph{Proof of (i).} In the boundary regime, the minimum of $g_n$ on the closure of $\mathcal I_p$ is attained at a finite endpoint $e^*$, and $g_n>g_n(e^*)$ on $\mathcal I_p$ unless $g_n$ is constant. By \eqref{eq:IP-crit}, the value of $U_n$ at a local maximizer $x_0$ is at most $2/\sqrt{g_n(x_0)}<2/\sqrt{g_n(e^*)}$. The limit of $U_n$ at a finite endpoint $e$ is $2/\sqrt{g_n(e)}\le2/\sqrt{g_n(e^*)}$, and at an infinite endpoint it is $0$.  Hence $\sup_{\mathcal I_p}U_n=2/\sqrt{g_n(e^*)}=\max_e2/|\lambda_{n+1}(e)|$, since $g_n(e)=\lambda_{n+1}(e)^2$ when $\tau_p(e)=0$, and it is not attained, except when $g_n$ is constant. The last equality of the claim follows from $\tau_p'(l) =\sqrt{D_{\tau_p}} = -\tau_p'(u)$.

\emph{Proof of (ii).} The bound follows in the same way from $g_n\ge g_n(x_n^*)$. 

We now work in the interior regime of (ii), without
requiring $\theta_n>0$. Before treating symmetric targets, we show that $U_n$ has no local minimizer $y_0\in \mathcal I_p$ with $y_0>x_n^*$; the case $y_0<x_n^*$ is symmetric. Suppose there is one. On $[x_n^*,\infty)$, $g_n$ is strictly increasing, so $U_n$ is not constant near $y_0$, and there is $z>y_0$ with $U_n(z)>U_n(y_0)$. The limit of $U_n$ at $u$ is either $0$ or $2/\sqrt{g_n(u)}<2/\sqrt{g_n(y_0)}\le U_n(y_0)$. In both cases it is smaller than $U_n(z)$, so $\sup_{[y_0,u)}U_n$ is attained at a local maximizer $x_1>y_0$ with $U_n(x_1)>U_n(y_0)$. But by \eqref{eq:IP-crit},
\[
U_n(x_1)\le\frac{2}{\sqrt{g_n(x_1)}}<\frac{2}{\sqrt{g_n(y_0)}}\le U_n(y_0),
\]
a contradiction. Consequently $U_n$ is unimodal on each side of $x_n^*$.

\emph{Proof of (iii).}
Assume now that $p$ is symmetric about $\mu$. Then $\tau_p(x)p(x)=\int_l^x(\mu-t)p(t)\,dt=\int_x^u(t-\mu)p(t)\,dt$ gives $\tau_p(\mu-y)=\tau_p(\mu+y)$, hence $\tau_p'(\mu)=0$ and $x_n^*=\mu$. Moreover $\underline P_k(\mu-y)=\overline P_k(\mu+y)$, so $U_n$ is symmetric about $\mu$. Suppose that $\sup_{[\mu,u)}U_n>U_n(\mu)$, then
\begin{itemize}
\item either the supremum is attained at some $x_1\in(\mu,u)$, and $U_n(x_1)>U_n(\mu)$;
\item or it is not attained, in which case it is the limit of $U_n$ at $u$, hence positive, so that $u<\infty$, the supremum equals $2/\sqrt{g_n(u)}$, and we may take $x_1<u$ close enough to $u$ that $U_n(x_1)>U_n(\mu)$ and $\min_{[\mu,u)}U_n=\min_{[\mu,x_1]}U_n$.
\end{itemize}
Let $y_0$ be a point where $U_n$ attains its minimum on $[\mu,x_1]$. In both cases $y_0\neq x_1$, and $y_0\notin(\mu,x_1)$ by the previous paragraph, so $y_0=\mu$. By symmetry, $\mu$ is a local minimizer of $U_n$. If $g_n(\mu)\le0$, this already
contradicts \eqref{eq:IP-crit}, which forces $U_n(\mu)^2g_n(\mu)\ge4$ at a local
minimizer. Otherwise, \eqref{eq:IP-crit} gives
\[
U_n(\mu)\ge\frac{2}{\sqrt{g_n(\mu)}}>\frac{2}{\sqrt{g_n(z)}}\ge\sup_{[\mu,u)}U_n>U_n(\mu),
\]
with $z=x_1$ in the first case and $z=u$ in the second, a contradiction. Hence $\sup_{\mathcal I_p}U_n=U_n(\mu)$. Finally, by symmetry $\underline R_n(\mu)=\overline R_n(\mu)$, so that $U_n(\mu)=1/\underline R_n(\mu)$. Since $\underline P_k(\mu)=\mathbb E|Z-\mu|^{k-1}/(2\,(k-1)!)$ for $k\ge1$, this gives the stated value.
\end{proof}

\begin{lemma}\label{lem:exp-Mn-Un}
Let $p(x)$ be the exponential density with scale $\beta$, and take $w\equiv 1$. Then for every $n\ge 1$ and $j\ge 0$,
\[
M^{n,j}_{p,1}(x) = -\beta^{\,j-n+1}, \qquad\qquad U^{n,j}_{p,1}(x) = \beta^{\,j-n+1},
\]
for all $x\in(0,\infty)$; in particular both quantities are constant in $x$.
\end{lemma}

\begin{proof}
Since
\(\overline P(x)=e^{-x/\beta}\) and
\(\overline m_{p,1}(x)=\overline P(x)/p(x)=\beta\), we have
\(\overline m_{p,1}^{(i)}\equiv0\) for every \(i\geq1\).
Consequently, for \(n\geq1\), the first terms in
\eqref{eq:def-M} and \eqref{eq:general-absolute-kernel-envelope} vanish,
and
\[
M^{n,j}_{p,1}(x) = -\underline m_{p,1}^{(n)}(x)\,\overline P_{j+1}(x), \qquad U^{n,j}_{p,1}(x) = \underline m_{p,1}^{(n)}(x)\,\overline P_{j+1}(x),
\]
where we used $\underline m_{p,1}^{(n)}>0$ to drop the absolute value in the second identity. A direct computation from $\underline m_{p,1}(x)=\beta(e^{x/\beta}-1)$ gives $\underline m_{p,1}^{(n)}(x)=\beta^{1-n}e^{x/\beta}$ for $n\ge1$, while memorylessness of the exponential distribution yields $\overline P_{j+1}(x)=\beta^{j}e^{-x/\beta}$ for $j\ge0$. Substituting these two expressions into the displayed identities, the factors $e^{\pm x/\beta}$ cancel exactly, giving the stated formulas.
\end{proof}

\begin{remark}
For \(n\geq1\), taking $j=n-1,n,n+1$ recovers, respectively,
\[
U^{n,n-1}_{p,1}=1, \qquad M^{n,n-1}_{p,1}=-1,\ \ U^{n,n}_{p,1}=\beta, \qquad M^{n,n}_{p,1}=-\beta,\ \ U^{n,n+1}_{p,1}=\beta^2.
\]
Thus the diagonal correction terms of Theorem~\ref{thm:kernel-representations-of-derivatives} vanish, but the second identity in
\eqref{eq:closed-three-neighbour-condition} fails because
\(M^{n,n}_{p,1}=-\beta\neq0\). 
\end{remark}

\begin{proof}[Proof of Lemma~\ref{lem:gamma-weighted-monotonicity}]
Proposition~\ref{prop:equivalence-m-tau-P} and
Lemma~\ref{lem:integrated-pearson-tail-recursions} reduce the first
inequality to
\[
\sgn \left((-1)^n\big(\tau_p\,\overline{m}_{p,\tau_p}^{(n)}\big)'  \right) =\sgn\left( \beta\,\overline{P}_{n+1} - (n+1)\,\overline{P}_{n+2}\right),
\]
so it suffices to establish the stronger bound
\(\overline{P}_m(x)\geq\beta\,\overline{P}_{m-1}(x)\) for every \(m\geq1\).
We prove this by joint induction on \(m\), using the statements
\[
(S_m):\ x\mapsto \overline{P}_m(x)e^{x/\beta}\ \text{is non-decreasing on } (0,\infty),
\qquad
(T_m):\overline{P}_m(x) \ge \beta\,\overline{P}_{m-1}(x).
\]
For the base case \((S_0)\),
\(p(x)e^{x/\beta}=x^{\alpha-1}/(\Gamma(\alpha)\beta^\alpha)\),
which is non-decreasing because \(\alpha\geq1\). Now
\((S_{m-1}) \Rightarrow (T_m)\): by definition and \((S_{m-1})\),
\begin{align*}
\overline{P}_m(x)
&= \int_x^\infty \overline{P}_{m-1}(t)\,dt\\
&= \int_x^\infty
\big[\overline{P}_{m-1}(t)e^{t/\beta}\big]e^{-t/\beta}\,dt\\
&\geq
\overline{P}_{m-1}(x)e^{x/\beta}
\int_x^\infty e^{-t/\beta}\,dt
= \beta\overline{P}_{m-1}(x).
\end{align*}
Finally, $(T_m) \Rightarrow (S_m)$ as
\[
\big(\overline{P}_m(x)e^{ x/\beta}\big)' = \big(\overline{P}_m'(x) + \overline{P}_m(x)/\beta\big)e^{x/\beta}
= \big(\overline{P}_m(x)/\beta - \overline{P}_{m-1}(x)\big)e^{x/\beta} \ge 0
\]
by $(T_m)$. In particular $(T_{n+2})$ gives $\overline{P}_{n+2}\ge\beta\,\overline{P}_{n+1}$, hence
$(n+1)\overline{P}_{n+2}\ge\beta\,\overline{P}_{n+1}$ for every $n\ge0$, which is the
required inequality.

For the second claim, \(\underline m_{p,\tau_p}\) is absolutely monotone,
while \(\tau_p(x)=\beta x\) and \(\tau'_p(x)=\beta\) are non-negative.
Hence
\[
(\tau_p\underline m_{p,\tau_p}^{(n)})'
=\tau'_p\underline m_{p,\tau_p}^{(n)}
+\tau_p\underline m_{p,\tau_p}^{(n+1)}
\geq0.
\]
\end{proof}

\begin{proof}[Proof of Proposition~\ref{prop:all-order-weighted-factor}]
Under \eqref{eq:closed-three-neighbour-condition},
we can differentiate \eqref{eq:general-diagonal-representation} then use
\eqref{eq:master-horizontal-move} at $(i,j)=(n,n)$,
together with $M_{p,w}^{n,n-1}=-1/w$, to obtain
\[
 f_{p,w,h}^{(n+1)}(x)
 =
 \mathbb E\bigl[
 K_{p,w}^{n+1,n}(x,Z)h^{(n)}(Z)
 \bigr]
 +
 \frac{h^{(n)}(x)}{w(x)}
\]
for almost every $x\in\mathcal I_p$.
Multiplying \eqref{eq:general-diagonal-representation} by $\omega'(x)$, the last display
by $\omega(x)$, and adding gives
\[
 \bigl(\omega f_{p,w,h}^{(n)}\bigr)'(x)
 =
 \mathbb E\bigl[\mathrm{K}_{p,w,\omega}^{n}(x,Z)h^{(n)}(Z)\bigr]
 +
 \frac{\omega(x)}{w(x)}h^{(n)}(x).
\]
Taking expectations in \eqref{eq:Lambdan} gives
\[
 \mathbb E[\mathrm{K}_{p,w,\omega}^{n}(x,Z)]
 =
 \omega'(x)M_{p,w}^{n,n}(x)
 +
 \omega(x)M_{p,w}^{n+1,n}(x).
\]
Since $M_{p,w}^{n,n}=0$, the recursion
\eqref{eq:M-derivative-recursion} gives
\[
 M_{p,w}^{n+1,n}
 =
 \bigl(M_{p,w}^{n,n}\bigr)'
 +
 M_{p,w}^{n,n-1}
 =
 -\frac1w.
\]
Hence
\(
 \mathbb E[\mathrm{K}_{p,w,\omega}^{n}(x,Z)]
 =
 -{\omega(x)}/{w(x)},
\)
thus 
\[ 
\mathbb E\bigl[\mathrm{K}_{p,w,\omega}^{n}(x,Z)h^{(n)}(Z)\bigr]
 +
 \frac{\omega(x)}{w(x)}h^{(n)}(x)= \mathbb E\bigl[\mathrm{K}_{p,w,\omega}^{n}(x,Z)(h^{(n)}(Z)-h^{(n)}(x))\bigr]
\]
proving \eqref{eq:weighted-n}.
Integrating the absolute value of \eqref{eq:Lambdan}
separately over $(l,x)$ and $(x,u)$ gives
\(
 \mathbb E\bigl[
 |\mathrm{K}_{p,w,\omega}^{n}(x,Z)|
 \bigr]
 =
 \mathrm U_{p,w,\omega}^{n}(x).
\)
from which \eqref{eq:weighted-n-crude} immediately follow.

Assume now \eqref{eq:weighted-all-order-sign-condition}.
Since $\underline P_n,\overline P_n\geq0$, both branches
in \eqref{eq:Lambdan} are non-positive. Thus
$\mathrm{K}_{p,w,\omega}^{n}(x,\cdot)\leq0$ and
\[
 \mathrm U_{p,w,\omega}^{n}(x)
 =
 -\mathbb E[\mathrm{K}_{p,w,\omega}^{n}(x,Z)]
 =
 \frac{\omega(x)}{w(x)}.
\]
In particular, $\omega(x)/w(x)>0$ at almost every point
under consideration. Substituting into
\eqref{eq:weighted-n-crude}, taking the essential supremum and using
$\operatorname{osc}(h^{(n)})\leq2\|h^{(n)}\|$
proves \eqref{eq:weighted-n-osc}.
\end{proof}

\section{Proofs from Section~\ref{sec:applications}}\label{app:proof-from-applications}

\subsection{Details for the Beta approximations}
\label{app:beta-edgeworth}

We use the exchangeable pair $(W_n,W_n')$ of
\cite[Propositions~2.2--2.3]{dobler_beta_2015}, with
$\Delta_n=W_n'-W_n\in\{0,\pm1/n\}$:
\begin{equation}
\label{eq:beta-app-conditional-moments}
 \mathbb E[\Delta_n\mid W_n]=\frac{\gamma(W_n)}{n(n+a+b-1)},
 \qquad
 \mathbb E[\Delta_n^2\mid W_n]=\frac{2\{\eta(W_n)+\ell(W_n)/n\}}{n(n+a+b-1)}.
\end{equation}
On the Beta side, $(\eta p)'=\gamma p$ gives
$\mathbb E[\gamma(Z)g(Z)]=-\mathbb E[\eta(Z)g'(Z)]$, without boundary
terms. The Stein equation reads \[\eta f_{p,\eta,h}'+\gamma f_{p,\eta,h}=h-\mathbb E[h(Z)].\]

\begin{proof}[Proof of Theorem~\ref{thm:edgeworth-one-term}]
Since $h''$ is Lipschitz, $f_{p,\eta,h}''$ is Lipschitz, with
$\|f_{p,\eta,h}^{(3)}\|\leq c_{p,\eta}^{3,3}\|h^{(3)}\|$.
By the trapezoidal rule,
\[f_{p,\eta,h}(W_n')-f_{p,\eta,h}(W_n)=\Delta_n\{f_{p,\eta,h}'(W_n)+f_{p,\eta,h}'(W_n')\}/2+\rho_n\] with
$|\rho_n|\leq|\Delta_n|^3\|f_{p,\eta,h}^{(3)}\|/12$. Exchangeability gives
\[\mathbb E[\Delta_n\{f_{p,\eta,h}(W_n')-f_{p,\eta,h}(W_n)\}]=-2\mathbb E[\gamma(W_n)f_{p,\eta,h}(W_n)]/\{n(n+a+b-1)\}\] and
$\mathbb E[\Delta_n^2f_{p,\eta,h}'(W_n')]=\mathbb E[\Delta_n^2f_{p,\eta,h}'(W_n)]$, so the Stein equation
and \eqref{eq:beta-app-conditional-moments} give
\begin{align*}
 \Bigl|\mathbb E[h(W_n)]-\mathbb E[h(Z)]+\frac1n\mathbb E[\ell(W_n)f_{p,\eta,h}'(W_n)]\Bigr|
 &\leq\frac{n(n+a+b-1)}{24}\,\mathbb E[\Delta_n^4]\,\|f_{p,\eta,h}^{(3)}\| \\
 &\leq\frac{ab}{12(a+b)n^2}\|f_{p,\eta,h}^{(3)}\|,
\end{align*}
because $\Delta_n^4=\Delta_n^2/n^2$, $\mathbb E[\eta(W_n)]=(1-1/n)\mathbb E[\eta(Z)]$
and $\mathbb E[\ell(W_n)]=ab/(a+b)$. 

Apply
\eqref{eq:beta-app-smooth-approximation} to $\ell f_{p,\eta,h}'$, whose second
derivative is $(b-a)f_{p,\eta,h}''+\ell f_{p,\eta,h}^{(3)}$, and use
\eqref{eq:urn-diagonal-factor}. Finally, differentiating the Stein
equation and integrating by parts twice gives
\[\mathbb E[\gamma(Z)h'(Z)]=2\mathbb E[\ell(Z)f_{p,\eta,h}'(Z)], \]
that is
\[-\mathbb E[\ell(Z)f_{p,\eta,h}'(Z)]=\mathbb E[\mathrm e_1(Z)h'(Z)].\]
\end{proof}
\paragraph{Formulas of Example~\ref{ex:e2}.}
For smooth $h$, expanding $\mathbb E[F_{p,\eta,h}(W_n')-F_{p,\eta,h}(W_n)]=0$ with $F_{p,\eta,h}'=f_{p,\eta,h}$ to
fourth order gives
\[\mathbb E[h(W_n)]-\mathbb E[h(Z)]
=-n^{-1}\mathbb E[\ell f_{p,\eta,h}'(W_n)]
-n^{-2}\mathbb E[(\gamma f_{p,\eta,h}''/6+\eta f_{p,\eta,h}^{(3)}/12)(W_n)]+O(n^{-3}).\]
Applying Theorem~\ref{thm:edgeworth-one-term} to $\ell f_{p,\eta,h}'$
expands the first term to order $n^{-2}$. Comparing the coefficients of
$n^{-2}$ with Remark~\ref{thm:beta-edgeworth-all-orders} and integrating
by parts gives \eqref{eq:E2-C}. The kernel form of
\eqref{eq:e2-general} follows from \eqref{eq:urn-diagonal-factor} and
Fubini's theorem, and the polynomial form from the expression of
$\mathrm e_2$ in Example~\ref{ex:e2} and $(\eta p)'=\gamma p$.

 \subsection{Details for the critical Curie--Weiss model}
\label{app:CW-edgeworth}

Let
\[
 S_n=\sum_{i=1}^n\sigma_i,
 \qquad
 W_n=n^{-3/4}S_n,
 \qquad
 y_n=S_n/n=n^{-1/4}W_n.
\]
We use the standard uniform moment bounds
\begin{equation}
\label{eq:CW-uniform-moments}
 \sup_{n\geq1}\mathbb E|W_n|^r<\infty,
 \qquad r\geq0,
\end{equation}
which follow from the critical Curie--Weiss tail estimates;
see \cite{eichelsbacher_lowe_2010,Chatterjee_2011}.
Throughout, $C$ denotes a finite constant independent of
$n$ and the $1$-Lipschitz test function $h$, whose value
may change from line to line.

Choose an index $I$ uniformly and independently of the
configuration, and resample $\sigma_I$ conditionally on
the remaining spins. Then $(W_n,W_n')$ is exchangeable,
and
\[
 \mathbb E[\sigma_I'\mid\sigma,I=i]
 =
 \tanh\left(\frac{S_n-\sigma_i}{n}\right),
 \qquad
 \Delta_n:=W_n-W_n'
 =
 n^{-3/4}(\sigma_I-\sigma_I').
\]
We use the sign convention of \cite{Chatterjee_2011},
opposite to that in Appendix~\ref{app:beta-edgeworth}.

\paragraph{Constants in \eqref{eq:improv-distances-CW}}
\label{app:CW-constants}
Using the notation from \cite{Chatterjee_2011}, let $c_1=p_4(0)=\sqrt2/(3^{1/4}\Gamma(1/4))\approx0.296383$, and have the constants of their hypotheses (H2) and (H3) be
\[
 c_2=\frac3{c_1^2}+\Bigl(\frac3{c_1^2}\Bigr)^{1/3}\approx37.3963,
 \qquad
 c_3=3+\frac{3^{5/3}}{c_1^{4/3}}\approx34.5789
\]
for $g(w)=c_0^{-1}w^3/3$.

Reusing the notation of the Stein factors, introduced before Corollary~\ref{cor:integrated-pearson-envelopes}, suppose $\|f_{p_4,1,h}\|\leq c^{0,1}_{p_4,1}\|h'\|$,
$\|f_{p_4,1,h}'\|\leq c^{1,1}_{p_4,1}\|h'\|$ and $\|f_{p_4,1,h}''\|\leq c^{ 2,1}_{p_4,1}\|h'\|$, then
\eqref{eq:CS-plug-in} and \cite[Lemma~5.1]{Chatterjee_2011} give
\[
 d_{\mathrm W}(W_n,Z)
 \leq
 \Bigl(\frac{15}2c^{1,1}_{p_4,1}+15c^{0,1}_{p_4,1}\Bigr)n^{-1/2}+2c^{2,1}_{p_4,1}n^{-3/4}.
\]
The upper bounds $c^{0,1}_{{p_4},1}< c_2$, $c^{1,1}_{{p_4},1} < (1+c_2)/c_1$, $c^{2,1}_{{p_4},1}<2(1+c_2)$ of
\cite[(5.4)]{Chatterjee_2011} give the coefficients $1532.57$ and $153.59$.
Our factors $c^{0,1}_{p_4,1}=\tau_{p_4}(0)=\sqrt{3\pi}/2$,
$c^{1,1}_{p_4,1}=c^{1,1}_{p_4,1}\approx1.02325$ and
$c^{2,1}_{p_4,1}\approx2.30554$ give $30.70$ and $4.62$.

To deal with the Kolmogorov distance, from the proof of \cite[Theorem~1.2]{Chatterjee_2011}, see their
(5.9)--(5.19), together with \cite[Lemma~5.1]{Chatterjee_2011}, we have that if uniformly in $z$,
$\|f_{p_4,1h_z}\|\leq a$, $\|\rho_{p_4}f_{p_4,1h_z}\|\leq b$ and $\|f_{p_4,1h_z}'\|\leq d$,
then
\[
 d_{\mathrm K}(W_n,Z)
 \leq
 \Bigl(\frac{15}2(1+b)+15a\Bigr)n^{-1/2}
 +\Bigl(2c_1+4\Bigl[5d+\frac{a c_1c_3}{2}(c_1+5)\Bigr]\Bigr)n^{-3/4}.
\]
Their bounds $a<2/c_1$, $b<2$, $d<4$ of \cite[(5.10)]{Chatterjee_2011}
give $123.73$ and $813.17$; ours, $a=c^{0,0}_{p_4,1}=1/(4c_1)$, $b=1$ and $d=c^{1,0}_{p_4,1}=1$ from
\eqref{eq:subbotin-kolmogorov-Stein-factor} and
\eqref{eq:subbotin-score-weighted-kolmogorov-factor}, give $27.66$ and
$112.17$. The derivation assumes $c_1c_3\,2n^{-3/4}\leq1$, that is,
$n\geq57$; for $n\leq56$ both right-hand sides exceed $1$, so the bounds
hold for every $n$.

\paragraph{Edgeworth correction}
The proof of Theorem~\ref{thm:critical-CW-edgeworth} pushes the
exchangeable-pair argument of \cite{Chatterjee_2011} one order further.
Lemma~\ref{lem:CW-conditional-moment-expansions} refines
\cite[Lemma~5.1]{Chatterjee_2011} by keeping the terms of order
$n^{-1/2}$ that their proof bounds in absolute value, and
Lemma~\ref{lem:CW-correction-transfer} transfers the resulting
correction functional from $W_n$ to $Z_4$.

\begin{lemma}[Conditional moment expansions]
\label{lem:CW-conditional-moment-expansions}
We have
\[
 1-\frac{n^{3/2}}2\mathbb E[\Delta_n^2\mid W_n]
 =
 n^{-1/2}W_n^2+R_{2,n},
\]
and
\[
 n^{3/2}\mathbb E[\Delta_n\mid W_n]
 =
 \frac{W_n^3}{3}
 +
 n^{-1/2}\left(W_n-\frac2{15}W_n^5\right)
 +
 R_{1,n},
\]
where
\[
 \mathbb E|R_{1,n}|+\mathbb E|R_{2,n}|
 \leq Cn^{-1}.
\]
\end{lemma}

\begin{proof}
Write $T=\tanh$ and $y=y_n$, and note that conditioning on the
configuration or on $W_n$ is the same here, since all the quantities
below depend on the configuration only through $S_n$.
By \cite[(5.26)]{Chatterjee_2011} and the computation of
$\mathbb E[\Delta_n^2\mid W_n]$ in the proof of \cite[Lemma~5.1]{Chatterjee_2011},
\[
 \mathbb E[\Delta_n\mid W_n]
 =
 n^{-3/4}\Bigl\{y-\frac1n\sum_{i=1}^nT(y-\sigma_i/n)\Bigr\},
 \qquad
 1-\frac{n^{3/2}}2\mathbb E[\Delta_n^2\mid W_n]
 =
 \frac1n\sum_{i=1}^n\sigma_iT(y-\sigma_i/n).
\]
The inequality
\(|T(x)-T(y)-(x-y)T'(y)|\leq2\cdot3^{-3/2}(x-y)^2\) of the same proof
gives, uniformly over the configuration,
\[
 \frac1n\sum_{i=1}^nT(y-\sigma_i/n)
 =
 T(y)-\frac ynT'(y)+\varepsilon_n,
 \qquad
 \frac1n\sum_{i=1}^n\sigma_iT(y-\sigma_i/n)
 =
 yT(y)-\frac1nT'(y)+\varepsilon_n',
\]
with \(|\varepsilon_n|,|\varepsilon_n'|\leq2\cdot3^{-3/2}n^{-2}\).
We now carry the expansion of $T$ one order further than
\cite{Chatterjee_2011}. Since $T^{(7)}$ is a polynomial in $T$ whose
maximum modulus on $[-1,1]$ is $|T^{(7)}(0)|=272$, Taylor's formula gives
\[
 \Bigl|T(y)-y+\frac{y^3}{3}-\frac{2y^5}{15}\Bigr|
 \leq\frac{17}{315}|y|^7,
 \qquad y\in\mathbb R,
\]
while $yT(y)\in[y^2-y^4/3,\,y^2]$ and $1-y^2\leq T'(y)\leq1$,
as shown in the proof of \cite[Lemma~5.1]{Chatterjee_2011}.
Substituting $y=n^{-1/4}W_n$, we obtain the stated expansions with
\[
 |R_{1,n}|
 \leq
 n^{-1}\Bigl(\frac{17}{315}|W_n|^7+|W_n|^3\Bigr)
 +2\cdot3^{-3/2}n^{-5/4},
 \qquad
 |R_{2,n}|
 \leq
 n^{-1}\Bigl(1+\frac{W_n^4}{3}\Bigr)+2\cdot3^{-3/2}n^{-2}.
\]
The claim follows from \eqref{eq:CW-uniform-moments}. (The explicit
bounds $\mathbb E|W_n|^3\leq15$ and $\mathbb E W_n^6\leq224.4$ of
\cite[(5.24), (5.33)]{Chatterjee_2011}, together with $|W_n|\leq n^{1/4}$,
already give $\mathbb E|R_{1,n}|+\mathbb E|R_{2,n}|\leq Cn^{-3/4}$, which
suffices for Theorem~\ref{thm:critical-CW-edgeworth}.)
\end{proof}
\begin{lemma}[Transfer of the correction functional]
\label{lem:CW-correction-transfer}
For $f_h=f_{p_4,1,h}$, set
\[
 F_h(x)
 =
 x^2f_h'(x)
 +
 \left(x-\frac2{15}x^5\right)f_h(x),
 \qquad
 L_n(h)=\mathbb E[F_h(W_n)],
 \qquad
 L(h)=\mathbb E[F_h(Z_4)].
\]
Uniformly over $1$-Lipschitz functions $h$,
\[
 L_n(h)=L(h)+O(n^{-1/4}).
\]
\end{lemma}

\begin{proof}
The quartic factors from
\eqref{eq:subbotin-solution-envelope}--
\eqref{eq:subbotin-second-derivative-envelope} give
\[
 \|f_h\|+\|f_h'\|+\|f_h''\|\leq C.
\]
Thus $F_h$ is locally absolutely continuous and
\begin{equation}
\label{eq:B.18}
 |F_h(x)|+|F_h'(x)|
 \leq C(1+|x|^5)
 \qquad\text{a.e.}
\end{equation}

Let $\chi_M$ be a smooth cut-off equal to $1$ on $[-M,M]$,
vanishing outside $[-2M,2M]$, with $0\leq\chi_M\leq1$
and $\|\chi_M'\|\leq C/M$, where $M\geq1$.
Then $\chi_MF_h$ is Lipschitz with constant at most $CM^5$.
The bound \eqref{eq:improv-distances-CW}, whose derivation in
Appendix~\ref{app:CW-constants} only uses
\cite[Lemma~5.1]{Chatterjee_2011} and the quartic Stein factors,
gives $d_{\mathrm W}(W_n,Z_4)\leq36n^{-1/2}$, hence
\begin{equation}
\label{eq:B.19}
 \left|
 \mathbb E[(\chi_MF_h)(W_n)]
 -
 \mathbb E[(\chi_MF_h)(Z_4)]
 \right|
 \leq CM^5n^{-1/2}.
\end{equation}
By \eqref{eq:CW-uniform-moments} and the moments of $Z_4$,
for every fixed $r>0$,
\begin{equation}
\label{eq:B.20}
 \mathbb E[|F_h(W_n)|\mathbf1_{\{|W_n|>M\}}]
 +
 \mathbb E[|F_h(Z_4)|\mathbf1_{\{|Z_4|>M\}}]
 \leq C_rM^{-r}.
\end{equation}
Combining the two bounds and taking $M=n^{1/40}$ and
$r=10$ yields
\[
 |L_n(h)-L(h)|
 \leq Cn^{-3/8}+C_{10}n^{-1/4}
 \leq Cn^{-1/4}.
\]
\end{proof}

\begin{proof}[Proof of Theorem~\ref{thm:critical-CW-edgeworth}]
Let $h$ be $1$-Lipschitz and write $f_h=f_{p_4,1,h}$.
The quartic factors give
\[
 \|f_h\|+\|f_h'\|+\|f_h''\|\leq C,
\]
where $f_h'$ is Lipschitz and $f_h''$ is understood almost
everywhere. Taylor's formula therefore gives
\[
 f_h(W_n')
 =
 f_h(W_n)-\Delta_nf_h'(W_n)+r_n,
 \qquad
 |r_n|\leq C\Delta_n^2.
\]
By exchangeability,
$\mathbb E[\Delta_n\{f_h(W_n)+f_h(W_n')\}]=0$, so
\[
 2\mathbb E[\Delta_nf_h(W_n)]
 -
 \mathbb E[\Delta_n^2f_h'(W_n)]
 +
 \mathbb E[\Delta_nr_n]
 =
 0.
\]
Conditioning on $W_n$ and using
Lemma~\ref{lem:CW-conditional-moment-expansions}, we obtain
\begin{align*}
 \mathbb E\left[
 f_h'(W_n)-\frac{W_n^3}{3}f_h(W_n)
 \right]
 &=
 n^{-1/2}L_n(h)\\
 &\quad+
 \mathbb E[R_{1,n}f_h(W_n)+R_{2,n}f_h'(W_n)]\\
 &\quad+
 \frac{n^{3/2}}2\mathbb E[\Delta_nr_n].
\end{align*}
By Lemma~\ref{lem:CW-conditional-moment-expansions} and since
$|\Delta_n|\leq2n^{-3/4}$, the last two terms
are $O(n^{-3/4})$, uniformly in $h$. The Stein equation
therefore gives
\begin{equation}
\label{eq:4.6}
 \mathbb E[h(W_n)]-\mathbb E[h(Z_4)]
 =
 n^{-1/2}L_n(h)+O(n^{-3/4}).
\end{equation}
By Lemma~\ref{lem:CW-correction-transfer},
\begin{equation}
\label{eq:4.7}
 \mathbb E[h(W_n)]-\mathbb E[h(Z_4)]
 =
 n^{-1/2}L(h)+O(n^{-3/4}).
\end{equation}

It remains to identify $L(h)$. The Stein equation gives
\[
 L(h)
 =
 \mathbb E\left[
 Z_4^2\{h(Z_4)-\mathbb E[h(Z_4)]\}
 +
 \left(Z_4+\frac15Z_4^5\right)f_h(Z_4)
 \right].
\]
Set $A(x)=x^2/2+x^6/30$. Since $f_h$ is bounded and
locally absolutely continuous, and $p_4$ has quartic
tails, $Af_h\in\mathcal F_{p_4,1}$. Applying the Stein
identity to $Af_h$ gives
\[
 \mathbb E[A'(Z_4)f_h(Z_4)]
 =
 -\mathbb E[A(Z_4)\{h(Z_4)-\mathbb E[h(Z_4)]\}].
\]
Hence
\[
\begin{split}
 L(h)
 &=
 \mathbb E\left[
 \left(\frac{Z_4^2}{2}-\frac{Z_4^6}{30}\right)
 \{h(Z_4)-\mathbb E[h(Z_4)]\}
 \right]\\
 &=
 \mathbb E[e_1(Z_4)h(Z_4)],
\end{split}
\]
because
$\mathbb E[Z_4^2/2-Z_4^6/30]=\mathbb E[Z_4^2]/5$.
Substitution into \eqref{eq:4.7} proves the claim.
\end{proof}\end{document}